\documentclass[10pt]{article}

\makeatletter 
\let\c@table\c@figure
\let\c@lstlisting\c@figure
\makeatother

\DeclareFontFamily{U}{matha}{\hyphenchar\font45}
\DeclareFontShape{U}{matha}{m}{n}{
      <5> <6> <7> <8> <9> <10> gen * matha
      <10.95> matha10 <12> <14.4> <17.28> <20.74> <24.88> matha12
      }{}
\DeclareSymbolFont{matha}{U}{matha}{m}{n}

\DeclareMathSymbol{\vvoplus}{\mathbin}{matha}{"60}
\DeclareMathSymbol{\vvominus}{\mathbin}{matha}{"61}
\DeclareMathSymbol{\vvotimes}{\mathbin}{matha}{"62}
\DeclareMathSymbol{\vvodot}{\mathbin}{matha}{"64}

\DeclareFontFamily{U}{mathb}{\hyphenchar\font45}
\DeclareFontShape{U}{mathb}{m}{n}{
      <5> <6> <7> <8> <9> <10> gen * mathb
      <10.95> mathb10 <12> <14.4> <17.28> <20.74> <24.88> mathb12
      }{}
\DeclareSymbolFont{mathb}{U}{mathb}{m}{n}

\DeclareMathSymbol{\mesh}{\mathbin}{mathb}{'005}
\DeclareMathSymbol{\meshc}{\mathbin}{mathb}{'015}

\DeclareFontFamily{U}{mathx}{\hyphenchar\font45}
\DeclareFontShape{U}{mathx}{m}{n}{
      <5> <6> <7> <8> <9> <10> gen * mathx
      <10.95> mathx10 <12> <14.4> <17.28> <20.74> <24.88> mathx12
      }{}
\DeclareSymbolFont{mathx}{U}{mathx}{m}{n}

\DeclareMathSymbol{\vvbigoplus}{\mathop}{mathx}{"C0}
\DeclareMathSymbol{\vvbigominus}{\mathop}{mathx}{"C1}
\DeclareMathSymbol{\vvbigotimes}{\mathop}{mathx}{"C2}
\DeclareMathSymbol{\vvbigodot}{\mathop}{mathx}{"C4}

\usepackage{amsmath}
\usepackage{amssymb}
\usepackage{math tools} 

\renewcommand{\oplus}{\vvoplus}
\renewcommand{\ominus}{\vvominus}

\renewcommand{\bigoplus}{\vvbigoplus}
\newcommand{\bigominus}{\vvbigominus}

\newcommand{\tosurj}{\mathrel{\mathrlap{\rightarrow}\mkern-1mu\rightarrow}}

\usepackage{lmodern}
\usepackage[T1]{fontenc}
\usepackage{amsmath}
\usepackage{amsthm}
\usepackage{url}
\usepackage{latexsym}
\usepackage[small, pagestyles]{titlesec}
\usepackage{units} 
\usepackage[small,it]{caption}

\usepackage{mathtools} 

\usepackage{multirow}

\usepackage[square,comma,numbers,sort&compress]{natbib}

\usepackage{tikz}
\usetikzlibrary{arrows, automata, decorations.pathreplacing, fit, matrix, patterns, positioning}
\usetikzlibrary{hobby} 
\usepackage{tikz-qtree}
\usepackage{forest}
\usepackage{xifthen} 

\usepackage{footmisc} 

\usepackage{color}
\definecolor{lightgray}{rgb}{0.8, 0.8, 0.8}
\definecolor{darkgray}{rgb}{0.7, 0.7, 0.7}

\usepackage[bookmarks]{hyperref}
\hypersetup{
	colorlinks=true,
	linkcolor=black,
	anchorcolor=black,
	citecolor=black,
	urlcolor=black,
	pdfpagemode=UseThumbs,
	pdftitle={Labelled well-quasi-order for permutation classes},
	pdfsubject={Combinatorics},
	pdfauthor={Brignall and Vatter},
}

\newcounter{todocounter}

\theoremstyle{plain}
\newtheorem{theorem}{Theorem}[section]
\newtheorem{proposition}[theorem]{Proposition}
\newtheorem{lemma}[theorem]{Lemma}
\newtheorem{corollary}[theorem]{Corollary}

\newtheorem{conjecture}[theorem]{Conjecture}
\newtheorem{question}[theorem]{Question}

\theoremstyle{definition}

\theoremstyle{plain} 

\makeatletter
\newfont{\footsc}{cmcsc10 at 8truept}
\newfont{\footbf}{cmbx10 at 8truept}
\newfont{\footrm}{cmr10 at 10truept}
\renewenvironment{abstract}%
                {
                  \begin{list}{}%
                     {\setlength{\rightmargin}{1in}%
                      \setlength{\leftmargin}{1in}}%
                   \item[]\ignorespaces\begin{small}}%
                 {\end{small}\unskip\end{list}}

\newcommand{\Av}{\operatorname{Av}}

\newcommand{\C}{\mathcal{C}}
\newcommand{\D}{\mathcal{D}}

\newcommand{\G}{\mathcal{G}}

\newcommand{\M}{\mathcal{M}}

\renewcommand{\S}{\mathcal{S}}
\newcommand{\U}{\mathcal{U}}
\renewcommand{\O}{\mathcal{O}}

\newcommand{\X}{\mathcal{X}}

\newcommand{\fnmatrix}[1]{\Big(\raisebox{0.5pt}{\scalebox{0.75}{\text{$\begin{matrix*}[r]#1\end{matrix*}$}}}\Big)}
\newcommand{\fnmatrixc}[1]{\Big(\raisebox{0.5pt}{\scalebox{0.75}{\text{$\begin{matrix*}[c]#1\end{matrix*}$}}}\Big)}

\newcommand{\Grid}{\operatorname{Grid}}
\newcommand{\Geom}{\operatorname{Geom}}
\newcommand{\zpm}{0/\mathord{\pm} 1}

\newcommand{\Xfig}{\begin{tikzpicture}[scale=1, anchor=base]
	\pgftransformxscale{0.112};
	\pgftransformyscale{0.112};
	\draw [semithick, line cap=round] (0,2)--(2,0) (0,0)--(2,2);
\end{tikzpicture}}
\newcommand{\Vfig}{\begin{tikzpicture}[scale=1, anchor=base]
	\pgftransformxscale{0.112};
	\pgftransformyscale{0.112};
	\draw [semithick, line cap=round] (0,2)--(1,0) (1,0)--(2,2);
\end{tikzpicture}}
\newcommand{\Vfigt}{\begin{tikzpicture}[scale=1, anchor=base]
	\pgftransformxscale{0.112};
	\pgftransformyscale{0.112};
	\draw [semithick, line cap=round] (0,0)--(1,2) (1,2)--(2,0);
\end{tikzpicture}}
\newcommand{\Vfigl}{\begin{tikzpicture}[scale=1, anchor=base]
	\pgftransformxscale{0.112};
	\pgftransformyscale{0.112};
	\draw [semithick, line cap=round] (2,0)--(0,1) (0,1)--(2,2);
\end{tikzpicture}}
\newcommand{\Vfigr}{\begin{tikzpicture}[scale=1, anchor=base]
	\pgftransformxscale{0.112};
	\pgftransformyscale{0.112};
	\draw [semithick, line cap=round] (0,0)--(2,1) (2,1)--(0,2);
\end{tikzpicture}}

\newcommand{\gridded}{\sharp}

\newcommand{\emptyperm}{\varepsilon}

\newcommand{\suplessthan}{\text{\begin{tiny}\ensuremath{<}\end{tiny}}}
\newcommand{\suplessthaneq}{\text{\begin{tiny}\ensuremath{\le}\end{tiny}}}

\newcommand\mybullet{\raisebox{-5pt}{\normalsize \ensuremath{\bullet}}}
\newcommand\mycirc{\raisebox{-5pt}{\normalsize \ensuremath{\circ}}}

\makeatletter
\def\absdot{\@ifnextchar[{\@absdotlabel}{\@absdotnolabel}}
	\def\@absdotlabel[#1]#2{%
		\node at #2 {\normalsize \mybullet};
		\node at #2 [below=2pt] {\ensuremath{#1}};
	}
	\def\@absdotnolabel#1{%
		\node at #1 {\normalsize \mybullet};
	}
\def\absdothollow{\@ifnextchar[{\@absdothollowlabel}{\@absdothollownolabel}}
	\def\@absdothollowlabel[#1]#2{%
		\node at #2 {\normalsize \textcolor{white}{\mybullet}};
		\node at #2 {\normalsize \mycirc};
		\node at #2 [below=2pt] {\ensuremath{#1}};
	}
	\def\@absdothollownolabel#1{%
		\node at #1 {\normalsize \textcolor{white}{\mybullet}};
		\node at #1 {\normalsize \mycirc};
	}
\makeatother

\newcommand{\plotperm}[1]{
	\foreach \j [count=\i] in {#1} {
		\absdot{(\i,\j)};
	};
}

\newcommand{\plotpartialperm}[1]{
	\foreach \i/\j in {#1} {
		\absdot{(\i,\j)};
	};
}

\newcommand\myencircle{\raisebox{-8.75pt}{\LARGE \ensuremath{\circ}}}

\newcommand\myencirclewhite{\raisebox{-8.75pt}{\LARGE \textcolor{white}{\ensuremath{\bullet}}}}
\newcommand{\plotpartialpermencirclewhite}[1]{
	\foreach \i/\j in {#1} {
		\node at (\i,\j) {\myencirclewhite};
		\node at (\i,\j) {\myencircle};
		\absdot{(\i,\j)};
	};
}

\newcommand{\plotpartialpermhollow}[1]{
	\foreach \i/\j in {#1} {
		\absdothollow{(\i,\j)};
	};
}

\newcommand{\plotpermbox}[4]{
	\draw [darkgray, very thick, rounded corners=0.01, line cap=round]
		({#1-0.5}, {#2-0.5}) rectangle ({#3+0.5}, {#4+0.5});
}

\newcommand{\plotpermgraph}[1]{
	\foreach \j [count=\i] in {#1} {
		\foreach \b [count=\a] in {#1} {
			\ifthenelse{\a<\i \AND \b>\j}{\draw (\a,\b)--(\i,\j);}{}
		};
	};
	\plotperm{#1};
}

\newcommand{\plotpermdyckpath}[1]{
	\draw [ultra thick, rounded corners=0.01, line cap=round] (0.5,0.5)
	\foreach \step in {#1} {
		\ifnum\step=1
			-- ++(0,1)
		\else
			-- ++(1,0)
		\fi
	};
}

\newcommand{\plotrotheperm}[1]{
	\foreach \j [count=\i] in {#1} {
		\absdot{(\j, -\i)};
	};
  \pgfmathsetmacro\len{dim(#1)}; 
	\foreach \x in {0,1,...,\len} {
		\draw[darkgray, thin, line cap=round] (0.5,{-\x-0.5})--({\len+0.5},{-\x-0.5});
		\draw[darkgray, thin, line cap=round] ({\x+0.5},-0.5)--({\x+0.5},{-\len-0.5});
	};
}

\newcommand{\plotrotheperminversions}[1]{
	\foreach \pii [count=\i] in {#1} {
		\foreach \pij [count=\j] in {#1} {
			\ifthenelse{\i<\j \AND \pii>\pij}{\absdothollow{(\pij, -\i)}}{}
		};
	};
	\plotrotheperm{#1}
}

\newcommand{\matrixpermwithzeros}[2]{
	\foreach \y [count=\x] in {#2} {
		\foreach \j in {1, 2, ..., #1} {
			\ifthenelse{\j=\y}{
				\node at (\x,\j) {\tiny $\mathbf{1}$};
			}{
				\node at (\x,\j) {\textcolor{darkgray}{\tiny $0$}};
			}
		}
	}
}

\newcommand{\plotdyckpath}[1]{
	\draw[ultra thick, line cap=round] (0.5,0)
	\foreach \step in {#1} {
		\ifnum\step=1
			-- ++(1,1)
		\else
			-- ++(1,-1)
		\fi
	};
}

\newcommand{\arcskinnyplain}[2]{
	\draw (#1,0) arc (180:0:{(#2-#1)/2});
}

\newcommand{\matchsmall}[1]{
	\begin{tikzpicture}[scale=.1, anchor=base]
		\def\h{0};
		\def\maxh{0};
		\foreach \i/\j in {#1} {
			\pgfmathparse{\j-\i};
			\let\h\pgfmathresult;
			\pgfmathifthenelse{\h>\maxh}{\h}{\maxh};
			\global\let\maxh\pgfmathresult;
		};
		\pgftransformyscale{{4.5/\maxh}};
		\foreach \i/\j in {#1} {
			\arcskinnyplain{\i}{\j};
		};
	\end{tikzpicture}
}

\newcommand{\matchpermsmall}[1]{
	\begin{tikzpicture}[scale=.1, anchor=base]
    \pgfmathsetmacro\len{dim(#1)}; 
		\def\h{0};
		\def\maxh{0};
		\foreach \j [count=\i] in {#1} {
			\pgfmathparse{2*\len+1-\j-\i};
			\let\h\pgfmathresult;
			\pgfmathifthenelse{\h>\maxh}{\h}{\maxh};
			\global\let\maxh\pgfmathresult;
		};
		\pgftransformyscale{{4.5/\maxh}};
		\foreach \j [count=\i] in {#1} {
			\arcskinnyplain{\i}{{2*\len+1-\j}};
		};
	\end{tikzpicture}
}

\newcommand{\inlinecycle}[1]{
	\begin{tikzpicture}[scale=0.17, anchor=base]
	\useasboundingbox (-1.5,-0.5) rectangle (1.5,1.5);
	\draw (0,0) circle (1);
	\foreach \i in {1,2,...,#1} {
		\draw ({90-(\i-1)*360/#1}: 1) node {\raisebox{-2.5pt}{\normalsize \ensuremath{\bullet}}};
	}
	\end{tikzpicture}
}

\newcommand{\plotpinsequence}[1]{
	\absdot{(0,0)}{};
	\edef\n{0}
	\edef\s{0}
	\edef\e{0}
	\edef\w{0}
	\edef\x{0}
	\edef\y{0}
	\foreach \pin [remember=\pin as \oldpin (initially 1), count=\i] in {#1} {
		\ifthenelse{\pin=1 \OR \pin=2}{
			\ifthenelse{\oldpin=3}{
				\xdef\x{\number\numexpr\e-1}
			}{
				\xdef\x{\number\numexpr\w+1}
			}
			\ifnum\i=1 
				\pgfmathparse{\e+1}
 				\xdef\e{\pgfmathresult}
			\fi	
		}{ 
			\ifthenelse{\oldpin=1}{
				\xdef\y{\number\numexpr\n-1}
			}{
				\xdef\y{\number\numexpr\s+1}
			}
			\ifnum\i=1 
				\pgfmathparse{\s-1}
 				\xdef\s{\pgfmathresult}
			\fi	
		}
		\ifnum\pin=1 
			\pgfmathparse{\n+2}
 			\xdef\n{\pgfmathresult}		
			\absdot{(\x,\n)}{};
			\ifnum\i>1
				\draw (\x,\n) -- (\x,\y-0.5);
			\else
			\fi
		\fi
		\ifnum\pin=2 
			\pgfmathparse{\s-2}
 			\xdef\s{\pgfmathresult}
			\absdot{(\x,\s)}{};
			\ifnum\i>1
				\draw (\x,\s) -- (\x,\y+0.5);
			\else
			\fi
		\fi
		\ifnum\pin=3 
			\pgfmathparse{\e+2}
 			\xdef\e{\pgfmathresult}
			\absdot{(\e,\y)}{};
			\ifnum\i>1
				\draw (\e,\y) -- (\x-0.5,\y);
			\else
			\fi
		\fi
		\ifnum\pin=4 
			\pgfmathparse{\w-2}
 			\xdef\w{\pgfmathresult}
			\absdot{(\w,\y)}{};
			\ifnum\i>1
				\draw (\w,\y) -- (\x+0.5,\y);
			\else
			\fi
		\fi		
	};
}

\newcommand{\gridsmallhoriz}[1]{
	\begin{tikzpicture}[scale=1, anchor=base]
    \def\gridheight{1};
    \pgfmathsetmacro\gridwidth{dim(#1)}; 
	  \pgftransformxscale{{0.225/\gridwidth}};
		\pgftransformyscale{{0.225/\gridheight}};
    \foreach \dir [count=\i] in {#1} {
      \ifthenelse{\dir>0}{
		    \draw [semithick, line cap=round] ({\i-1}, 0)--(\i, 1);
      }{
        \draw [semithick, line cap=round] ({\i-1}, 1)--(\i, 0);
      };
    };
  \end{tikzpicture}
}

\newcommand{\gridsmallhorizfn}[1]{
	\begin{tikzpicture}[scale=1, anchor=base]
    \def\gridheight{1};
    \pgfmathsetmacro\gridwidth{dim(#1)}; 
	  \pgftransformxscale{{0.155/\gridwidth}};
		\pgftransformyscale{{0.155/\gridheight}};
    \foreach \dir [count=\i] in {#1} {
      \ifthenelse{\dir>0}{
		    \draw [semithick, line cap=round] ({\i-1}, 0)--(\i, 1);
      }{
        \draw [semithick, line cap=round] ({\i-1}, 1)--(\i, 0);
      };
    };
  \end{tikzpicture}
}

\newcommand{\gridhoriz}[1]{
	\begin{tikzpicture}[scale=1, anchor=base]
	  \pgftransformxscale{0.225};
		\pgftransformyscale{0.225};
    \foreach \dir [count=\i] in {#1} {
      \ifthenelse{\dir>0}{
		    \draw [semithick, line cap=round] ({\i-1}, 0)--(\i, 1);
      }{
        \draw [semithick, line cap=round] ({\i-1}, 1)--(\i, 0);
      };
    };
  \end{tikzpicture}
}

\newcommand{\gridsmallvert}[1]{
	\begin{tikzpicture}[scale=1, anchor=base]
    \pgfmathsetmacro\gridheight{dim(#1)}; 
	  \pgftransformxscale{0.225};
		\pgftransformyscale{0.225/2};
    \foreach \dir [count=\i] in {#1} {
      \ifthenelse{\dir>0}{
		    \draw [semithick, line cap=round] (0, {\i-1})--(1, \i);
      }{
        \draw [semithick, line cap=round] (0, \i)--(1, {\i-1});
      };
    };
  \end{tikzpicture}
}

\newcommand{\gridsmalltwobyvert}[2]{
	\begin{tikzpicture}[scale=1, anchor=base]
    \def\gridwidth{2};
    \pgfmathsetmacro\gridheight{dim(#1)}; 
	  \pgftransformxscale{{0.225/\gridwidth}};
		\pgftransformyscale{{0.225/\gridheight}};
    \foreach \dir [count=\i] in {#1} {
			\ifthenelse{\dir=1 \OR \dir=-1}{
				\ifthenelse{\dir>0}{
					\draw [semithick, line cap=round] (0, {\i-1})--(1, \i);
				}{
					\draw [semithick, line cap=round] (0, \i)--(1, {\i-1});
				};
			}{};
    };
    \foreach \dir [count=\i] in {#2} {
			\ifthenelse{\dir=1 \OR \dir=-1}{
				\ifthenelse{\dir>0}{
					\draw [semithick, line cap=round] (1, {\i-1})--(2, \i);
				}{
					\draw [semithick, line cap=round] (1, \i)--(2, {\i-1});
				};
			}{};
    };
  \end{tikzpicture}
}

\newcommand{\gridverysmallvert}[1]{
	\begin{tikzpicture}[scale=1, anchor=base]
    \def\gridwidth{1};
    \pgfmathsetmacro\gridheight{dim(#1)}; 
	  \pgftransformxscale{{0.175/\gridwidth}};
		\pgftransformyscale{{0.175/\gridheight}};
    \foreach \dir [count=\i] in {#1} {
      \ifthenelse{\dir>0}{
		    \draw [semithick, line cap=round] (0, {\i-1})--(1, \i);
      }{
        \draw [semithick, line cap=round] (0, \i)--(1, {\i-1});
      };
    };
  \end{tikzpicture}
}

\newpagestyle{main}[\small]{
        \headrule
        \sethead[\usepage][][]
        {\sc LWQO for Permutation Classes}{}{\usepage}}

\title{\sc Labelled Well-Quasi-Order for Permutation Classes}

\author{\centering
	\begin{tabular}{ccc}
	Robert Brignall
	&\rule{0pt}{0pt}&
	Vincent Vatter%
	\footnote{Vatter's research was partially supported by the Simons Foundation via award number 636113.}
	\\[-0.25ex]
	\small School of Mathematics and Statistics
	&&
	\small Department of Mathematics\\[-0.5ex]
	\small The Open University
	&&
	\small University of Florida\\[-0.5ex]
	\small Milton Keynes, England UK
	&&
	\small Gainesville, Florida USA\\[-1.5ex]
	\end{tabular}
	\vspace{0.3in}
}

\titleformat{\section}
        {\large\sc}
        {\thesection.}{1em}{}   

\begin{document}
\maketitle

\pagestyle{main}

\begin{abstract}
While the theory of labelled well-quasi-order has received significant attention in the graph setting, it has not yet been considered in the context of permutation patterns.
We initiate this study here, and show how labelled well quasi order provides a lens through which to view and extend previous well-quasi-order results in the permutation patterns literature.
Connections to the graph setting are emphasised throughout. In particular, we establish that a permutation class is labelled well-quasi-ordered if and only if its corresponding graph class is also labelled well-quasi-ordered.
\end{abstract}

\maketitle

\section{Introduction}
\label{sec-lwqo-intro}

A prominent theme of the past 85 years%
\footnote{Our figure of 85 years dates the study of well-quasi-order to Wagner~\cite{wagner:uber-eine-eigen:}.}
of combinatorics research has been the study of well-quasi-order (although as Kruskal laments in \cite{kruskal:the-theory-of-w:}, the property goes by a mishmash of names). Suppose we have a universe of finite combinatorial objects and a notion of embedding one object into another that is at least reflexive and transitive, that is, the notion of embedding forms a \emph{quasi-order} (in this paper the order is also generally antisymmetric, so it in fact forms a \emph{partial order}%
\footnote{A well-quasi-ordered partial order is sometimes called a \emph{partially-well-ordered} or \emph{well-partially-ordered} set, or it is simply called a \emph{partial well order}. In particular, these terms are used in some of the early work on well-quasi-order in the permutation patterns context. We tend to agree with Kruskal's sentiment from~\cite[p.~298]{kruskal:the-theory-of-w:}, where he wrote that ``at the casual level it is easier to work with [partial orders] than [quasi-orders], but in advanced work the reverse is true.''}).
Assuming that this notion of embedding does not permit infinite strictly descending chains (as is usually the case for orders on finite combinatorial objects), it is \emph{well-quasi-ordered} (abbreviated \emph{wqo}, and written \emph{belordonn\'e} in French) if it does not contain an infinite antichain---that is, there is no infinite subset of pairwise incomparable objects. (This is but one of several ways to define wqo; others are presented in Section~\ref{subsec-wqo}.)

Three of the most celebrated results in combinatorics---Higman's lemma~\cite{higman:ordering-by-div:}, Kruskal's tree theorem~\cite{kruskal:well-quasi-orde:}, and Robertson and Seymour's graph minor theorem~\cite{robertson:graph-minors-i-xx:}---establish that certain notions of embedding constitute well-quasi-orders. For further background on well-quasi-order in general we refer the reader to the excellent panoramas provided by the recent surveys of Cherlin~\cite{cherlin:forbidden-subst:} and Huczynska and Ru\v{s}kuc~\cite{huczynska:well-quasi-orde:}. It should be noted that well-quasi-order also has significant applications to algorithmic questions, in particular questions about fixed-parameter tractability, for which the reader is referred to the book of Downey and Fellows~\cite[Part IV]{downey:fundamentals-of:}.

While the graph minor theorem establishes that the set of (finite) graphs is wqo under the minor order, it is clearly not wqo under the induced subgraph order. For example, the set of chordless cycles $\inlinecycle{3}$, $\inlinecycle{4}$, $\inlinecycle{5}$, $\dots$ forms an infinite antichain, as does the set of \emph{double-ended forks}%
\footnote{The graphs we call double-ended forks are also called $H$-graphs and split-end paths in some works.},
examples of which are shown in Figure~\ref{fig-double-ended-forks}. Another order that is not wqo is the containment order on permutations, as described shortly.

\begin{figure}
\begin{center}
	\begin{tikzpicture}[scale=0.5]
		\plotpartialperm{-1/-0.5,-1/0.5,0/0,1/0,2/0,3/0,4/-0.5,4/0.5};
		\draw (-1,-0.5)--(0,0)--(-1,0.5);
		\draw (4,-0.5)--(3,0)--(4,0.5);
		\draw (0,0)--(3,0);
	\end{tikzpicture}
	\quad\quad
	\begin{tikzpicture}[scale=0.5]
		\plotpartialperm{-1/-0.5,-1/0.5,0/0,1/0,2/0,3/0,4/0,5/-0.5,5/0.5};
		\draw (-1,-0.5)--(0,0)--(-1,0.5);
		\draw (5,-0.5)--(4,0)--(5,0.5);
		\draw (0,0)--(4,0);
	\end{tikzpicture}
	\quad\quad
	\begin{tikzpicture}[scale=0.5]
		\plotpartialperm{-1/-0.5,-1/0.5,0/0,1/0,2/0,3/0,4/0,5/0,6/-0.5,6/0.5};
		\draw (-1,-0.5)--(0,0)--(-1,0.5);
		\draw (6,-0.5)--(5,0)--(6,0.5);
		\draw (0,0)--(5,0);
	\end{tikzpicture}
\end{center}
\caption{The set of all double ended-forks, three of which are shown here, forms an infinite antichain in the induced subgraph order.}
\label{fig-double-ended-forks}
\end{figure}
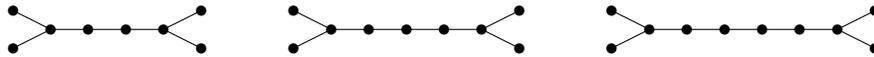

Even when the ultimate goal is to show that a given notion of embedding is wqo, experience suggests that it is often helpful to employ stronger properties than wqo. One much-studied example of such a stronger property is that of better-quasi-order, introduced in 1965 by Nash-Williams~\cite{nash-williams:on-well-quasi-o:infinite} and notably applied by Laver~\cite{laver:on-frai-sses-or:} to prove a conjecture of Fra\"\i ss\'e~\cite{fraisse:sur-la-comparai:} (see also Marcone~\cite{marcone:foundations-of-:}).

We explore the applications of a different strengthening of wqo---labelled well-quasi-order, lwqo for short, or h\'er\'editairement belordonn\'e in French---in the context of the containment order on permutations, the study of which is often called \emph{permutation patterns}.
While the notion of lwqo is implicit in the work of Higman~\cite{higman:ordering-by-div:} and Kruskal~\cite{kruskal:well-quasi-orde:}, it was not until the work of Pouzet in the 1970s (in particular, his 1972 paper~\cite{pouzet:un-bel-ordre-da:}) that this notion was made explicit%
\footnote{Another valuable reference for the early history of lwqo is Pouzet's 1985 survey paper~\cite[Section 3]{pouzet:applications-of:}, while Ding's 1992 paper~\cite{ding:subgraphs-and-w:} includes another rediscovery of the concept, in the not-necessarily-induced subgraph context.}.

The study of lwqo has recently received renewed attention in the induced subgraph context%
\footnote{We refer to Daligault, Rao, and Thomass\'e~\cite{daligault:well-quasi-orde:}, Korpelainen and Lozin~\cite{korpelainen:two-forbidden-i:}, Atminas and Lozin~\cite{atminas:labelled-induce:}, and Brignall, Engen, and Vatter~\cite{brignall:a-counterexampl:} for investigations of lwqo in the induced subgraph context.},
but the present work constitutes the first consideration of lwqo in the permutation context. The specific highlights of this work are as follows.
\begin{itemize}
\item {\bf Theorem~\ref{thm-lwqo-downward-closure}.} \emph{If a set of permutations is lwqo, then its downward closure is also lwqo.}
\item {\bf Theorem~\ref{thm-lwqo-C-lwqo-C+1}.} \emph{If a permutation class is lwqo, then so is the class of its one-point extensions.}
\item {\bf Theorem~\ref{thm-lwqo-sum-closure}.} \emph{If a permutation class is lwqo, then so are its sum closure and skew closure.}
\item {\bf Theorem~\ref{thm-subst-closure-lwqo}.} \emph{If a permutation class is lwqo, then so is its substitution closure.}
\item {\bf Theorem~\ref{thm-C-lwqo-GC-lwqo}.} \emph{A permutation class is lwqo if and only if the corresponding graph class is lwqo.}
\item {\bf Theorem~\ref{thm-ggc-lwqo}.} \emph{Every geometric grid class is lwqo.}
\end{itemize}
For the remainder of this introduction we review the various pieces of notation required for the later sections.

\subsection{Permutation containment and permutation classes}

In the course of this work, we view permutations in several slightly different ways, the most common being one-line, or list, notation. In this viewpoint, a permutation of length $n$ is simply an arrangement of the numbers $1$ through $n$ in a sequence. As done in Figure~\ref{fig-three-antichains}, we also often identify a permutation~$\pi$ with its \emph{plot}: the set of points $\{(i,\pi(i))\}$ in the plane. When we talk about an entry being to the left or right of, or above or below, another entry, we are referring to their relative positions in the plot of the permutation.

Every sequence of distinct real numbers is \emph{order-isomorphic}, or \emph{reduces}, to a unique permutation, namely the permutation whose entries are in the same relative order as the terms of the sequence. We call this permutation the \emph{reduction} of the sequence. For example, the sequence $3,-1,\nicefrac{22}{7},e$ is order-isomorphic to the permutation $3,1,4,2$, which we abbreviate to~$3142$. Given permutations $\sigma=\sigma(1)\cdots\sigma(k)$ and $\pi=\pi(1)\cdots\pi(n)$, we say that~$\sigma$ is \emph{contained} in~$\pi$ if~$\pi$ contains a subsequence that reduces to~$\sigma$. If~$\pi$ does not contain~$\sigma$, then we say that it \emph{avoids} it. For example, $\pi=432679185$ contains $\sigma=32514$, as witnessed by the subsequence~$32918$, but avoids $54321$ because it has no decreasing subsequence of length five.

A \emph{class} of permutations is a set of permutation closed downward under this containment order%
\footnote{In particular, this implies that every nonempty permutation class contains the empty permutation, that we denote by $\emptyperm$.}.
It is common to specify permutation classes by the permutations they avoid. Thus given any set $B$ of permutations, we define the class
\[
	\Av(B)=\{\pi : \pi\mbox{ avoids all $\beta\in B$}\}.
\]
We may always insist that the set $B$ in the above construction is an antichain; in that case $B$ is the set of \emph{minimal} permutations not in the class, $B$ uniquely describes the class, and we call $B$ the \emph{basis} of the class. For a comprehensive survey of permutation classes, we refer the reader to Vatter~\cite{vatter:permutation-cla:}.

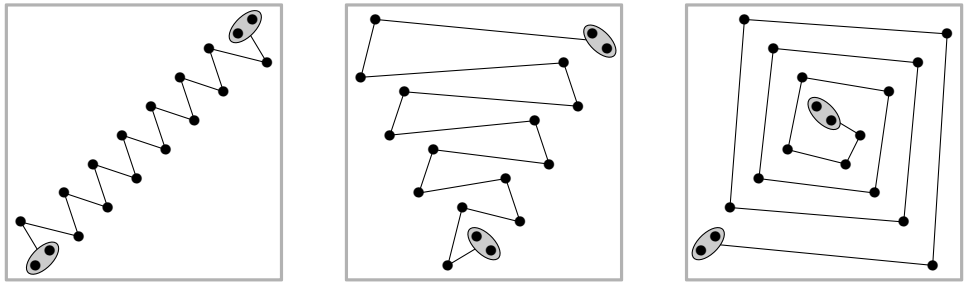
\begin{figure}
\begin{center}
	\begin{tikzpicture}[scale=0.1925, baseline=(current bounding box.south)]
		\draw (2.5,1.5)--(1,4)--(5,3)--(4,6)--(7,5)--(6,8)--(9,7)--(8,10)--(11,9)--(10,12)--(13,11)--(12,14)--(15,13)--(14,16)--(18,15)--(16.5,17.5);
		\draw[fill=lightgray, rotate around={-45:(2.5,1.5)}] (2.5,1.5) ellipse (20pt and 40pt);
		\draw[fill=lightgray, rotate around={-45:(16.5,17.5)}] (16.5,17.5) ellipse (20pt and 40pt);
		\plotpermbox{0.5}{0.5}{18.5}{18.5};
		\plotperm{4,1,2,6,3,8,5,10,7,12,9,14,11,16,13,17,18,15};
	\end{tikzpicture}
\quad\quad
	\begin{tikzpicture}[scale=0.1925, baseline=(current bounding box.south)]
		\draw (9.5,2.5)--(7,1)--(8,5)--(12,4)--(11,7)--(5,6)--(6,9)--(14,8)--(13,11)--(3,10)--(4,13)--(16,12)--(15,15)--(1,14)--(2,18)--(17.5,16.5);
		\draw[fill=lightgray, rotate around={45:(9.5,2.5)}] (9.5,2.5) ellipse (20pt and 40pt);
		\draw[fill=lightgray, rotate around={45:(17.5,16.5)}] (17.5,16.5) ellipse (20pt and 40pt);
		\plotpermbox{0.5}{0.5}{18.5}{18.5};
		\plotperm{14,18,10,13,6,9,1,5,3,2,7,4,11,8,15,12,17,16};
	\end{tikzpicture}
\quad\quad
	\begin{tikzpicture}[scale=0.1925, baseline=(current bounding box.south)]
		\draw (1.5,2.5)--(17,1)--(18,17)--(4,18)--(3,5)--(15,4)--(16,15)--(6,16)--(5,7)--(13,6)--(14,13)--(8,14)--(7,9)--(11,8)--(12,10)--(9.5,11.5);
		\draw[fill=lightgray, rotate around={-45:(1.5,2.5)}] (1.5,2.5) ellipse (20pt and 40pt);
		\draw[fill=lightgray, rotate around={45:(9.5,11.5)}] (9.5,11.5) ellipse (20pt and 40pt);
		\plotpermbox{0.5}{0.5}{18.5}{18.5};
		\plotperm{2,3,5,18,7,16,9,14,12,11,8,10,6,13,4,15,1,17};
	\end{tikzpicture}
\end{center}
\caption{Plots of typical members of three families of infinite antichains of permutations. It should be noted that the edges in these drawings are not formally defined and are intended only to demonstrate that these permutations are ``path-like'' in some sense, that shall also not be formally defined.}
\label{fig-three-antichains}
\end{figure}

It is frequently of interest whether classes given by certain structural definitions are \emph{finitely based}, that is, whether their bases are finite sets. Because the set of all permutations is not a wqo under the containment order, there are infinite antichains of permutations, and thus infinitely-based permutation classes. The generally-held intuition about the construction of these antichains is that their members consist of a ``body'' together with some irregularities at the ``beginning'' and ``end'' that form ``anchors''. For example, Figure~\ref{fig-three-antichains} shows members of three infinite antichains of permutations with their anchors enclosed in ellipses; these are the three antichains defined in the early work of Atkinson, Murphy, and Ru\v{s}kuc~\cite[Section 3]{atkinson:partially-well-:}. The typical approach to the construction of bodies is to establish that a smaller body can embed into a larger one only in some contiguous sense (a sense that varies with the form of the bodies in the particular construction but is always similar to how one path can embed as an induced subgraph into another).

Note that if a class is finitely based, then the membership problem for that class (``is the permutation~$\pi$ of length $n$ a member of~$\C$?'') can be answered in polynomial time (in $n$). For obvious cardinality reasons, the same cannot be said about the membership problem of a general permutation class (as there are uncountably many permutation classes but only countably many algorithms, there exist permutation classes for which the membership problem is undecidable).

\subsection{Well-quasi-order in general}
\label{subsec-wqo}

We begin with the formal definition. A quasi-ordering $\le$ on a set $X$ is \emph{well-quasi-ordered} or is a \emph{well-quasi-ordering} (both abbreviated \emph{wqo}) if every infinite sequence $x_1$, $x_2$, $\dots$ of elements from $X$ contains a \emph{good pair}, that is defined as a pair $(x_i,x_j)$ with $i<j$ and $x_i\le x_j$. As a trivial observation, note that finite quasi-orderings are always wqo, as any infinite sequence from such a quasi-ordering must contain two occurrences of the same element, which then form a good pair. The following two alternative characterisations follow easily from Ramsey-type arguments, and are essentially folklore%
\footnote{It is also not unreasonable to date these equivalent definitions to Higman's 1952 paper~\cite{higman:ordering-by-div:}, where they comprise three of the six parts of his Theorem~2.1---the definition of wqo we have given in terms of good pairs is Higman's condition (v), our Proposition~\ref{prop-wqo-ramsey} is his condition (vi), and our Proposition~\ref{prop-wqo-ramsey-inf-inc} is his condition (iv). For what it is worth, Higman does not himself give the proof of Proposition~\ref{prop-wqo-ramsey-inf-inc} (in his presentation, the equivalence of conditions (iv) and (v) of his Theorem~2.1), instead citing an unpublished manuscript of Erd\H{o}s and Rado for this result. Precisely \emph{which} then-unpublished manuscript of Erd\H{o}s and Rado this refers to is in a bit of doubt; in 1972, Kruskal~\cite[p.~300]{kruskal:the-theory-of-w:} wrote ``incidentally, Higman refers to an unpublished manuscript of Erd\H{o}s and Rado that was probably an early version of \cite{rado:partial-well-or:} or \cite{erdos:solution-to-pro:}''. Curiously, \cite{rado:partial-well-or:} is a single-authored paper by Rado, while \cite{erdos:solution-to-pro:} is a solution to a \emph{Monthly} problem posed by Erd\H{o}s in 1949~\cite{erdos:problem-4358:}. Another possibility is that the manuscript Higman refers to became~\cite{erdos:a-theorem-on-pa:}.}.

\begin{proposition}
\label{prop-wqo-ramsey}
A quasi-ordering $\le$ on the set $X$ is wqo if and only if $X$ contains neither an infinite antichain nor an infinite strictly decreasing sequence $x_1>x_2>\cdots$.
\end{proposition}
\begin{proof}
If $(X,\le)$ were to contain an infinite antichain or an infinite strictly decreasing sequence then it would not be wqo. Now suppose that $(X,\le)$ contains neither an infinite antichain nor an infinite strictly decreasing sequence and let $x_1,x_2,\dots$ be any infinite sequence of elements of $X$. Let~$G$ denote the complete graph on the vertices $\{1,2,\dots\}$. For $i<j$, colour the edge $ij$ of~$G$ one of three colours: red if $x_i\le x_j$, blue if $x_i>x_j$, or green if $x_i$ and $x_j$ are incomparable. By Ramsey's theorem,~$G$ must contain an infinite induced subgraph all of whose edges are the same colour. Because $(X,\le)$ contains neither an infinite antichain nor an infinite strictly decreasing subsequence, the edges of this induced subgraph cannot all be blue or green, so they must all be red. It follows that the sequence $x_1,x_2,\dots$ contains a good pair; in fact, it contains infinitely many.
\end{proof}

We say that a quasi-order without infinite strictly decreasing sequences is \emph{well founded}. Thus Proposition~\ref{prop-wqo-ramsey} implies that a well-founded quasi-order is wqo if and only if it does not contain an infinite antichain. In fact, our proof of Proposition~\ref{prop-wqo-ramsey} yields something seemingly much stronger.

\begin{proposition}
\label{prop-wqo-ramsey-inf-inc}
A quasi-ordering $\le$ on the set $X$ is wqo if and only if every infinite sequence $x_1$, $x_2$, $\dots$ of elements from~$X$ contains an infinite increasing subsequence, that is, there are indices $1\le i_1<i_2<\cdots$ such that $x_{i_1}\le x_{i_2}\le\cdots$.
\end{proposition}

From the result above we obtain the well-quasi-order of products quite easily.

\begin{proposition}
\label{prop-wqo-product}
If the quasi-orders $(X,\mathord{\le_X})$ and $(Y,\mathord{\le_Y})$ are both wqo, then the quasi-order $X\times Y$ is wqo under the product order in which $(x_1,y_1)\le (x_2,y_2)$ if and only if $x_1\le_X x_2$ and $y_1\le_Y y_2$.
\end{proposition}
\begin{proof}
Consider an infinite sequence $(x_1,y_1)$, $(x_2,y_2)$, $\dots$ from $X\times Y$. By Proposition~\ref{prop-wqo-ramsey-inf-inc}, the sequence $x_1$, $x_2$, $\dots$ contains an infinite increasing subsequence $x_{i_1}\le_X x_{i_2}\le_X \cdots$. Applying Proposition~\ref{prop-wqo-ramsey-inf-inc} to the subsequence $y_{i_1}$, $y_{i_2}$, $\dots$ shows that it also has an infinite increasing subsequence $y_{i_{j_1}}\le_Y y_{i_{j_2}}\le_Y \cdots$, so the subsequence $(x_{i_{j_1}}, y_{i_{j_1}})\le (x_{i_{j_2}}, y_{i_{j_2}})\le \cdots$ is an infinite increasing subsequence of our original sequence, and thus $(X\times Y,\le)$ is wqo by Proposition~\ref{prop-wqo-ramsey-inf-inc}.
\end{proof}

As an immediate consequence of Proposition~\ref{prop-wqo-product} we obtain the following result, that is often called Dickson's lemma because Dickson employed a special case of it in a 1913 paper~\cite{dickson:finiteness-of-t:}.

\begin{proposition}
\label{prop-wqo-vector}
For any well-quasi-order $(X,\le)$, the set of $n$-tuples over~$X$ is wqo under the product order, in which $(x_1,\dots,x_n)\le (x_1',\dots,x_n')$ if and only if $x_i\le x_i'$ for all indices~$i$.
\end{proposition}

We denote the quasi-order of $n$-tuples appearing in Proposition~\ref{prop-wqo-vector} by $(X^n,\le)$.

%
%
%

\subsection{Well-quasi-order for permutation classes}
\label{subsec-wqo-perm-classes}

Well-quasi-ordered permutation classes possess many favourable properties. Below we state two of these and give short proofs of them. A few notes are in order before that, however. First, note that we could have stated this result in much more generality, but we have instead chosen to specialise our treatment to the context of permutation classes. Second, this result is essentially folklore, and our reference to the work of Atkinson, Murphy, and Ru\v{s}kuc is simply the first place where a result such as this appeared in the literature on permutation classes.

\begin{proposition}[Cf.~Atkinson, Murphy, and Ru\v{s}kuc~{\cite[Proposition 1.1]{atkinson:partially-well-:}}]
The following conditions on a permutation class~$\C$ are equivalent:
\begin{enumerate}
\item[(a)]~$\C$ is wqo,
\item[(b)]~$\C$ contains at most countably many subclasses,
\item[(c)]~$\C$ satisfies the \emph{descending chain condition}, that is, there does not exist an infinite sequence
	\[
		\C=\C^{(0)}\supsetneq \C^{(1)}\supsetneq \C^{(2)}\supsetneq\cdots
	\]
	of subclasses of~$\C$.
\end{enumerate}
\end{proposition}
\begin{proof}
We first show that (a) and (b) are equivalent. Note that all subclasses of~$\C$ are of the form~$\C\cap\Av(B)$ for an antichain $B\subseteq\C$. Thus if (a) holds, then all such antichains $B$ are finite, and since~$\C$ is itself at most countable, it has at most countably many finite subsets. On the other hand, if~$\C$ were to contain an infinite antichain $A$, then
\[
	\{\C\cap\Av(B) : B\subseteq A\}
\]
would be an uncountable family of distinct subclasses of~$\C$.

Next we show that (a) and (c) are equivalent. Suppose to the contrary that the wqo class~$\C$ contains an infinite strictly decreasing sequence of subclasses~$\C=\C^{(0)}\supsetneq \C^{(1)}\supsetneq \C^{(2)}\supsetneq\cdots$. For each $i\ge 1$, choose $\beta_i\in\C^{(i-1)}\setminus\C^{(i)}$. The set of minimal elements of $\{\beta_1,\beta_2\ldots\}$ is an antichain and therefore finite, so there is an integer $m$ such that $\{\beta_1,\beta_2\ldots,\beta_m\}$ contains these minimal elements. In particular, $\beta_{m+1}\ge\beta_i$ for some $1\le i\le m$. However, we chose $\beta_{m+1}\in\C^{(m)}\setminus\C^{(m+1)}$, and because $\beta_{m+1}$ contains $\beta_i$, it does not lie in~$\C^{(i)}$ and thus cannot lie in~$\C^{(m)}$, a contradiction. To establish the other direction, suppose that~$\C$ is not wqo, so it contains an infinite antichain $A=\{\alpha_1,\alpha_2,\dots\}$. Then
\[
	\C\supsetneq
	\C\cap\Av(\{\alpha_1\})\supsetneq
	\C\cap\Av(\{\alpha_1,\alpha_2\})\supsetneq
	\cdots
\]
would be an infinite sequence of subclasses of~$\C$.
\end{proof}

\subsection{Induced subgraphs and classes of graphs}
\label{subsec-induced-subgraphs}

As demonstrated throughout this paper, studies of the permutation containment order and of the induced subgraph order are intimately linked. Let $G=(V,E)$ be a graph. Given a subset~${X\subseteq V}$ of vertices, $G[X]$ denotes the subgraph of~$G$ \emph{induced} by~$X$, which is the graph with vertex set~$X$ and an edge between two vertices $u,v\in X$ if and only if~$G$ contains an edge between~$u$ and~$v$. Alternatively, $G[X]$ can be formed from~$G$ by deleting all of the vertices in $V\setminus X$ and their incident edges.

A \emph{hereditary property} or (throughout this paper) \emph{class} of graphs is a set of finite graphs that is closed downward under the induced subgraph ordering and under isomorphism. Thus if~$\C$ is a class, $G\in\C$, and $H$ is an induced subgraph of~$G$, then $H\in\C$. Many natural sets of graphs form classes, such as the set of perfect graphs or the set of comparability graphs. For an extensive survey we refer to the encyclopaedic text of Brandst\"adt, Le, and Spinrad~\cite{brandstadt:graph-classes:-:}. A common way to describe a graph class is via its antichain of \emph{minimal forbidden induced subgraphs}, that is, the minimal (under the induced subgraph order) graphs that do not lie in the class. This is analogous to how a permutation class can be described by its basis.

Despite the fact that the set of chordless cycles forms an infinite antichain in this order, some important classes of graphs are nevertheless wqo under the induced subgraph ordering. Perhaps the most fundamental result of this type is that the class of \emph{cographs} (short for \emph{complement-reducible graphs}) is wqo in the induced subgraph order, as first observed by Damaschke~\cite[Theorem 4]{damaschke:induced-subgrap:}. This class is quite easily defined by its sole minimal forbidden induced subgraph, the path on four vertices $P_4$.

To give a more constructive definition of this class, recall that the \emph{join}, denoted by $G\ast H$, of the vertex-disjoint graphs~$G$ and $H$ is formed from the disjoint union $G\uplus H$ by adding all possible edges with one endpoint in~$G$ and the other in $H$. (In this context,~$G$ and $H$ are referred to as the \emph{join components} of the resulting graph.) Then a graph is a cograph if and only if, starting with the one-vertex graph $K_1$, it can be built by repeatedly taking the disjoint union or join of two cographs%
\footnote{The term \emph{complement-reducible graph} is due to a different version of this structural result: the one-vertex graph $K_1$ is a cograph, and a graph on two or more vertices is a cograph if and only if it or its complement can be expressed as the disjoint union of two smaller cographs.}.

\begin{figure}
\begin{footnotesize}
\begin{center}
	\begin{tabular}{ccccccc}
	\begin{tikzpicture}[scale=0.1925, baseline=(current bounding box.center)]
		\plotpermbox{0.5}{0.5}{8.5}{8.5};
		\plotpartialperm{1/3,4/8,5/5,7/1,8/4};
	\end{tikzpicture}
	&
	\begin{tikzpicture}[baseline=(current bounding box.center)]
		\node {$\le$};
	\end{tikzpicture}
	&
	\begin{tikzpicture}[scale=0.1925, baseline=(current bounding box.center)]
		\plotpermbox{0.5}{0.5}{8.5}{8.5};
		\plotperm{3,6,2,8,5,7,1,4};
		\plotpartialpermencirclewhite{1/3, 4/8, 5/5, 7/1, 8/4};
	\end{tikzpicture}
	&
	\quad\quad\quad\quad
	&
	\begin{tikzpicture}[scale=0.1925, baseline=(current bounding box.center)]
	    \draw (4,8) to [out=324.666666667, in=117] (7,1);
	    \draw (5,5) to [out=270, in=144] (7,1);
	    \draw (1,3) to [out=270, in=225] (7,1);
	    \draw (4,8) to [out=45, in=45] (8,4);
	    \draw (5,5) to [out=315, in=180] (8,4);
	    \draw (4,8) to [out=284.5, in=90] (5,5);
		\plotpermbox{0.5}{0.5}{8.5}{8.5};
		\plotpartialperm{1/3,4/8,5/5,7/1,8/4};
	\end{tikzpicture}
	&
	\begin{tikzpicture}[baseline=(current bounding box.center)]
		\node {$\le$};
	\end{tikzpicture}
	&
	\begin{tikzpicture}[scale=0.1925, baseline=(current bounding box.center)]
	    \draw (6,7) to [out=300, in=90] (7,1);
	    \draw (4,8) to [out=324.666666667, in=117] (7,1);
	    \draw (5,5) to [out=270, in=144] (7,1);
	    \draw (2,6) to [out=315, in=171] (7,1);
	    \draw (3,2) to [out=315, in=198] (7,1);
	    \draw (1,3) to [out=270, in=225] (7,1);
	    \draw (4,8) to [out=45, in=45] (8,4);
	    \draw (6,7) to [out=330, in=90] (8,4);
	    \draw (2,6) to [out=45, in=135] (8,4);
	    \draw (5,5) to [out=315, in=180] (8,4);
	    \draw (1,3) to [out=0, in=135] (3,2);
	    \draw (2,6) to [out=270, in=90] (3,2);
	    \draw (2,6) to [out=0] (5,5);
	    \draw (4,8) to [out=284.5, in=90] (5,5);
	    \draw (4,8) to [out=4.833333333] (6,7);
		\plotpermbox{0.5}{0.5}{8.5}{8.5};
		\plotperm{3,6,2,8,5,7,1,4};
		\plotpartialpermencirclewhite{1/3, 4/8, 5/5, 7/1, 8/4};
	\end{tikzpicture}
	\\
	$25413$
	&
	$\le$
	&
	$36285714$
	&&
	$G_{25413}$
	&
	$\le$
	&
	$G_{36285714}$
	\end{tabular}
\end{center}
\end{footnotesize}
\caption{The containment order on permutations and their corresponding inversion graphs.}
\label{fig-perm-contain}
\end{figure}
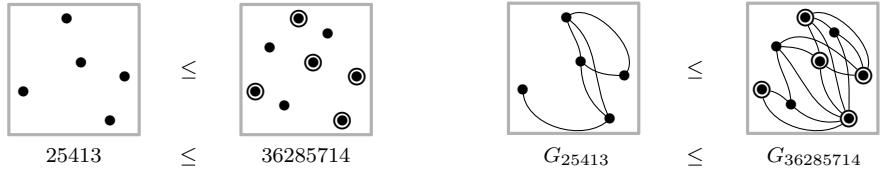

\subsection{Inversion graphs of permutations}
\label{subsec-inversion-graphs}

The connection between the permutation containment order and the induced subgraph order on graphs comes via what we call inversion graphs (but which are more commonly called permutation graphs in the graph theory literature). The \emph{inversion graph} of the permutation $\pi=\pi(1)\cdots\pi(n)$ is the graph $G_\pi$ on the vertices $\{1,\dots,n\}$ in which~$i$ is adjacent to $j$ if and only if $\pi(i)$ and $\pi(j)$ form an \emph{inversion}, meaning that $i<j$ and $\pi(i)>\pi(j)$. In the graph context, we typically only care about isomorphism classes, and so this mapping is many-to-one as witnessed by the fact that $G_{2413}\cong G_{3142}\cong P_4$.

As shown on the right of Figure~\ref{fig-perm-contain}, to obtain the inversion graph of a permutation from its plot we simply add all edges between pairs of entries in which one lies northwest of the other. Figure~\ref{fig-perm-contain} should also convince the reader that if~$\sigma$ is contained in~$\pi$ then $G_\sigma$ is an induced subgraph of $G_\pi$. However, the converse does not hold generally (returning to our example from above, $G_{2413}$ is an induced subgraph of $G_{3142}$ because the two graphs are isomorphic, but of course the permutation $2413$ is not contained in the permutation $3142$). The most one can say in general is the following.

\begin{proposition}
\label{prop-graph-to-perm-containment}
If $G_\sigma$ is an induced subgraph of $G_\pi$, then there is a permutation $\tau\le\pi$ such that $G_\tau\cong G_\sigma$.
\end{proposition}
\begin{proof}
A witness to the embedding of $G_\sigma$ in $G_\pi$ is a set of vertices of $G_\pi$ that forms an induced subgraph isomorphic to $G_\sigma$. Thus, we may take $\tau$ to be the permutation that is order isomorphic to the corresponding set of entries in~$\pi$.
\end{proof}

\begin{figure}
\begin{center}
	\begin{footnotesize}
	\begin{tikzpicture}[scale=1.5]
		\draw [darkgray, ultra thick, rounded corners=0.01, line cap=round] (-1,-1) rectangle (1,1);
		\draw [darkgray, dashed] (0,0)--(1.2,0);
		\draw [darkgray, dashed] (0,0)--(-1.2,0);
		\draw [->] (0,1.4)-- +(0.2,0);
		\draw [->] (0,1.4)-- +(-0.2,0);
		\node at (0,1.4) [above] {$\pi^{\textrm{r}}$};
		\draw [darkgray, dashed] (0,0)--({1+sqrt(2)/10}, {1+sqrt(2)/10});
		\draw [darkgray, dashed] (0,0)--({-1-sqrt(2)/10}, {-1-sqrt(2)/10});
		\draw [->] ({1+2*sqrt(2)/10}, {1+2*sqrt(2)/10})-- +({sqrt(2)/10}, {-sqrt(2)/10});
		\draw [->] ({1+2*sqrt(2)/10}, {1+2*sqrt(2)/10})-- +({-sqrt(2)/10}, {sqrt(2)/10});
		\node at ({1+1.8*sqrt(2)/10}, {1+1.8*sqrt(2)/10}) [above right] {$\pi^{-1}$};
		\draw [darkgray, dashed] (0,0)--(0,1.2);
		\draw [darkgray, dashed] (0,0)--(0,-1.2);
		\draw [->] (1.4,0)-- +(0,0.2);
		\draw [->] (1.4,0)-- +(0,-0.2);
		\node at (1.4,0) [right] {$\pi^{\textrm{c}}$};
		\draw [darkgray, dashed] (0,0)--({1+sqrt(2)/10}, {-1-sqrt(2)/10});
		\draw [darkgray, dashed] (0,0)--({-1-sqrt(2)/10}, {1+sqrt(2)/10});
		\draw [->] ({-1-2*sqrt(2)/10}, {1+2*sqrt(2)/10})-- +({sqrt(2)/10}, {sqrt(2)/10});
		\draw [->] ({-1-2*sqrt(2)/10}, {1+2*sqrt(2)/10})-- +({-sqrt(2)/10}, {-sqrt(2)/10});
		\node at ({-1-1.8*sqrt(2)/10}, {1+1.8*sqrt(2)/10}) [above left] {$\left((\pi^{\textrm{r}})^{-1}\right)^{\textrm{r}}$};
		\draw [->] (0.3,0) arc (0:270: 0.3);
		\node at ({0.3*cos(247.5)-0.065}, {0.3*sin(247.5)-0.025}) [below] {$(\pi^{\textrm{r}})^{-1}$};
		\draw [->] (0.5,0) arc (0:180: 0.5);
		\node at ({0.5*cos(157.5)}, {0.5*sin(157.5)}) [left] {$\pi^{\textrm{rc}}$};
		\draw [->] (0.7,0) arc (0:90: 0.7);
		\node at ({0.7*cos(67.5)}, {0.7*sin(67.5)+0.025}) [above] {$(\pi^{-1})^{\textrm{r}}$};
	\end{tikzpicture}
	\end{footnotesize}
\end{center}
\caption{The symmetries of the square, labelled by their effect on a permutation~$\pi$.}
\label{fig-symmetries-square-perms}
\end{figure}
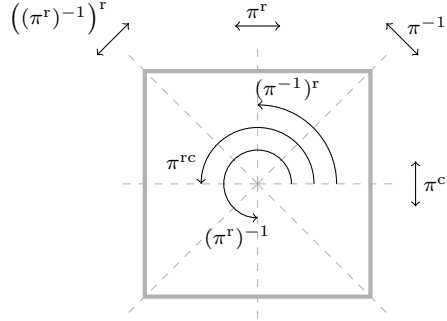

When identifying permutations with their plots, it is clear that the permutation containment order respects all eight symmetries of the square shown in Figure~\ref{fig-symmetries-square-perms}. Of these symmetries, three are particularly important to this work: the \emph{group-theoretic inverse}, $\pi^{-1}$, obtained by reflecting the plot of~$\pi$ about the line $y=x$; the \emph{reverse complement}, $\pi^{\text{rc}}$, obtained by reflecting the plot of~$\pi$ about the line $y=-x$ (and then shifting); and the symmetry obtained by composing these two, $(\pi^{\text{rc}})^{-1}$. Note these symmetries do not affect the corresponding inversion graphs: for all permutations~$\pi$, we have
\[
	G_\pi
	\cong
	G_{\pi^{-1}}
	\cong
	G_{\pi^{\text{rc}}}
	\cong
	G_{(\pi^{\text{rc}})^{-1}}.
\]

Complete graphs are inversion graphs because $K_k\cong G_{k\cdots 21}$. Indeed, $k\cdots 21$ is the only permutation whose inversion graph is (isomorphic to) $K_k$, so every clique in $G_\pi$ arises from a decreasing subsequence of~$\pi$. By symmetry, $12\cdots \ell$ is the only permutation whose inversion graph has no edges, so every independent set in $G_\pi$ arises from an increasing subsequence of~$\pi$. Of course, not every graph is the inversion graph of a permutation. For example, induced cycles of length five or more never appear in inversion graphs.

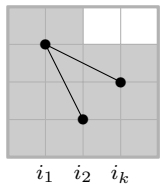
\begin{figure}
\begin{footnotesize}
\begin{center}
	\begin{tikzpicture}[scale=0.5, baseline=(current bounding box.center)]
		\draw [fill=lightgray, color=lightgray] (0,0) rectangle (2,4);
		\draw [fill=lightgray, color=lightgray] (0,0) rectangle (4,3);
		\draw [color=darkgray, line cap=round] (0,0) grid (4,4);
		\plotpermbox{0.5}{0.5}{3.5}{3.5};
		\draw (2,1)--(1,3)--(3,2);
		\plotpartialperm{1/3,2/1,3/2};
		\node at (1,0) [below] {$i_1$};
		\node at (2,0) [below] {$i_2$};
		\node at (3,0) [below] {$i_k$};
	\end{tikzpicture}
\end{center}
\end{footnotesize}
\caption{Inversion graphs do not contain induced cycles on five or more vertices.}
\label{fig-no-cycles}
\end{figure}

\begin{proposition}
\label{prop-no-cycles}
For all $k\ge 5$, the cycle $C_k$ is not an induced subgraph of any inversion graph.
\end{proposition}
\begin{proof}
Suppose to the contrary that some cycle $C_k$ for $k\ge 5$ were contained as an induced subgraph in the inversion graph $G_\pi$. Let $\pi(i_1)$ denote the leftmost entry of~$\pi$ that corresponds to a vertex in this copy of $C_k$. Let the indices of~$\pi$ that correspond to the other vertices of this cycle be $i_2,\dots,i_k$, so that
\(
	i_1\sim i_2\sim \cdots\sim i_k\sim i_1
\)
in $G_\pi$.

We may assume that $\pi(i_2)$ lies to the left of $\pi(i_k)$ because otherwise we could consider the cycle in the reverse order. The vertices $i_2$ and $i_k$ are not adjacent because $k\ge 4$, so they must correspond to a noninversion in~$\pi$, and since $\pi(i_1)$ lies to the left of all other entries corresponding to vertices of this cycle, $\pi(i_2)$ and $\pi(i_k)$ must lie to the southeast of $\pi(i_1)$, as shown in Figure~\ref{fig-no-cycles}. In this figure, the shaded regions cannot contain any other entries of~$\pi$ that correspond to vertices of the cycle for various reasons:
(i) $\pi(i_1)$ is the leftmost such entry;
(ii) only~$\pi(i_2)$ and $\pi(i_k)$ may lie southeast of $\pi(i_1)$ since no other vertex on the cycle is adjacent to~$i_1$;
and
(iii) since $k\ge 5$, there is no vertex adjacent to both $i_2$ and $i_k$ other than $i_1$.
However, this implies that we cannot finish the cycle---there is no way that there could be a vertex adjacent to~$i_2$ but not to $i_1$ or $i_k$---completing the proof.
\end{proof}

Inversion graphs can contain induced cycles of lengths $3$ and $4$, because ${G_{321}\cong K_3\cong C_3}$ and~${G_{3412}\cong C_4}$, but it follows from Proposition~\ref{prop-no-cycles} and an investigation of permutations of lengths~$3$ and~$4$ that~$321$ and~$3412$ are the only permutations whose inversion graphs are isomorphic to cycles. Thus we immediately obtain characterisations of the bipartite and acyclic inversion graphs.

\begin{proposition}
The bipartite inversion graphs are precisely the inversion graphs of permutations in the class $\Av(321)$.
\end{proposition}

\begin{proposition}
\label{prop-acyclic-inversion-graphs}
The acyclic inversion graphs are precisely the inversion graphs of permutations in the class $\Av(321, 3412)$.
\end{proposition}

There has been extensive study of both the $321$-avoiding permutations~\cite{richards:ballot-sequence:,billey:kazhdan-lusztig:,mansour:321-polygon-avo:,stankova:explicit-enumer:,guillemot:pattern-matchin:,albert:growth-rates-fo:,albert:generating-and-:,albert:rationality-for:} and their graphical analogues, the bipartite inversion graphs~\cite{spinrad:bipartite-permu:,lai:bipartite-permu:,brandstadt:on-the-linear-s:,lozin:minimal-univers:,uehara:linear-structur:,korpelainen:bipartite-induc:,lozin:canonical-antic:,kiyomi:bipartite-permu:,heggernes:induced-subgrap:,kiyomi:finding-a-chain:,lin:linear-time-alg:}.
The acyclic inversion graphs and the corresponding permutation class $\Av(321, 3412)$ have not received nearly as much attention, although Tenner~\cite{tenner:pattern-avoidan:} and Petersen and Tenner~\cite{petersen:the-depth-of-a-:} have considered them from the Bruhat order perspective.

\subsection{Order-preserving and reflecting mappings}
\label{subsec-order-pres}

The mapping $\pi\mapsto G_\pi$ from permutations to their inversion graphs is order-preserving because~$G_\sigma$ is an induced subgraph of $G_\pi$ whenever~$\sigma$ is contained in~$\pi$. Such mappings arise frequently in our proofs; in general, a mapping $\Phi : (X,\mathord{\le_X})\to (Y,\mathord{\le_Y})$ from one poset to another is \emph{order-preserving}~if 
\[
	x_1 \le_X x_2
	\implies
	\Phi(x_1) \le_Y \Phi(x_2)
\]
for all $x_1,x_2\in X$. In addition to the mapping $\pi\mapsto G_\pi$, we note that every mapping from an antichain to a poset is order-preserving. We frequently employ the following elementary fact.

\begin{proposition}
\label{prop-wqo-order-preserving}
Suppose that $(X,\mathord{\le_X})$ and $(Y,\mathord{\le_Y})$ are quasi-orders and that the mapping
\[
	\Phi : (X,\mathord{\le_X})\tosurj (Y,\mathord{\le_Y})
\]
is an order-preserving surjection. If $(X,\mathord{\le_X})$ is wqo, then $(Y,\mathord{\le_Y})$ is also wqo.
\end{proposition}
\begin{proof}
Let $y_1,y_2,\dots$ be an infinite sequence of elements from $Y$. Because $\Phi$ is surjective, for each $y_i$ we can choose some $x_i\in X$ such that $\Phi(x_i)=y_i$. Because $(X,\mathord{\le_X})$ is wqo, the sequence $x_1,x_2,\dots$ has a good pair, that is, there are indices $i<j$ so that $x_i\le_X x_j$. It follows that $y_i=\Phi(x_i)\le\Phi(x_j)=y_j$, so the sequence $y_1,y_2,\dots$ also has a good pair.
\end{proof}

Applying Proposition~\ref{prop-wqo-order-preserving} in this context immediately yields the following result relating wqo permutation classes and wqo classes of graphs. Here and in what follows, if~$X$ is a set (or class) of permutations, then we denote by $G_X$ the set (or class) of inversion graphs of its members.

\begin{proposition}
\label{prop-wqo-perms-graphs}
Let~$\C$ be a permutation class and $G_\C$ the corresponding graph class.
\begin{enumerate}
\item[(a)] If~$\C$ is wqo in the permutation containment order, then $G_\C$ is wqo in the induced subgraph order.
\item[(b)] Contrapositively, if $G_\C$ is \emph{not} wqo in the induced subgraph order, then~$\C$ is not wqo in the permutation containment order.
\end{enumerate}
\end{proposition}

Intriguingly, the converse of Proposition~\ref{prop-wqo-perms-graphs} is not known to hold. (Although see Proposition~\ref{prop-GC-wqo-C-simples-wqo} for a partial answer.)

\begin{question}
\label{question-prop-wqo-perms-graphs-converse}
Let~$\C$ be a permutation class and $G_\C$ the corresponding graph class. If $G_\C$ is wqo in the induced subgraph order, must~$\C$ be wqo in the permutation containment order?
\end{question}

Recall from Section~\ref{subsec-inversion-graphs} that $G_\pi\cong G_{\pi^{-1}}\cong G_{\pi^{\text{rc}}}\cong G_{(\pi^{\text{rc}})^{-1}}$ for all permutations~$\pi$. Thus
\[
	G_X = G_{X\,\cup\,X^{-1}\,\cup\,X^{\text{rc}}\,\cup\,(X^{\text{rc}})^{-1}}
\]
for all sets~$X$ of permutations. These considerations show that an affirmative answer to Question~\ref{question-prop-wqo-perms-graphs-converse} cannot follow from the same argument as used to prove Proposition~\ref{prop-wqo-perms-graphs}---given an antichain $A$ of permutations, it is certainly not always the case that the corresponding set of inversion graphs $G_A$ is also an antichain.

It is frequently more convenient to work backward; the mapping $\Psi : (X,\mathord{\le_X})\to (Y,\mathord{\le_Y})$ is \emph{order-reflecting}%
\footnote{Note that what is ``reflected'' in the definition of order-reflecting is the implication arrow.}
if
\[
	x_1 \le_X x_2
	\impliedby
	\Psi(x_1) \le_Y \Psi(x_2)
\]
for all $x_1,x_2\in X$. Note that if $(X,\le_X)$ is a poset (as when we restrict our attention to permutation classes) and $\Psi$ is an order-reflecting mapping with domain $(X,\le)$, then $\Psi$ must be injective: if $\Psi(x_1)=\Psi(x_2)$, then we have both $x_1\le_X x_2$ and $x_2\le_X x_1$, which implies that~${x_1=x_2}$. This fact tends to motivate constructions of order-reflecting mappings---we must be able to ``reconstruct'' $x$ from $\Psi(x)$---although this is not a sufficient condition for $\Psi$ to be order-reflecting.

The analogue of Proposition~\ref{prop-wqo-order-preserving} for order-reflecting mappings follows easily.

\begin{proposition}
\label{prop-wqo-order-reflecting}
Suppose that $(X,\mathord{\le_X})$ and $(Y,\mathord{\le_Y})$ are quasi-orders and that the mapping
\[
	\Psi : (X,\mathord{\le_X})\to (Y,\mathord{\le_Y})
\]
is order-reflecting. If $(Y,\mathord{\le_Y})$ is wqo then $(X,\mathord{\le_X})$ is also wqo.
\end{proposition}
\begin{proof}
Let $x_1,x_2,\dots$ be any infinite sequence of elements from~$X$. Because $(Y,\mathord{\le_Y})$ is wqo, the sequence $\Psi(x_1),\Psi(x_2),\dots$ has a good pair, meaning that there are indices $i<j$ such that~${\Psi(x_i)\le_Y \Psi(x_j)}$. It follows that $x_i\le_X x_j$, so the sequence $x_1,x_2,\dots$ also has a good pair.
\end{proof}

\subsection{Sums and increasing oscillations}
\label{subsec-sums}

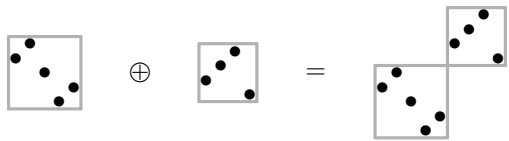
\begin{figure}
\begin{center}
\begin{tabular}{ccccc}
	\begin{tikzpicture}[scale=0.1925, baseline=(current bounding box.center)]
		\plotpermbox{1}{1}{5}{5};
		\plotperm{4,5,3,1,2};
	\end{tikzpicture}
&
	\begin{tikzpicture}[baseline=(current bounding box.center)]
		\node {$\oplus$};
	\end{tikzpicture}
&
	\begin{tikzpicture}[scale=0.1925, baseline=(current bounding box.center)]
		\plotpermbox{1}{1}{4}{4};
		\plotperm{2,3,4,1};
	\end{tikzpicture}
&
	\begin{tikzpicture}[baseline=(current bounding box.center)]
		\node {$=$};
	\end{tikzpicture}
&
	\begin{tikzpicture}[scale=0.1925, baseline=(current bounding box.center)]
		\plotpermbox{1}{1}{5}{5};
		\plotpermbox{6}{6}{9}{9};
		\plotperm{4,5,3,1,2,7,8,9,6};
	\end{tikzpicture}
\end{tabular}
\end{center}
\caption{An example of a direct sum: $45312\oplus 2341=45312\ 7896$.}
\label{fig-ex-sum}
\end{figure}

The (direct) \emph{sum} of the permutations~$\sigma$ of length $m$ and $\tau$ of length $n$ is the permutation $\sigma\oplus\tau$ defined by
\[
	(\sigma\oplus\tau)(i) =
	\left\{
	\begin{array}{ll}
	\sigma(i)&\mbox{for $1\le i\le m$,}\\
	\tau(i-m)+m&\mbox{for $m+1\le i\le m+n$.}
	\end{array}
	\right.
\]
The plot of $\sigma\oplus\tau$ consists of the plot of $\tau$ above and to the right of the plot of~$\sigma$, as shown in Figure~\ref{fig-ex-sum}. Analogously, given permutations~$\sigma$ of length $m$ and $\tau$ of length $n$, their \emph{skew sum} is the permutation $\sigma\ominus\tau$ defined by
\[
	(\sigma\ominus\tau)(i) =
	\left\{
	\begin{array}{ll}
	\sigma(i)+n&\mbox{for $1\le i\le m$,}\\
	\tau(i-m)&\mbox{for $m+1\le i\le m+n$.}
	\end{array}
	\right.
\]
A permutation is \emph{sum indecomposable} (or \emph{connected}) if it cannot be expressed as the direct sum of two shorter permutations and \emph{skew sum indecomposable} (or simply \emph{skew indecomposable}) if it cannot be expressed as the skew sum of two shorter permutations. We leave the routine proof of the following result to the reader.

\begin{proposition}\label{prop-sum-indecomp-connected}
The permutation~$\pi$ is sum indecomposable if and only if $G_\pi$ is connected.
\end{proposition}

A permutation is \emph{separable} if it can be built from the permutation $1$ using only sums and skew sums. For example, the permutation $453127896$ of Figure~\ref{fig-ex-sum} is separable:
\begin{eqnarray*}
453127896
&=&45312\oplus 2341\\
&=&(12\ominus 1\ominus 12)\oplus(123\ominus 1)\\
&=&((1\oplus 1)\ominus 1\ominus (1\oplus 1))\oplus((1\oplus 1\oplus 1)\ominus 1).
\end{eqnarray*}
The term ``separable'' is due to Bose, Buss, and Lubiw~\cite{bose:pattern-matchin:}, who proved that the separable permutations are $\Av(2413,3142)$, although these permutations first appeared in the much earlier work of Avis and Newborn~\cite{avis:on-pop-stacks-i:}. The graphical analogues of separable permutations are the cographs defined in Section~\ref{subsec-induced-subgraphs}. More precisely, an inductive argument quickly yields the following.

\begin{proposition}\label{prop-cograph-separable}
The cographs are precisely the inversion graphs of separable permutations.
\end{proposition}

We conclude our discussion of inversion graphs by considering those permutations whose inversion graphs are paths. Proposition~\ref{inc-osc-path}, below, shows that these permutations are precisely the sum indecomposable permutations that are order isomorphic to subsequences of the \emph{increasing oscillating sequence},
\[
	2,4,1,6,3,8,5,\dots,2k,2k-3,\dots.
\]
We call such permutations \emph{increasing oscillations}%
\footnote{The term increasing oscillation dates to Murphy's thesis~\cite{murphy:restricted-perm:}, although we note that under our definition the permutations $1$, $21$, $231$, and $312$ are increasing oscillations while in his thesis and some other works they are not.}.
Thus the set of increasing oscillations is
\[
	\{1,
	21,
	231,
	312,
	2413,
	3142,
	24153,
	31524,
	241635,
	315264,
	2416375,
	3152746,
	\dots\}.
\]
As promised, we show that these permutations have the property we seek.

\begin{proposition}\label{inc-osc-path}
The inversion graph $G_\pi$ is a path if and only if~$\pi$ is an increasing oscillation.
\end{proposition}
\begin{proof}
It is evident that the inversion graphs of increasing oscillations are paths. To establish the other direction, suppose that $G_\pi$ is a path on $k\ge 2$ vertices, and denote the indices of~$\pi$ as~${i_1,\dots,i_k}$ so that
\[
	i_1\sim i_2\sim \cdots \sim i_k
\]
in $G_\pi$. Without loss of generality, we may assume that $\pi(i_1)$ lies to the left of $\pi(i_k)$. Since $i_1$ is adjacent only to $i_2$ in $G_\pi$, $\pi(i_1)$ is either the bottommost or the leftmost entry of~$\pi$, and we assume it is the leftmost (the other argument being a symmetry of this). 

The placement of $\pi(i_2)$ is now determined: it must lie southeast of $\pi(i_1)$. The entry $\pi(i_3)$---if it exists---cannot lie southeast of $\pi(i_1)$, and thus must lie above $\pi(i_1)$ and horizontally between~$\pi(i_1)$ and~$\pi(i_2)$. Continuing iteratively in this manner, we see that for $\ell>1$, $\pi(i_{2\ell})$ must lie to the right of $\pi(i_{2\ell-2})$ and vertically between $\pi(i_{2\ell-2})$ and $\pi(i_{2\ell-1})$, while $\pi(i_{2\ell+1})$ must lie above $\pi(i_{2\ell-1})$ and horizontally between $\pi(i_{2\ell-1})$ and $\pi(i_{2\ell})$. The resulting permutation is the increasing oscillation of length $n$ whose first entry is $2$.
\end{proof}

\begin{figure}
\begin{center}
	\begin{tikzpicture}[scale=0.1925, baseline=(current bounding box.center)]
		\plotpermbox{0.5}{0.5}{16.5}{16.5};
		\plotpermgraph{4,1,2,6,3,8,5,10,7,12,9,14,11,15,16,13};
	\end{tikzpicture}
\quad\quad
	\begin{tikzpicture}[scale=0.1925, baseline=(current bounding box.center)]
		\plotpermbox{0.5}{0.5}{16.5}{16.5};
		\draw (1,2)--(4,1)--(3,4)--(6,3)--(5,6)--(8,5)--(7,8)--(10,7)--(9,10)--(12,9)--(11,12)--(14,11)--(13,14)--(16,13)--(15,16);
		\draw (2,15) to[out=255, in=135] (4,1);
		\draw (2,15) to[out=265, in=110] (3,4);
		\draw (2,15) to[out=275, in=135] (6,3);
		\draw (2,15) to[out=285, in=115] (5,6);
		\draw (2,15) to[out=295, in=135] (8,5);
		\draw (2,15) to (7,8);
		\draw (2,15) to (10,7);
		\draw (2,15) to (9,10);
		\draw (2,15) to[out=-25, in=135] (12,9);
		\draw (2,15) to[out=-15, in=155] (11,12);
		\draw (2,15) to[out=-5, in=135] (14,11);
		\draw (2,15) to[out=5, in=160] (13,14);
		\draw (2,15) to[out=15, in=135] (16,13);
		\plotperm{2,15,4,1,6,3,8,5,10,7,12,9,14,11,16,13};
	\end{tikzpicture}
\quad\quad
		\begin{tikzpicture}[scale=0.1925, baseline=(current bounding box.center)]
		\plotpermbox{0.5}{0.5}{15.5}{15.5};
		\draw (2,2)--(5,1)--(4,4)--(7,3)--(6,6)--(9,5)--(8,8)--(11,7)--(10,10)--(13,9)--(12,12)--(15,11)--(14,14);
		\draw (1,13) to[out=15, in=135] (15,11);
		\draw (1,13) to[out=5, in=150] (12,12);
		\draw (1,13) to[out=-5, in=135] (13,9);
		\draw (1,13) to[out=-15, in=145] (10,10);
		\draw (1,13) to[out=-25, in=135] (11,7);
		\draw (1,13) to[out=-35, in=135] (8,8);
		\draw (1,13) to (9,5); 
		\draw (1,13) to[out=-55, in=135] (6,6);
		\draw (1,13) to[out=-65, in=135] (7,3);
		\draw (1,13) to[out=-75, in=125] (4,4);
		\draw (1,13) to[out=-85, in=140] (5,1);
		\draw (1,13) to[out=-95, in=115] (2,2);
		\draw (3,15) to[out=-105, in=135] (5,1);
		\draw (3,15) to[out=-95, in=110] (4,4);
		\draw (3,15) to[out=-85, in=135] (7,3);
		\draw (3,15) to[out=-75, in=115] (6,6);
		\draw (3,15) to[out=-65, in=135] (9,5);
		\draw (3,15) to[out=-55, in=135] (8,8);
		\draw (3,15) to (11,7); 
		\draw (3,15) to[out=-35, in=135] (10,10);
		\draw (3,15) to[out=-25, in=135] (13,9);
		\draw (3,15) to[out=-15, in=145] (12,12);
		\draw (3,15) to[out=-5, in=130] (15,11);
		\draw (3,15) to[out=5, in=155] (14,14);
		\plotperm{13,2,15,4,1,6,3,8,5,10,7,12,9,14,11};
		\end{tikzpicture}
\end{center}
\caption{The inversion graphs of typical members of three different infinite antichains of permutations based on the increasing oscillating sequence.}
\label{fig-three-inc-osc-antichains}
\end{figure}
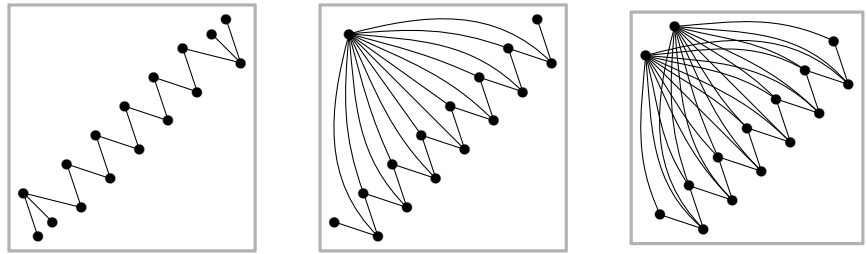

Increasing oscillations can be used to build several infinite antichains. One such antichain is pictured on the left of Figure~\ref{fig-three-antichains}, another member of which is shown on the left of Figure~\ref{fig-three-inc-osc-antichains}. This is essentially the same as the antichain constructed in 2000 by Spielman and B\'ona~\cite{spielman:an-infinite-ant:}. However, infinite antichains of permutations date back at least to the early 1970s, if not the late 1960s. In his 1972 paper~\cite[Lemma 6]{tarjan:sorting-using-n:}, Tarjan presented the antichain shown in the centre of Figure~\ref{fig-three-inc-osc-antichains}. A year later, Pratt~\cite[Figure 3]{pratt:computing-permu:} presented the antichain shown on the right of Figure~\ref{fig-three-inc-osc-antichains}. However, Tarjan's construction may have been preceded by a construction of Laver~\cite[pg. 9]{laver:well-quasi-orde:}; while Laver's paper was not published until 1976, his antichain is mentioned in Kruskal's 1972 paper~\cite[pg. 304]{kruskal:the-theory-of-w:}, and Laver writes that this antichain is derived from a construction presented in the penultimate paragraph of Jenkyns and Nash-Williams's 1968 paper~\cite{jenkyns:counterexamples:}. (The Jenkyns--Nash-Williams construction is not an infinite antichain of permutations, but it nevertheless bears a striking resemblance to the antichains based on the increasing oscillating sequence.)

A \emph{linear forest} is a disjoint union of paths. Let~$\O_I$ denote the class of permutations whose graphs are linear forests. It follows immediately from Proposition~\ref{inc-osc-path} that the class~$\O_I$ consists precisely of all permutations that are order isomorphic to subsequences of the increasing oscillating sequence. Viewed from this angle, the result below gives the basis of this permutation class. This result was first stated without proof in Murphy's thesis~\cite[Proposition 36]{murphy:restricted-perm:} and a proof was later given by Brignall, Ru\v{s}kuc, and Vatter~{\cite[Proposition 14]{brignall:simple-permutat:decide:}}, but the proof below using inversion graphs is much more straight-forward.

\begin{proposition}
\label{prop-OI-fb}
The linear forests are precisely the inversion graphs of permutations in the class $\O_I=\Av(321, 2341,\allowbreak 3412, 4123)$.
\end{proposition}

\begin{proof}
The inversion graphs $G_{321}$ and $G_{3412}$ are cycles, and the inversion graphs $G_{2341}$ and $G_{4123}$ are both isomorphic to the claw $K_{1,3}$. Thus if $G_\pi$ is a disjoint union of paths,~$\pi$ must avoid $321$, $3412$, $2341$, and $4123$. In the other direction, Proposition~\ref{prop-acyclic-inversion-graphs} shows us that $G_\pi$ is acyclic if~$\pi\in\Av(321,3412)$, and if~$\pi$ further avoids $2341$ and $4123$, then $G_\pi$ cannot have a vertex of degree $3$ or greater, so $G_\pi$ is indeed a disjoint union of paths.
\end{proof}

We establish in Proposition~\ref{prop-OI-wqo} that the permutation class~$\O_I$ is wqo. It follows from part (a) of Proposition~\ref{prop-wqo-perms-graphs} that the graph class of linear forests is wqo, although there are a multitude of ways to see this latter fact.

At this point we know that paths are inversion graphs, while cycles of length five or more are not. It follows that the minimal forbidden induced subgraph characterisation of the class of inversion graphs contains $\{C_k : k\ge 5\}$ and thus is infinite. We remark that the entire infinite minimal forbidden induced subgraph characterisation of the class of inversion graphs was found in the seminal work of Gallai~\cite{gallai:transitiv-orien:}.

\subsection{Labelled well-quasi-order}
\label{subsec-lwqo}

We first define labelled well-quasi-order in the graph context. Let $(L,\mathord{\le_L})$ be any quasi-order (although we soon require that it be a wqo). An \emph{$L$-labeling} of the graph~$G$ is a mapping $\ell_G$ from the vertices of~$G$ to $L$, and the pair $(G,\ell_G)$ is called an \emph{{$L$-labelled} graph}. The {$L$-labelled} graph $(H,\ell_H)$ is a \emph{labelled induced subgraph} of the {$L$-labelled} graph $(G,\ell_G)$ if $H$ is isomorphic to an induced subgraph of~$G$ and this isomorphism maps each vertex $v\in H$ to a vertex $w\in G$ such that $\ell_H(v)\le_L \ell_G(w)$.

Given a class~$\C$ of graphs and a quasi-order $(L,\le_L)$, we denote by~$\C\wr L$ the set of {$L$-labelled} graphs from~$\C$. We say that~$\C$ is \emph{labelled well-quasi-ordered (lwqo)} if~$\C\wr L$ is wqo under the labelled induced subgraph order for \emph{every} wqo set $(L,\mathord{\le_L})$. Note that this is equivalent to saying that the set of graphs in~$\C$ labelled by $(L,\mathord{\le_L})$ does not contain an infinite antichain. It is worth emphasising that the definition of lwqo ranges over \emph{every} wqo set of labels $(L,\mathord{\le_L})$, and not merely finite sets of labels.

\begin{figure}
\begin{center}
	\begin{tikzpicture}[scale=0.5]
		\draw (0,0)--(3,0);
		\plotpartialperm{1/0,2/0};
		\plotpartialpermencirclewhite{0/0,3/0};
	\end{tikzpicture}
	\quad\quad
	\begin{tikzpicture}[scale=0.5]
		\draw (0,0)--(4,0);
		\plotpartialperm{1/0,2/0,3/0};
		\plotpartialpermencirclewhite{0/0,4/0};
	\end{tikzpicture}
	\quad\quad
	\begin{tikzpicture}[scale=0.5]
		\draw (0,0)--(5,0);
		\plotpartialperm{1/0,2/0,3/0,4/0};
		\plotpartialpermencirclewhite{0/0,5/0};
	\end{tikzpicture}
\end{center}
\caption{The set of all paths labelled by a two-element antichain forms an infinite antichain in the labelled induced subgraph order.}
\label{fig-double-ended-forks-lwqo}
\end{figure}
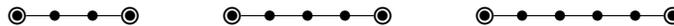

The lwqo property is much stronger than wqo. For one example, the set of all chordless paths, which is trivially wqo, is not lwqo, as indicated in Figure~\ref{fig-double-ended-forks-lwqo}. This shows that the class of linear forests is not lwqo, despite being wqo. On the other hand, as noted by Atminas and Lozin~\cite{atminas:labelled-induce:}, the class of cographs \emph{is} lwqo.

In the permutation context, we view an $L$-labeling of the permutation~$\pi$ of length $n$ as a mapping $\ell_\pi$ from the indices of~$\pi$ to elements of $L$, that is, $\ell_\pi : \{1,2,\dots,n\}\to L$. We think of~$\ell_\pi(i)$ as being the label attached to the entry $\pi(i)$, and we call the pair $(\pi,\ell_\pi)$ an \emph{{$L$-labelled} permutation}. Given two {$L$-labelled} permutations $(\pi,\ell_\pi)$ and $(\sigma,\ell_\sigma)$ where~$\pi$ and~$\sigma$ have lengths~$n$ and~$k$, respectively, we say that $(\sigma,\ell_\sigma)$ is contained in $(\pi,\ell_\pi)$ if there is an increasing sequence of~$k$~indices $1\le i_1<i_2<\cdots<i_k\le n$ such that the subsequence $\pi(i_1)\pi(i_2)\cdots\pi(i_k)$ is order isomorphic to~$\sigma$ and $\ell_\sigma(j) \le_L \ell_\pi(i_j)$ for all $1\le j\le k$.

As in the graph context, we let~$\C\wr L$ denote the set of {$L$-labelled} permutations of~$\C$ and say that the permutation class~$\C$ is \emph{labelled well-quasi-ordered (lwqo)} if~$\C\wr L$ is wqo (under the labelled containment order defined above) for \emph{every} wqo $(L,\mathord{\le_L})$.

Just as we associated inversion graphs to permutations in Section~\ref{subsec-inversion-graphs}, we can associate labelled inversion graphs to labelled permutations. Given an {$L$-labelled} permutation $(\pi,\ell_\pi)$, we define its \emph{{$L$-labelled} inversion graph} to be the pair $(G_\pi, \ell_\pi)$. Thus in the labelled graph $(G_\pi, \ell_\pi)$, the vertex~$i$ receives the same label as the entry $\pi(i)$ in the labelled permutation $(\pi,\ell_\pi)$. Applying Proposition~\ref{prop-wqo-order-preserving} in this context gives us the following analogue of Proposition~\ref{prop-wqo-perms-graphs}.

\begin{proposition}
\label{prop-lwqo-perms-graphs}
Let~$\C$ be a permutation class and $G_\C$ the corresponding graph class.
\begin{enumerate}
\item[(a)] If~$\C$ is lwqo in the permutation containment order, then $G_\C$ is lwqo in the induced subgraph order.
\item[(b)] Contrapositively, if $G_\C$ is \emph{not} lwqo in the induced subgraph order, then~$\C$ is not lwqo in the permutation containment order.
\end{enumerate}
\end{proposition}
\begin{proof}
Suppose the permutation class~$\C$ is lwqo and take $(L,\le_L)$ to be an arbitrary wqo set. Because~$\C$ is lwqo, the set of {$L$-labelled} members of~$\C$ is wqo. The mapping $(\pi,\ell_\pi)\mapsto (G_\pi,\ell_\pi)$ is easily seen to be order-preserving, and maps the {$L$-labelled} members of~$\C$ surjectively onto the {$L$-labelled} members of $G_\C$. Therefore Proposition~\ref{prop-wqo-order-preserving} shows that the {$L$-labelled} members of $G_\C$ are wqo, and since $(L,\le_L)$ was an arbitrary wqo set, this shows that $G_\C$ is lwqo.
\end{proof}

Analogous to Question~\ref{question-prop-wqo-perms-graphs-converse}, it is natural to ask: if $G_\C$ is lwqo in the induced subgraph order, is it necessarily true that~$\C$ is lwqo in the permutation containment order? We prove that the answer to this question is ``yes'' with Theorem~\ref{thm-C-lwqo-GC-lwqo}.

There are a variety of notions of structure that interpolate between wqo and lwqo. Specialising his notion to our context, Pouzet~\cite{pouzet:un-bel-ordre-da:} defined the class~$\C$ to be \emph{$n$-well-quasi-ordered} (\emph{$n$-wqo}) if the set of all permutations in~$\C$ labelled by an $n$-element antichain is wqo. The following result is trivial, but it arises in several of our discussions related to $n$-wqo, so we make it explicit here.

\begin{proposition}
\label{prop-n-wqo-other-poset}
If the permutation class~$\C$ is $n$-wqo and $(L,\le_L)$ is any $n$-element poset, then~$\C\wr L$ is wqo.
\end{proposition}
\begin{proof}
Let $(A,\le_A)$ be an $n$-element antichain. By the hypotheses, we know that~$\C\wr A$ is wqo. Let $\phi : A\to L$ denote any bijection between $A$ and $L$, so $\phi$ is an order-preserving surjection. It follows that the mapping $(\pi,\ell_\pi)\mapsto (\pi,\phi\circ\ell_\pi)$ from~$\C\wr A$ to~$\C\wr L$ is also an order-preserving surjection, so~$\C\wr L$ is wqo by Proposition~\ref{prop-wqo-order-preserving}.
\end{proof}

In his 1972 paper, Pouzet made the following still-open%
\footnote{\label{fn-kriz}It should also be noted that Conjecture~\ref{conj-pouzet-2-wqo} has an interpretation in category theory, as detailed by K\v{r}\'{i}\v{z} and Thomas~\cite{kriz:on-well-quasi-o:}, and in this more general context, the conjecture was shown to be false by K\v{r}\'{i}\v{z} and Sgall~\cite{kriz:well-quasiorder:}. In addition, a possible approach to proving Conjecture~\ref{conj-pouzet-2-wqo} has been outlined by Daligault, Rao, and Thomass\'e~\cite{daligault:well-quasi-orde:}, who conjectured that every $2$-wqo class of graphs has bounded clique-width (a term we do not discuss further here), and further asked if the same conclusion held for every wqo class of graphs. Lozin, Razgon, and Zamaraev~\cite{lozin:well-quasi-orde:}, however, answered the question negatively by constructing a wqo class of graphs with unbounded clique-width, although the original conjecture about $2$-wqo classes remains open.}
conjecture, that we have specialised to our context.

\begin{conjecture}[Cf. Pouzet~\cite{pouzet:un-bel-ordre-da:}]
\label{conj-pouzet-2-wqo}
A permutation class is $2$-wqo if and only if it is $n$-wqo for all $n\ge 1$.
\end{conjecture}

Since finite antichains are trivially wqo, it follows that an lwqo class is $n$-wqo for every $n$.
In fact, we are not aware of a class that is $2$-wqo that is not also known to be lwqo.
The following question was raised in the graph context~\cite{brignall:a-counterexampl:}, but we see no reason not to also ask it in the permutation context.

\begin{question}[Cf. Brignall, Engen, and Vatter~\cite{brignall:a-counterexampl:}]
\label{ques-pouzet-2-wqo-lwqo}
Is every $2$-wqo permutation class also lwqo?
\end{question}

While the $n$-wqo property has received scant attention in the permutation context, the labelled permutations that arise in its definition have been studied fairly extensively. Permutations whose entries are labelled by members of a finite antichain are precisely the same as the \emph{coloured permutations} that were considered by Mansour in a 2001 paper~\cite{mansour:pattern-avoidan:color} and also in numerous subsequent articles by Mansour and other authors. In the special case of $n=2$, the labelled permutations are typically identified with \emph{signed permutations} (or from the algebraic perspective, members of the hyperoctahedral group). The study of pattern avoidance in this context dates back a bit further, to the 2000 work of Simion~\cite{simion:combinatorial-s:B}.

\begin{figure}
\begin{center}
	\begin{tikzpicture}[scale=0.1925, baseline=(current bounding box.south)]
		\plotpermbox{0.5}{0.5}{16.5}{16.5};
		\draw (2,1)--(1,3)--(4,2)--(3,5)--(6,4)--(5,7)--(8,6)--(7,9)--(10,8)--(9,11)--(12,10)--(11,13)--(14,12)--(13,15)--(16,14)--(15,16);
		\plotperm{3,1,5,2,7,4,9,6,11,8,13,10,15,12,16,14};
		\plotpartialpermencirclewhite{2/1,15/16};
	\end{tikzpicture}
	\quad\quad
	\begin{tikzpicture}[scale=0.1925, baseline=(current bounding box.south)]
		\plotpermbox{0.5}{0.5}{16.5}{16.5};
		\draw (9,2)--(7,1)--(8,4)--(11,3)--(10,6)--(5,5)--(6,8)--(13,7)--(12,10)--(3,9)--(4,12)--(15,11)--(14,14)--(1,13)--(2,16)--(16,15);
		\plotperm{13,16,9,12,5,8,1,4,2,6,3,10,7,14,11,15};
		\plotpartialpermencirclewhite{9/2,16/15};
	\end{tikzpicture}
	\quad\quad
	\begin{tikzpicture}[scale=0.1925, baseline=(current bounding box.south)]
		\plotpermbox{0.5}{0.5}{16.5}{16.5};
		\draw (1,2)--(15,1)--(16,15)--(3,16)--(2,4)--(13,3)--(14,13)--(5,14)--(4,6)--(11,5)--(12,11)--(7,12)--(6,8)--(9,7)--(10,9)--(8,10);
		\plotperm{2,4,16,6,14,8,12,10,7,9,5,11,3,13,1,15};
		\plotpartialpermencirclewhite{1/2,8/10};
	\end{tikzpicture}
\end{center}
\caption{Typical members of the antichains of Figure~\ref{fig-three-antichains}, shown as labelled permutations. Note that, except in the leftmost picture, the lines between entries do not coincide with the edges of their inversion graphs.}
\label{fig-three-antichains-lwqo}
\end{figure}
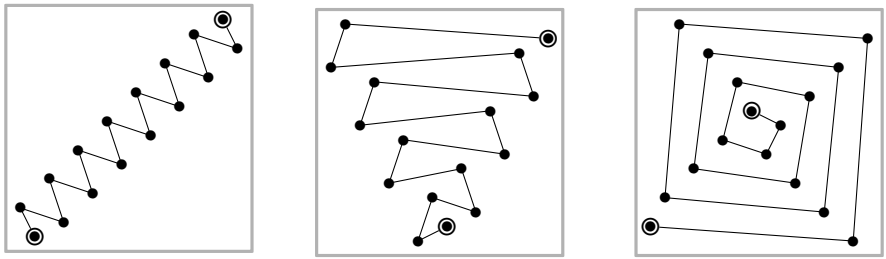

Finally, we remark that the lwqo property---and in fact also the notion of $2$-wqo---simplifies the intuition behind constructing infinite antichains as it allows us to replace ``anchors'' by labelled entries and focus instead on the more important task of constructing bodies, as shown in Figure~\ref{fig-three-antichains-lwqo}. Putting this intuition in the context of our comments above, practice seems to indicate that infinite antichains of signed permutations are in some sense more ``natural'' than infinite antichains of unsigned permutations.

\section{Finite Bases}
\label{sec-finite-bases}

In their seminal study of wqo permutation classes, Atkinson, Murphy, and Ru\v{s}kuc~\cite{atkinson:partially-well-:} define a permutation class to be \emph{strongly finitely based} if it and all of its subclasses are finitely based. Proposition~\ref{prop-lwqo-fin-basis} shows that all lwqo permutation classes are strongly finitely based. Before that, we give a characterisation of strongly-finitely-based classes.

\begin{proposition}[Cf.~Atkinson, Murphy, and Ru\v{s}kuc~{\cite[Proposition 1.1]{atkinson:partially-well-:}}]
\label{prop-strong-fin-based}
The permutation class~$\C$ is strongly finitely based if and only if it is both finitely based and wqo.
\end{proposition}
\begin{proof}
First assume that the class~$\C$ is finitely based and wqo and let $B$ denote its (finite) basis. The basis of any subclass $\D\subseteq\C$ consists of a subset of $B$ together with an antichain belonging to~$\C$. As $B$ is finite and~$\C$ is wqo, this basis must be finite.

Now suppose that the class~$\C$ and all of its subclasses are finitely based, let $B$ denote the basis of~$\C$, and suppose to the contrary that~$\C$ contains the infinite antichain $A$. Define $\D=\Av(A\cup B)$. Obviously, $\D$ is a subclass of~$\C$, and it is tempting to conclude that it is infinitely based, but we must be a bit careful. Indeed, the basis of $\D$ must be a subset of $A\cup B$, but it need not be all of $A\cup B$ because members of $A$ could be contained in members of $B$. Nevertheless, because $A\subseteq\C=\Av(B)$, it follows that no member of $A$ contains a member of $B$, and thus that every member of $A$ is contained in the basis of $\D$. Therefore $\D$ is in fact an infinitely-based subclass of~$\C$, and this contradiction completes our proof.
\end{proof}

Few general results have been established about strongly-finitely-based classes, but the result above allows us to show that the union of two strongly-finitely-based classes is itself strongly finitely based%
\footnote{\label{fn-intersection}The intersection of two strongly-finitely-based classes is also strongly finitely based, but this fact is a trivial consequence of the definition.}.
There are two ingredients to this proof: first, the fact that the union of two wqo classes is wqo, which is self-evident, and second, that the union of two-finitely-based classes is finitely based, which was first observed by Atkinson.

\begin{proposition}[Atkinson~{\cite[Theorem 2.1]{atkinson:restricted-perm:}}]
\label{prop-union-fb}
A finite union of finitely-based permutation classes is itself finitely based.
\end{proposition}
\begin{proof}
It suffices to prove that the union of two finitely-based classes, say~$\C$ and $\D$, is finitely based. Consider a basis element $\beta$ of~$\C\cup\D$. Since $\beta\notin\C,\D$, it must contain basis elements of both~$\C$ and $\D$; say these basis elements are $\gamma$ and $\delta$, respectively. By the minimality property of basis elements, no proper subpermutation of $\beta$ may contain both $\gamma$ and $\delta$, so $\beta$ must in fact be comprised entirely of a copy of $\gamma$ together with a copy of $\delta$, perhaps sharing some entries. Since~$\C$ and $\D$ are finitely based, there is a bound on the length of their basis elements, so there is a bound on the length of the basis elements of~$\C\cup\D$. This implies that~$\C\cup\D$ is finitely based, as desired.
\end{proof}

Returning to the context of wqo, we show below that $2$-wqo permutation classes are finitely based, which implies that they are strongly finitely based. While this is the first appearance of the permutation version of this result, we note that this version is a special case of a 1972 result of Pouzet~\cite{pouzet:un-bel-ordre-da:} (in French; see \cite[Theorem 3.1]{pouzet:applications-of:} for an English description of this result), who established the finite basis property for lwqo classes in the more general context of relational structures (multirelations, in French). A proof similar to that below has also been given in the context of graphs by Daligault, Rao, and Thomass\'e~\cite[Proposition 3]{daligault:well-quasi-orde:}. For lwqo classes, we strengthen this result later with Theorem~\ref{thm-lwqo-C-lwqo-C+1}.

\begin{proposition}[Cf.~Pouzet~\cite{pouzet:un-bel-ordre-da:}]
\label{prop-lwqo-fin-basis}
Every $2$-wqo (and thus in particular, every lwqo) permutation class is finitely based.
\end{proposition}
\begin{proof}
Suppose that the class~$\C$ with basis $B$ is 2-wqo and let $L=\{\circ,\bullet\}$ be a $2$-element antichain, so that~$\C\wr L$ is wqo.

For every basis element $\beta\in B$ we denote by $\beta^-$ the permutation obtained from $\beta$ by removing its rightmost entry. Label each entry of $\beta^-$ by $\circ$ if the corresponding entry of $\beta$ lies below the rightmost entry of $\beta$, or $\bullet$ otherwise. Let $\Psi(\beta)$ denote the resulting labelled permutation, as depicted in Figure~\ref{fig-prop-lwqo-fin-basis}. Since $\beta^-\in\C$ for all $\beta\in B$ by the definition of basis elements, we have that $\Psi(\beta)\in\C\wr L$, that is, $\Psi : B\to\C\wr L$.

\begin{figure}
\begin{footnotesize}
\begin{center}
	\begin{tikzpicture}[scale=0.1925, baseline=(current bounding box.center)]
	\plotpermbox{0.5}{0.5}{9.5}{9.5};
	\plotperm{2, 8, 7, 3, 6, 9, 1, 5, 4};
    \end{tikzpicture}
	\quad
	\begin{tikzpicture}[baseline=(current bounding box.center)]
		\node {$\to$};
	\end{tikzpicture}
	\quad
	\begin{tikzpicture}[scale=0.1925, baseline=(current bounding box.center)]
	\plotpermbox{0.5}{0.5}{9.5}{9.5};
	\draw [darkgray, thick] (0,4)--(10,4);
	\plotpartialpermhollow{1/2, 4/3, 7/1};
	\plotpartialperm{2/8, 3/7, 5/6, 6/9, 8/5, 9/4};
	\plotpartialpermencirclewhite{9/4};
    \end{tikzpicture}
	\quad
	\begin{tikzpicture}[baseline=(current bounding box.center)]
		\node {$\to$};
	\end{tikzpicture}
	\quad
	\begin{tikzpicture}[scale=0.1925, baseline=(current bounding box.center)]
	\plotpermbox{0.5}{0.5}{8.5}{8.5};
	\plotpartialpermhollow{1/2, 4/3, 7/1};
	\plotpartialperm{2/7, 3/6, 5/5, 6/8, 8/4};
    \end{tikzpicture}
\end{center}
\end{footnotesize}
\caption{The entry removal process described in the proof of Proposition~\ref{prop-lwqo-fin-basis}, applied to the potential basis element $287369154$.}
\label{fig-prop-lwqo-fin-basis}
\end{figure}
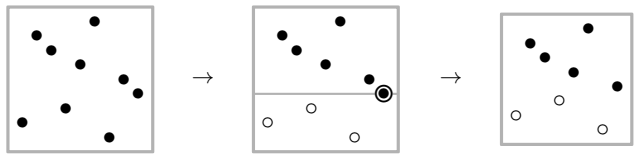

If we can show that $\Psi$ is order-reflecting, then, since~$\C\wr L$ is wqo, Proposition~\ref{prop-wqo-order-reflecting} will imply that $B$ is wqo. Because $B$ is an antichain by definition, this will imply that $B$ is finite and thus complete the proof. To this end, suppose that $\Psi(\beta)\le\Psi(\gamma)$ for $\beta,\gamma\in B$.

Let $k+1$ and $n+1$ denote the lengths of $\beta$ and $\gamma$, respectively, so there is an increasing sequence $1\le i_1<i_2<\cdots<i_k\le n$ of indices such that the subsequence $\gamma(i_1)\cdots\gamma(i_k)$ of $\gamma^-$ is order isomorphic to $\beta^-$ and such that for all $1\le j\le k$, the label of $\gamma(i_j)$ is the same as the label of $\beta(j)$. This condition ensures that the relative position of $\beta(k+1)$ amongst the entries of $\beta^-$ is the same as the relative position of $\gamma(n+1)$ amongst the entries of the subsequence $\gamma(i_1)\cdots\gamma(i_k)$, and thus the subsequence $\gamma(i_1)\cdots\gamma(i_k)\gamma(n+1)$ is order isomorphic to $\beta$, so we conclude that $\beta\le\gamma$ (and in fact $\beta=\gamma$, since both are basis elements), $\Psi$ is order-reflecting, and the proof is completed.
\end{proof}

Having established that lwqo classes are strongly finitely based, one might wonder what, if anything, differentiates these two properties. In fact, we have also observed that the family of strongly-finitely-based classes is closed under union (a consequence of Proposition~\ref{prop-union-fb}) and intersection (see Footnote~\ref{fn-intersection}). That the family of lwqo classes is closed under union and intersection follows trivially from the definition. Thus the reader could be forgiven for thinking that these properties might coincide, but this is not the case.

\begin{proposition}
\label{prop-strong-fb-not-lwqo-perms}
There is a permutation class that is both finitely based and wqo, but not lwqo.
\end{proposition}

Indeed, the proof of Proposition~\ref{prop-strong-fb-not-lwqo-perms} has mostly been completed in the previous section. Recall that the downward closure of the set of increasing oscillations is the class~$\O_I$, and that this class is finitely based by Proposition~\ref{prop-OI-fb}. Moreover, it is not hard to see that~$\O_I$ is also wqo, a fact we formally establish with Proposition~\ref{prop-OI-wqo}. But~$\O_I$ is \emph{not} lwqo, as one may label the increasing oscillations as indicated on the left of Figure~\ref{fig-three-antichains-lwqo}, which shows that~$\O_I$ is not even~$2$-wqo.

Thus lwqo is strictly stronger than the property of being strongly finitely based in the permutation context, but we could still wonder if the two properties might coincide in the graph context. In this direction, note that the proofs of Propositions~\ref{prop-strong-fin-based} and~\ref{prop-lwqo-fin-basis} can readily be adapted to show that a graph class is strongly finitely based if and only if it is finitely based and wqo, and that every 2-wqo graph class is strongly finitely based.

Korpelainen, Lozin, and Razgon~\cite{korpelainen:boundary-proper:} conjectured%
\footnote{This conjecture was later restated by Atminas and Lozin~\cite{atminas:labelled-induce:} in their work on lwqo.}
that, contrary to Proposition~\ref{prop-strong-fb-not-lwqo-perms} for permutations, the converse of Proposition~\ref{prop-lwqo-fin-basis} holds for graphs. The counterexample we have given for permutations does not translate to a counterexample for graphs because the analogous class of graphs (the linear forests) is not finitely based (its basis contains all cycles). However, there is another permutational counterexample that \emph{does} translate to a graphical counterexample.

\begin{proposition}[Brignall, Engen, and Vatter~\cite{brignall:a-counterexampl:}]
\label{prop-strong-fb-not-lwqo-graphs}
There is a permutation class~$\C$ such that the corresponding graph class $G_{\C}$ is wqo and defined by finitely many forbidden induced subgraphs, but is not lwqo.
\end{proposition}

\begin{figure}
\begin{footnotesize}
\begin{center}
  \begin{tikzpicture}[scale=0.1925, baseline=(current bounding box.center)]
    \draw[thick, darkgray, ->] (2,-2) to[curve through={(1,2) (-2,1) (-1,-5) (4,-4) (3,4) (-4,3) (-3,-7) (6,-6) (5,6) (-6,5) (-5,-9) (8,-8) (7,8) (-8,7) (-7,-11)}] (5,-12.5);
    \plotpartialperm{2/-2, 1/2, -2/1, -1/-5, 4/-4, 3/4, -4/3, -3/-7, 6/-6, 5/6, -6/5, -5/-9, 8/-8, 7/8, -8/7, -7/-11}
  \end{tikzpicture}
\quad\quad\quad\quad
  \begin{tikzpicture}[scale=0.1925, baseline=(current bounding box.center)]
	\draw [darkgray, very thick, rounded corners=0.01, line cap=round]
		(-8.5,-0.5)--(8.5,-0.5);
	\draw [darkgray, very thick, rounded corners=0.01, line cap=round]
		(0.5,-11.5)--(0.5,8.5);
    \draw (1,2)--(1,-2);
	\draw (-2,1)--(2,1);
	\draw (-1,-5)--(-1,2);
	\draw (4,-4)--(-2,-4);
	\draw (3,4)--(3,-5);
	\draw (-4,3)--(4,3);
	\draw (-3,-7)--(-3,4);
	\draw (6,-6)--(-4,-6);
	\draw (5,6)--(5,-7);
	\draw (-6,5)--(6,5);
	\draw (-5,-9)--(-5,6);
	\draw (8,-8)--(-6,-8);
	\draw (7,8)--(7,-9);
	\draw (-8,7)--(8,7);
	\draw (-7,-11)--(-7,8);
    \plotpartialperm{1/2, -2/1, -1/-5, 4/-4, 3/4, -4/3, -3/-7, 6/-6, 5/6, -6/5, -5/-9, 8/-8, 7/8, -8/7, 2/-2, -7/-11};
    \plotpartialpermhollow{2/-2, -7/-11};
  \end{tikzpicture}
\end{center}
\end{footnotesize}
\caption{Two additional depictions of a widdershins spiral, first shown on the right of Figure~\ref{fig-three-antichains-lwqo}.}
\label{fig-widdershins}
\end{figure}
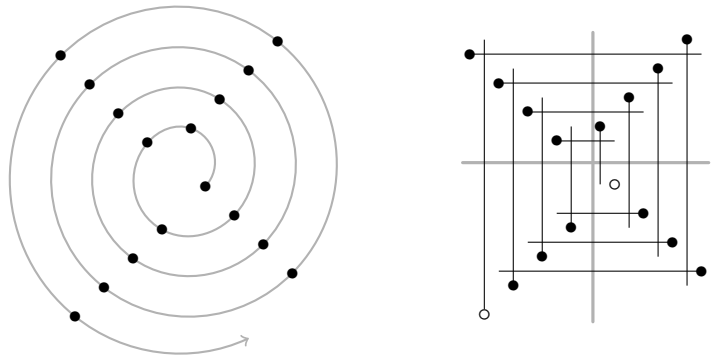

The construction used in~\cite{brignall:a-counterexampl:} to establish Proposition~\ref{prop-strong-fb-not-lwqo-graphs} is based on a set of permutations named \emph{widdershins spirals}%
\footnote{Widdershins is a Lower Scots word meaning ``to go anti-clockwise''.}
by Murphy in his thesis~\cite[Section 3.2]{murphy:restricted-perm:}. The clockwise symmetries of Widdershins spirals have already appeared in Figures~\ref{fig-three-antichains} and \ref{fig-three-antichains-lwqo}, and two additional drawings are shown in Figure~\ref{fig-widdershins}. As with the class~$\O_I$, it can readily be shown---by labeling the widdershins spirals as on the right of Figure~\ref{fig-widdershins}---that the downward closure of this set of permutations is not $2$-wqo (and thus is not lwqo). A bit more analysis shows that the same holds for the corresponding graph class. Further analysis given in~\cite{brignall:a-counterexampl:} then establishes that the downward closure of this permutation class is wqo and finitely based\footnote{The basis consists of $2143$, $2413$, $3412$, $314562$, $412563$, $415632$, $431562$, $512364$, $512643$, $516432$, $541263$, $541632$, and $543162$.}. The final step to establish Proposition~\ref{prop-strong-fb-not-lwqo-graphs} is to show that the corresponding class of graphs is finitely based.

\section{Finite Sets and Downward Closures}

We begin with the simplest possible permutation classes, which are rather trivially lwqo.

\begin{proposition}
\label{prop-finite-lwqo}
Every finite set of permutations is lwqo.
\end{proposition}
\begin{proof}
Suppose that~$\C$ is a finite set of permutations and $(L,\mathord{\le_L})$ is wqo. If $\pi\in\C$ has length $m$, then the poset of labelings of~$\pi$ is isomorphic to $(L^m,\mathord{\le_L})$, which is wqo by Proposition~\ref{prop-wqo-vector} (Dickson's lemma). It follows that~$\C\wr L$ can be expressed as the union of finitely many wqo posets, and thus~$\C\wr L$ is itself wqo, as required.
\end{proof}

Next we establish that downward closures of lwqo sets (not necessarily classes) of permutations are also lwqo. As we elaborate on after the proof, this is one of the more striking differences between wqo and lwqo.

\begin{theorem}
\label{thm-lwqo-downward-closure}
If the set~$X$ of permutations is lwqo, then its \emph{downward closure},
\[
	X^{\suplessthaneq}
	=
	\{\sigma : \text{$\sigma \le \pi$ for some $\pi\in X$}\},
\]
is also lwqo.
\end{theorem}
\begin{proof}
Suppose that~$X$ is an lwqo set of permutations and take $(L,\le_L)$ to be an arbitrary wqo set. We want to show that $X^{\suplessthaneq}\wr L$ is wqo. We begin by adjoining to~$L$ a new minimum element~${0\notin L}$ to form the poset $(L_0,\le_0)$, that is trivially wqo. Because~$X$ is lwqo and $(L_0,\le_0)$ is wqo, we know that $X\wr L_0$ is wqo. We prove the proposition by constructing an order-preserving surjection~${\Phi : X\wr L_0 \tosurj X^{\suplessthaneq}\wr L}$ and then appealing to Proposition~\ref{prop-wqo-order-preserving}.

Let $(\pi,\ell_\pi)\in X\wr L_0$ be arbitrary. We define the {$L$-labelled} permutation $\Phi((\pi,\ell_\pi))\in X^{\suplessthaneq}\wr L$ by deleting all entries of~$\pi$ labelled by $0$, keeping the remaining entries together with their labels, and then reducing these entries to obtain a labelled permutation in $X^{\suplessthaneq}\wr L$.

We first verify that $\Phi$ is surjective. Suppose $(\sigma,\ell_\sigma)\in X^{\suplessthaneq}\wr L$ where~$\sigma$ has length $k$. Because~${\sigma\in X^{\suplessthaneq}}$, there is some permutation $\pi\in X$ with $\sigma\le\pi$. Fix such a permutation~$\pi$, let $n$ denote its length, and fix an embedding of~$\sigma$ into~$\pi$. Supposing that this embedding is given by the indices $1\le i_1<\cdots<i_k\le n$, we define the $L_0$-labeling $\ell_\pi$ of~$\pi$ by
\[
	\ell_\pi(i)=
	\begin{cases}
	\ell_\sigma(j)&\text{if $i=i_j$,}\\
	0&\text{otherwise.}
	\end{cases}
\]
Because $\Phi$ maps $(\pi,\ell_\pi)$ to $(\sigma,\ell_\sigma)$, we see that $\Phi$ is indeed surjective.

Next we must verify that $\Phi$ is order-preserving. Let $(\tau,\ell_\tau)$ and $(\pi,\ell_\pi)$ be $L_0$-labelled members of~$X$ of lengths $k$ and $n$ respectively such that $(\tau,\ell_\tau)$ is contained in $(\pi,\ell_\pi)$. Further suppose that this containment is given by the indices $1\le i_1<\cdots<i_k\le n$. We know that $\ell_\tau(j)\le\ell_\pi(i_j)$ for all indices $1\le j\le k$, so since $0$ is minimal in $L_0$ we know that $\ell_\tau(j)=0$ whenever $\ell_\pi(i_j)=0$. Therefore, whenever $\Phi$ deletes an entry from $(\pi,\ell_\pi)$, either that entry does not participate in this embedding of $(\tau,\ell_\tau)$ into $(\pi,\ell_\pi)$, or $\Phi$ also deletes the corresponding entry from $(\tau,\ell_\tau)$. As $\Phi$ does not change the labels of the remaining entries, we see that $\Phi((\tau,\ell_\tau))$ is contained in $\Phi((\pi,\ell_\pi))$. This verifies that $\Phi$ is order-preserving and completes the proof.
\end{proof}

The wqo analogue of Theorem~\ref{thm-lwqo-downward-closure} fails catastrophically: take $A=\{\alpha_1,\alpha_2,\dots\}$ to be an infinite antichain of permutations and set $X=\{\alpha_1\oplus\cdots\oplus\alpha_k : k\ge 1\}$. Then~$X$ is a chain, so it is wqo, but its downward closure contains the infinite antichain $A$. 
There are also an abundance of counterexamples in more general settings, for example, if we take $X=\{0\}$ in the poset $(\mathbb{Z},\le)$, then $X^{\suplessthaneq}=\{0,-1,-2,\dots\}$ contains an infinite strictly decreasing sequence and thus is not wqo. 

Thus, that Theorem~\ref{thm-lwqo-downward-closure} holds is remarkable, and emphasises the strength of the lwqo property. Indeed, this result is not peculiar to permutations: for example, in the case of graphs one simply needs to fix an embedding of a graph $H$ in the downset as an induced subgraph of a graph~$G$ from the set itself, and adopt a ``zero label'' to mark those vertices of~$G$ that are not included in the embedding of $H$.

\section{One-Point Extensions}
\label{sec-one-pt-extensions}

If $\beta$ is a basis element of a permutation class~$\C$, then the removal of any entry of $\beta$ yields a member of~$\C$. Here we consider the inverse operation, where we add entries to members of a class in all possible ways. As will be demonstrated, this study sheds further light on bases of permutation classes and allows us to further explore the relationship between wqo and lwqo.

Given a permutation class~$\C$ and an integer $t\ge 0$, we let~$\C^{+t}$ denote the set of all permutations~$\pi$ for which there exists a collection of $t$ or fewer entries such that by removing these entries from~$\pi$ and taking the reduction we obtain a member of~$\C$. Note that~$\C^{+t}$ is itself a permutation class. Moreover, since
\[
	\C^{+t}=\C^{+1+1+\cdots+1},
\]
for most purposes it suffices to consider classes of the form~$\C^{+1}$. We call members of~$\C^{+1}$ the \emph{one-point extensions} of~$\C$ (although~$\C\subseteq\C^{+1}$, so some members of~$\C^{+1}$ are also members of~$\C$). Thus the nonempty permutation~$\pi$ lies in~$\C^{+1}$ if and only if there is an entry $\pi(a)$ of~$\pi$ such that $\pi-\pi(a)\in\C$, where we define $\pi-\pi(a)$ to be the result of deleting the entry $\pi(a)$ from~$\pi$ and then reducing the remaining entries to obtain a permutation.

Every basis element of the class~$\C$ is necessarily a one-point extension of~$\C$, because the removal of \emph{any} entry of a basis element of~$\C$ yields a permutation in~$\C$. Thus if~$\C^{+1}$ is wqo, then~$\C$ is finitely based. In fact, since~$\C\subseteq\C^{+1}$, if~$\C^{+1}$ is wqo, then~$\C$ is strongly finitely based, a conclusion we record below.

\begin{proposition}
\label{prop-C+1-wqo-implies-C-sfb}
If the permutation class~$\C^{+1}$ is wqo, then the class~$\C$ is both finitely based and wqo (that is, strongly finitely based).
\end{proposition}

The converse of Proposition~\ref{prop-C+1-wqo-implies-C-sfb} is not true, and one counterexample is the class~$\O_I$. This class is finitely based by Proposition~\ref{prop-OI-fb} and is wqo by our upcoming Proposition~\ref{prop-OI-wqo}, but $\O_I^{+1}$ is not wqo, as demonstrated by the infinite antichain depicted in the centre of Figure~\ref{fig-three-inc-osc-antichains}. However, it is true that if the class~$\C$ is finitely based, then the class~$\C^{+1}$ is also finitely based.

\begin{proposition}[Atkinson and Beals~{\cite[Lemma 7]{atkinson:permuting-mecha:}}]
\label{prop-C-fb-C+1-fb}
If the permutation class~$\C$ is finitely based then the class~$\C^{+1}$ is also finitely based.
\end{proposition}
\begin{proof}
Suppose that the longest basis element of~$\C$ has length $m$ and suppose that $\gamma\not\in\C^{+1}$. Because $\gamma\not\in\C$, it contains a subsequence, say $\gamma(i_1)\cdots\gamma(i_m)$, that is not order isomorphic to a member of~$\C$. Furthermore, since $\gamma\not\in\C^{+1}$, the permutation obtained from $\gamma$ by removing the entry $\gamma(i_k)$ for any $1\le k\le m$ also contains a basis element of~$\C$ of length at most $m$. By considering the $m$ entries $\gamma(i_1)\cdots\gamma(i_m)$ together with the at most $m^2$ entries of $\gamma$ arising from these additional basis elements of~$\C$, we see that $\gamma$ contains a permutation of length at most $m+m^2$ that does not lie in~$\C^{+1}$, proving that no basis element of~$\C^{+1}$ may be longer than this.
\end{proof}

Combining Propositions~\ref{prop-C+1-wqo-implies-C-sfb} and \ref{prop-C-fb-C+1-fb}, we immediately obtain the following result.

\begin{proposition}
\label{prop-C+1-wqo-implies-sfb}
If the permutation class~$\C^{+1}$ is wqo, then it is strongly finitely based.
\end{proposition}

We are not aware of a counterexample to the converse of Proposition~\ref{prop-C-fb-C+1-fb}, and so we raise the following question.

\begin{question}
\label{quest-C+1-fb-C-fb}
Is there an infinitely-based permutation (or graph) class~$\C$ such that~$\C^{+1}$ is finitely based?
\end{question}

Finally, we connect one-point extensions and lwqo, showing that if the class~$\C$ is lwqo then the class~$\C^{+1}$ is also lwqo. In particular, this implies via Proposition~\ref{prop-C+1-wqo-implies-C-sfb} that every lwqo class is finitely based. Thus this result strengthens Proposition~\ref{prop-lwqo-fin-basis} for lwqo classes.

\newcommand{\nelabel}{
	\begin{tikzpicture}[scale=0.09, anchor=base]
		\draw[line cap=round] (0,0)--(0.95,1.9);
		\draw[line cap=round, ultra thin, fill=black] (1,2)--({1+0.75*cos(270+18.4349488)}, {2+0.75*sin(270+18.4349488)})--({1+0.75*cos(180+18.4349488)}, {2+0.75*sin(180+18.4349488)})--cycle;
 	\end{tikzpicture}
}
\newcommand{\selabel}{
	\begin{tikzpicture}[scale=0.09, yscale=-1, anchor=base]
		\draw[line cap=round] (0,0)--(0.95,1.9);
		\draw[line cap=round, ultra thin, fill=black] (1,2)--({1+0.75*cos(270+18.4349488)}, {2+0.75*sin(270+18.4349488)})--({1+0.75*cos(180+18.4349488)}, {2+0.75*sin(180+18.4349488)})--cycle;
 	\end{tikzpicture}
}
\newcommand{\nwlabel}{
	\begin{tikzpicture}[scale=0.09, xscale=-1, anchor=base]
		\draw[line cap=round] (0,0)--(0.95,1.9);
		\draw[line cap=round, ultra thin, fill=black] (1,2)--({1+0.75*cos(270+18.4349488)}, {2+0.75*sin(270+18.4349488)})--({1+0.75*cos(180+18.4349488)}, {2+0.75*sin(180+18.4349488)})--cycle;
 	\end{tikzpicture}
}
\newcommand{\swlabel}{
	\begin{tikzpicture}[scale=0.09, xscale=-1, yscale=-1, anchor=base]
		\draw[line cap=round] (0,0)--(0.95,1.9);
		\draw[line cap=round, ultra thin, fill=black] (1,2)--({1+0.75*cos(270+18.4349488)}, {2+0.75*sin(270+18.4349488)})--({1+0.75*cos(180+18.4349488)}, {2+0.75*sin(180+18.4349488)})--cycle;
 	\end{tikzpicture}
}

\newcommand{\compass}{
	\begin{tikzpicture}[scale=0.0775, anchor=base]
		\draw[ultra thin, fill=black] (60:0.75)--(150:0.4)--(240:0.75)--(330:0.4)--cycle;
		\draw[thin, color=white] (150:0.5)--(330:0.5);
		\draw[color=white, fill=white] (0,0) circle (0.125);
		\draw (0,0) circle (1);
	\end{tikzpicture}
}
\newcommand{\altcompass}{
	\begin{tikzpicture}[scale=0.0775, anchor=base]
		\draw[ultra thin, fill=black] (60:1)--(150:0.5)--(240:1)--(330:0.5)--cycle;
		\draw[thin, color=white] (150:0.55)--(330:0.55);
		\draw[color=white, fill=white] (0,0) circle (0.2);
	\end{tikzpicture}
}

\begin{theorem}[Cf. Oudrar~{\cite[Proposition~5.32]{oudrar:sur-lenumeratio:}}]
\label{thm-lwqo-C-lwqo-C+1}
The permutation class~$\C$ is lwqo if and only the class~$\C^{+1}$ is lwqo.
\end{theorem}
\begin{proof}
One direction is clear, so suppose the class~$\C$ is lwqo and take $(L,\le_L)$ to be an arbitrary wqo set. We want to show that~$\C^{+1}\wr L$ is wqo. To this end, let $\{\swlabel,\selabel,\nelabel,\nwlabel\}$ denote a $4$-element antichain disjoint from $L$ and define the poset $L_{\compass}$ as the Cartesian product $L\times L\times\{\swlabel,\selabel,\nelabel,\nwlabel\}$. Because $L_{\compass}$ is wqo by Proposition~\ref{prop-wqo-product} and~$\C$ is lwqo by the hypotheses,~$\C\wr L_{\compass}$ is wqo. We prove the proposition by constructing an order-reflecting mapping $\Psi : \C^{+1}\wr L\to \C\wr L_{\compass}$ and then appealing to Proposition~\ref{prop-wqo-order-reflecting}. In fact, we restrict our domain to the members of~$\C^{+1}\wr L$ of length at least two, since if these labelled permutations are wqo then it follows that~$\C^{+1}$ is lwqo.

\begin{figure}
\begin{footnotesize}
\begin{center}
	\begin{tikzpicture}[scale=0.1925, baseline=(current bounding box.center)]
	\plotpermbox{0.5}{0.5}{9.5}{9.5};
	\plotperm{5, 7, 1, 8, 3, 4, 6, 9, 2};
	\plotpartialpermencirclewhite{6/4};
    \end{tikzpicture}
	\quad
	\begin{tikzpicture}[baseline=(current bounding box.center)]
		\node {$\to$};
	\end{tikzpicture}
	\quad
	\begin{tikzpicture}[scale=0.1925, baseline=(current bounding box.center)]
	\plotpermbox{0.5}{0.5}{9.5}{9.5};
	\draw [darkgray, thick] (0,4)--(10,4);
	\draw [darkgray, thick] (6,0)--(6,10);
	\plotpartialperm{6/4};
	\plotpartialpermencirclewhite{6/4};
	\node at (7,6) {\scriptsize $\swlabel$};
	\node at (8,9) {\scriptsize $\swlabel$};
	\node at (1,5) {\scriptsize $\selabel$};
	\node at (2,7) {\scriptsize $\selabel$};
	\node at (4,8) {\scriptsize $\selabel$};
	\node at (3,1) {\scriptsize $\nelabel$};
	\node at (5,3) {\scriptsize $\nelabel$};
	\node at (9,2) {\scriptsize $\nwlabel$};
    \end{tikzpicture}
	\quad
	\begin{tikzpicture}[baseline=(current bounding box.center)]
		\node {$\to$};
	\end{tikzpicture}
	\quad
	\begin{tikzpicture}[scale=0.1925, baseline=(current bounding box.center)]
	\plotpermbox{0.5}{0.5}{8.5}{8.5};
	\node at (6,5) {\scriptsize $\swlabel$};
	\node at (7,8) {\scriptsize $\swlabel$};
	\node at (1,4) {\scriptsize $\selabel$};
	\node at (2,6) {\scriptsize $\selabel$};
	\node at (4,7) {\scriptsize $\selabel$};
	\node at (3,1) {\scriptsize $\nelabel$};
	\node at (5,3) {\scriptsize $\nelabel$};
	\node at (8,2) {\scriptsize $\nwlabel$};
    \end{tikzpicture}
\end{center}
\end{footnotesize}
\caption{A depiction of the labeling process described in the proof of Theorem~\ref{thm-lwqo-C-lwqo-C+1}, applied to the permutation $\pi=571834692$.}
\label{fig-prop-lwqo-C-lwqo-C+1}
\end{figure}
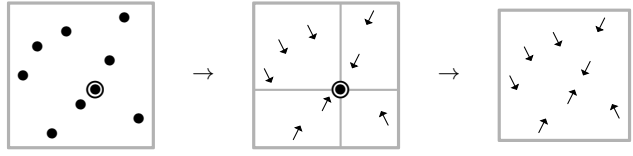

Take $(\pi,\ell_\pi)\in\C^{+1}\wr L$ where~$\pi$ has length $n+1\ge 2$, and choose some entry $\pi(a)$ such that $\pi-\pi(a)\in\C$. We define $\Psi((\pi,\ell_\pi))$ to be the permutation $(\pi^a,\ell^a_{\pi})\in\C\wr L_{\compass}$ where ${\pi^a=\pi-\pi(a)\in\C}$ and the labeling $\ell^a_{\pi} : \{1,2,\dots,n\}\to L_{\compass}$ is defined by
\[
	\ell^a_{\pi}(i)=
	\begin{cases}
	(\ell_\pi(i),\ell_\pi(a),\selabel)&\text{if $i<a$ and $\pi(i)>\pi(a)$,}\\
	(\ell_\pi(i),\ell_\pi(a),\nelabel)&\text{if $i<a$ and $\pi(i)<\pi(a)$,}\\
	(\ell_\pi(i+1),\ell_\pi(a),\swlabel)&\text{if $i\ge a$ and $\pi(i+1)>\pi(a)$,}\\
	(\ell_\pi(i+1),\ell_\pi(a),\nwlabel)&\text{if $i\ge a$ and $\pi(i+1)<\pi(a)$.}\\
	\end{cases}
\]
Intuitively, as indicated in Figure~\ref{fig-prop-lwqo-C-lwqo-C+1}, the labeling $\ell^a_{\pi}$ retains the labeling of all entries of~$\pi$---entries other than $\pi(a)$ keep their labeling in $\ell_\pi^a$, while the label of $\pi(a)$ is recorded in the labelings of \emph{all} remaining entries---while also recording the position of $\pi(a)$ relative to all of the remaining entries of~$\pi$. It follows that $\Psi$ is injective: from $\Psi((\pi,\ell_\pi))$ we can recover~$\pi$ and $\ell_\pi$.

In order to show that $\Psi$ is order-reflecting, suppose that there are two members of~$\C^{+1}\wr L$, say $(\pi,\ell_\pi)$ and $(\sigma,\ell_\sigma)$, such that $\Psi((\sigma,\ell_\sigma))=(\sigma^a,\ell^a_\sigma)$ is contained in $\Psi((\pi,\ell_\pi))=(\pi^b,\ell^b_\pi)$ in the order on~$\C\wr L_{\compass}$. Letting $k+1$ and $n+1$ denote the lengths of~$\sigma$ and~$\pi$, respectively, this means that there is an increasing sequence $1\le i_1<i_2<\cdots<i_k\le n$ of indices such that the subsequence $\pi^b(i_1)\pi^b(i_2)\cdots\pi^b(i_k)$ is order isomorphic to $\sigma^a$ and that $\ell^a_\sigma(j)\le_{L_{\compass}} \ell^b_\pi(i_j)$ for all~${1\le j\le k}$.

First, note that $\ell^a_\sigma$ must have precisely $a-1$ labels whose third components are $\selabel$ or $\nelabel$ because there are precisely $a-1$ entries to the left of $\sigma(a)$ in~$\sigma$. It follows that precisely $a-1$ of the labels $\ell^b_\pi(i_j)$ have third component equal to $\selabel$ or $\nelabel$, so the entry $\pi(b)$ lies between the entries of~$\pi$ corresponding to~$\pi^b(i_{a-1})$ and $\pi^b(i_a)$ in~$\pi$, that is, $i_{a-1}<b<i_a+1$. In other words, the horizontal position of $\pi(b)$ amongst the entries of~$\pi$ corresponding to the subsequence $\pi^b(i_1)\pi^b(i_2)\cdots\pi^b(i_k)$ is the same as the horizontal position of $\sigma(a)$ amongst the other entries of~$\sigma$. By counting labels whose third component is equal to either $\nelabel$ or $\nwlabel$, the same claim holds for the vertical positions of these two entries. This shows that~$\sigma$ is contained in~$\pi$ in the positions $1\le i_1<\cdots<i_{a-1}<b<i_a+1<\cdots<i_k+1\le n+1$. Moreover, the labels of the corresponding entries are comparable as desired: $\ell_\sigma(a)\le\ell_\pi(b)$ because these labels are encoded in the second components of all the labels of $\ell^a_\sigma$ and $\ell^b_\pi$, and the other labels have the desired comparisons because of the first components of $\ell^a_\sigma$ and $\ell^b_\pi$. This shows that $\Psi$ is order-reflecting and completes the proof.
\end{proof}

Note that Theorem~\ref{thm-lwqo-C-lwqo-C+1} does not hold if lwqo is replaced by wqo, as demonstrated by the example of~$\O_I$ and $\O_I^{+1}$.

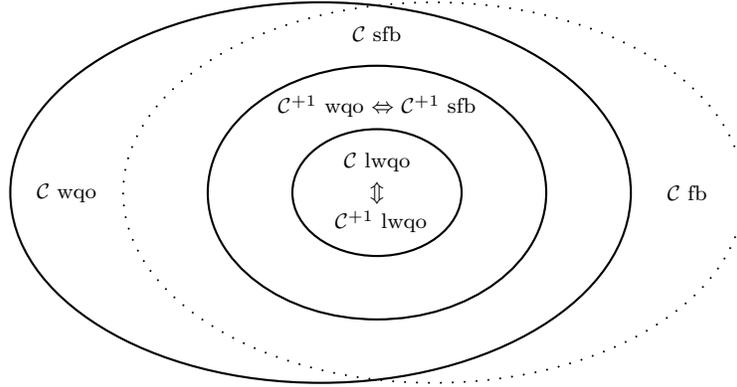
\begin{figure}
\begin{footnotesize}
\begin{center}
	\begin{tikzpicture}[scale=0.75, yscale=0.75, baseline=(current bounding box.center)]
	\draw [line width=0.8pt, dash pattern=on 0pt off 4.5pt, line cap=round] (1,0) ellipse (5.5 and 4.5);
	\draw [thick] (-1,0) ellipse (5.5 and 4.5);
	\draw [thick] (0,0) circle (1.5);
	\draw [thick] (0,0) circle (3);
	\node at (0,0) [align=center] {$\C$ lwqo\\[2pt] $\Updownarrow$\\[2pt]~$\C^{+1}$ lwqo};
	\node at (0,2.05) {$\C^{+1}$ wqo $\Leftrightarrow$~$\C^{+1}$ sfb};
	\node at (0,3.75) {$\C$ sfb};
	\node at (5.5,0) {$\C$ fb};
	\node at (-5.5,0) {$\C$ wqo};
    \end{tikzpicture}
\end{center}
\end{footnotesize}
\caption{Properties of~$\C$ and~$\C^{+1}$. Finitely and strongly finitely based are abbreviated as fb and sfb, respectively. Not pictured is the property that~$\C^{+1}$ is finitely based.}
\label{fig-inclusions-C-C+1}
\end{figure}

Figure~\ref{fig-inclusions-C-C+1} displays the inclusions between properties of~$\C$ and~$\C^{+1}$. We have already seen examples showing that all of the inclusions in this diagram are strict except, possibly, the inclusion showing that if~$\C$ is lwqo (or equivalently, by Theorem~\ref{thm-lwqo-C-lwqo-C+1}, if~$\C^{+1}$ is lwqo) then~$\C^{+1}$ is wqo (and thus by Proposition~\ref{prop-C+1-wqo-implies-sfb}, also strongly finitely based). In fact, we are not aware of a permutation class (or a graph class for that matter)~$\C$ for which~$\C^{+1}$ is wqo but~$\C$ is not lwqo.

\begin{conjecture}
\label{conj-C+1-wqo-C-lwqo}
If the permutation class~$\C^{+1}$ is wqo, then~$\C$, and thus also~$\C^{+1}$, is lwqo.
\end{conjecture}

Theorem~\ref{thm-lwqo-C-lwqo-C+1} implies that if~$\C$ is lwqo then~$\C^{+t}$ is lwqo for all integers $t\ge 0$. Perhaps more interestingly, with only minor modifications, the proof given shows that if~$\C$ is a $4$-wqo permutation class, then~$\C^{+1}$ is wqo. Conjecture~\ref{conj-pouzet-2-wqo} would imply that every $2$-wqo class is $4$-wqo, and thus in particular would imply the following. (Which would also be implied by the stronger Conjecture~\ref{ques-pouzet-2-wqo-lwqo}.)

\begin{conjecture}
\label{conj-2-wqo-C-lwqo-C+1}
If the permutation class~$\C$ is $2$-wqo, then the class~$\C^{+1}$ is wqo.
\end{conjecture}

The graphical analogue of Conjecture~\ref{conj-2-wqo-C-lwqo-C+1} is true, and is essentially equivalent to Proposition~\ref{prop-lwqo-fin-basis}. Note that a one-vertex extension of an inversion graph need not be an inversion graph---extending our $+1$ notation to graph classes, we always have $G_{\C^{+1}}\subseteq G_\C^{+1}$, but this inclusion is usually strict%
\footnote{The only reason the inclusion $G_{\C^{+1}}\subseteq G_\C^{+1}$ is not \emph{always} strict is because of trivial cases such as~$\C=\emptyset$.}.
Therefore, this discussion does not relate to the relationship of wqo between permutations and inversion graphs, as asked in Question~\ref{question-prop-wqo-perms-graphs-converse}.

Returning to the permutation context, the key difference between the proofs of Proposition~\ref{prop-lwqo-fin-basis} and Theorem~\ref{thm-lwqo-C-lwqo-C+1} is that in the proof of Proposition~\ref{prop-lwqo-fin-basis} we are allowed to delete the \emph{last} entry of a permutation, while in the proof of Theorem~\ref{thm-lwqo-C-lwqo-C+1} we must be prepared to delete \emph{any} of its entries.

Continuing in this direction, our proof of Theorem~\ref{thm-lwqo-C-lwqo-C+1} shows that if~$\C$ is $4n^2$-wqo then~$\C^{+1}$ is $n$-wqo. If Conjecture~\ref{conj-pouzet-2-wqo} (about the equivalence of $2$-wqo and $n$-wqo for all $n\ge 1$) holds, then the conjecture below would also be true.

\begin{conjecture}
\label{conj-2-wqo-C-lwqo-C+t}
If the permutation class~$\C$ is $2$-wqo, then the class~$\C^{+t}$ is $2$-wqo for every $t\ge 0$.
\end{conjecture}

\section{Minimal Bad Sequences}
\label{sec-min-bad-seq}

We need a bit more machinery to prove further results about lwqo in the three sections after this. In particular, we make significant use of the notion of minimal bad sequences. These were first introduced in Nash-Williams's incredibly influential 3-page paper~\cite{nash-williams:on-well-quasi-o:infinite} in which he used them to give an elegant proof of Kruskal's tree theorem~\cite{kruskal:well-quasi-orde:}%
\footnote{While it is well outside the scope of this paper, it is nevertheless an interesting fact that the existence of minimal bad sequences is, in the sense of reverse mathematics, stronger than Kruskal's tree theorem. See Rathjen and Weiermann~\cite{rathjen:proof-theoretic:}.}.
We also appeal to Higman's lemma, which is an easy consequence of the existence of minimal bad sequences. However, before introducing these tools we must address a technical matter.

Recall that a quasi-order is \emph{well founded} if every nonempty subset of its elements contains a minimal element (one that is not greater than any other element). Every wqo set is well founded, as otherwise it would contain an infinite strictly decreasing sequence. While not wqo, the containment order on permutations is also well founded. It is also not hard to see that if $(X,\le)$ is a well-founded quasi-order, then the product $(X^m,\le)$ is also well founded (under the product order defined in Proposition~\ref{prop-wqo-vector}). Since we need well foundedness in order to guarantee the existence of minimal bad sequences, we first establish that sets of labelled permutations are well founded.

\begin{proposition}
\label{prop-labelled-well-founded}
For any nonempty set~$X$ of permutations and any quasi-order $(L,\le_L)$, $X\wr L$ is well founded if and only if~$L$ is well founded.
\end{proposition}
\begin{proof}
If $L$ is not well founded, then since~$X$ is nonempty, it is easy to see that $X\wr L$ is also not well founded. For the other direction, suppose that $L$ is well founded and let $S$ denote a nonempty subset of $X\wr L$. Further let $U$ denote the set of underlying permutations of members of $S$, so
\[
	U=\{\pi\in X : \text{$S$ contains an $L$-labeling of~$\pi$}\}.
\]
Because the containment order on permutations is well founded, $U$ has a minimal element, say~$\pi$. Suppose~$\pi$ has length $n$, so each $L$-labeling of~$\pi$ lying in $S$ can be identified with an $n$-tuple in~$L^n$. The order on $L$-labelings of~$\pi$ is precisely the product order on~$L^n$, which is well founded because $L$ is well founded. Therefore this set of labelings has a minimal element, and~$\pi$ with this minimal labeling is a minimal element of the set $S$.
\end{proof}

We may now define minimal bad sequences. Suppose that the quasi-order $(X,\le)$ is \emph{not} wqo. We define a \emph{bad sequence} from~$X$ to be an infinite sequence $x_1,x_2,\dots$ of elements of~$X$ that does not contain a good pair, meaning that for all indices $i<j$ we have $x_i\not\le x_j$. A bad sequence $x_1,x_2,\dots$ from~$X$ is \emph{minimal} if, for all indices~$i$, there does not exist a bad sequence $x_1,x_2,\dots,x_{i-1},y_i,y_{i+1},\dots$ with $y_i<x_i$.

\begin{proposition}[Nash-Williams~{\cite[Proof of Lemma 2]{nash-williams:on-well-quasi-o:}}]
\label{prop-wqo-iff-min-bad-seq}
A well-founded quasi-order $(X,\le)$ is wqo if and only if it does not contain a minimal bad sequence.
\end{proposition}
\begin{proof}
It follows from the definition that wqo is equivalent to the absence of bad sequences. Thus it suffices to prove that if $(X,\le)$ contains a bad sequence, then it also contains a minimal bad sequence. Because $(X,\le)$ is well founded we may choose an element $x_1\in X$ to be minimal such that it begins a bad sequence. We may then choose an element $x_2\in X$ to be minimal such that~${x_1,x_2}$ begins a bad sequence. Proceeding by induction\footnote{Note that this inductive construction of a bad sequence requires, at a minimum, the axiom of dependent choice in order to make the countably infinite number of choices of entries $x_i$. Nash-Williams~\cite{nash-williams:on-well-quasi-o:} assumed the axiom of choice in constructing his bad sequence.}, if we assume that $x_1,x_2,\dots,x_i$ begins a bad sequence, we may choose an element $x_{i+1}\in X$ to be minimal such that $x_1,x_2,\dots,x_i,x_{i+1}$ begins a bad sequence. The result is a minimal bad sequence.
\end{proof}

Given any subset $S$ of a quasi-order $(X,\le)$, we define its \emph{proper closure} to be the set
\[
	S^{\suplessthan} = \{y : \text{$y < x$ for some $x\in S$}\}.
\]
All we require about minimal bad sequences, other than their existence, is the following result of Nash-Williams stating that their proper closures are wqo. A strengthening of this result has been given by Gustedt~\cite[Theorem~6]{gustedt:finiteness-theo:}; indeed, \cite{gustedt:finiteness-theo:} contains a more thorough treatment of all of the material in this section.

\begin{proposition}[Nash-Williams~{\cite[Proof of Lemma 2]{nash-williams:on-well-quasi-o:}}]
\label{prop-min-bad-seq-prop-closure-wqo}
If $S=\{x_1,x_2,\dots\}$ is a minimal bad sequence in a quasi-order $(X,\le)$, then its proper closure $S^{\suplessthan}$ is wqo%
\footnote{One might wonder if $S^{\suplessthan}$ is ever lwqo under these hypotheses. For us to be able to define lwqo, $(X,\le)$ must consist of objects with ground sets that can be labelled. When~$X$ is a permutation class (or graph class), $S^{\suplessthan}$ \emph{cannot} be lwqo. This is because $S^{\suplessthan}$ is a class itself and has an infinite basis (the minimal elements of $S$, of which there must be infinitely many), but lwqo permutation classes must have finite bases by Proposition~\ref{prop-lwqo-fin-basis} (and as already remarked, the analogous result holds for graph classes).}.
\end{proposition} 
\begin{proof}
Suppose to the contrary that $S^{\suplessthan}$ is not wqo, so it contains a bad sequence $y_1,y_2,\dots$. By the definition of $S^{\suplessthan}$, there is a function $f$ such that $y_n<x_{f(n)}$ for all $n\ge 1$. Choose $m$ such that
\[
	f(m)=\min\{ f(n) : n=1,2,\dots \}.
\]
We claim that
\[
	x_1,x_2,\dots,x_{f(m)-1},y_m,y_{m+1},\dots
\]
is a bad sequence. Note that since this claim contradicts the minimality of $x_1,x_2,\dots$, its proof will complete the proof of the proposition. Because the elements $x_i$ and the elements $y_j$ belong to bad sequences, it suffices to show that $x_i\not\le y_j$ for $1\le i<f(m)$ and $j\ge m$. For each such~$i$ we have $i<f(m)\le f(j)$ by our choice of $m$, so $x_i\not\le x_{f(j)}$. On the other hand, we have~${y_j<x_{f(j)}}$, implying that $x_i\not\le y_j$ and completing the proof of the claim and the proposition.
\end{proof}

For the final result of this section we present a special case of Higman's lemma and its short derivation using minimal bad sequences. It should be noted both that Higman's lemma applies in more general settings than we encounter here and that Higman's original proof predates the one below by over a decade%
\footnote{Higman's result is stated in the general context of abstract algebras with a wqo set of finitary operations. The version presented here and in most of the literature is the specialisation of his result to the case of the single binary operation of concatenation of words. Higman's proof proceeds by induction on the arity of the operations, at each step arguing by ``descent'': any counterexample must give rise to a smaller one, but this process cannot continue indefinitely.}.

Given a poset $(X,\le)$, we denote by $X^\ast$ the set (also called a language or free monoid) of all words with letters from~$X$ (equivalently, finite sequences with elements from~$X$). The \emph{generalised subword order} on $X^\ast$ is defined by stipulating that $v=v_1\cdots v_k$ is contained in $w=w_1\cdots w_n$ if and only if $w$ has a subsequence $w_{i_1}w_{i_2}\cdots w_{i_k}$ such that $v_j\le w_{i_j}$ for all $j$.

\newtheorem*{higmans-lemma}{\rm \textbf{Higman's Lemma}~\cite{higman:ordering-by-div:}}
\begin{higmans-lemma}
If $(X,\le)$ is wqo, then $X^\ast$ is also wqo, under the generalised subword order.
\end{higmans-lemma}
\begin{proof}
Suppose to the contrary that there was a minimal bad sequence $S=\{w_1,w_2,\dots\}$ from~$X^\ast$. Express each word $w_i$ as $w_i=\ell_iu_i$ where $\ell_i\in X$ is the first letter of $w_i$ and $u_i\in S^{\suplessthan}$ is the rest of the word, and define $\Psi : S\to X\times S^{\suplessthan}$ by $\Psi(w_i)=(\ell_i,u_i)$. It follows that $\Psi$ is order-reflecting: if
\[
	\Psi(w_i)
	=
	(\ell_i,u_i)
	\le
	(\ell_j,u_j)
	=
	\Psi(w_j),
\]
then $w_i\le w_j$.
However,~$X$ is wqo by our hypotheses and $S^{\suplessthan}$ is wqo by Proposition~\ref{prop-min-bad-seq-prop-closure-wqo}, so Proposition~\ref{prop-wqo-order-reflecting} implies that $S$ is wqo, and this contradiction completes the proof.
\end{proof}

\section{Sums and Skew Sums}
\label{sec-sums}

Our first application of the tools of the previous section is to sums and skew sums of permutation classes. It should be noted that sums and skew sums are but a small part of the substitution decomposition, which we cover in depth in the next section. However, we are able to establish stronger results in this context than in the more general context of the substitution decomposition.

The sum and skew sum of two permutations was defined in Section~\ref{subsec-sums}. Given permutation classes~$\C$ and $\D$, we now define their \emph{sum} by
\[
	\C\oplus\D
	=
	\{\sigma\oplus\tau : \text{$\sigma\in\C$ and $\tau\in\D$}\}.
\]
This set is always a permutation class itself, due to our convention that every nonempty permutation class must contain the empty permutation. The \emph{skew sum} of the classes~$\C$ and $\D$, denoted by~$\C\ominus\D$, is defined analogously. 

If both~$\C$ and $\D$ are wqo, then Proposition~\ref{prop-wqo-product} shows that their Cartesian product~$\C\times\D$ is wqo. We can then conclude by Proposition~\ref{prop-wqo-order-preserving} that~$\C\oplus\D$ is wqo because the surjective mapping ${\Phi : \C\times\D\to\C\oplus\D}$ given by
\[
	\Phi((\sigma,\tau))=\sigma\oplus\tau
\]
is order-preserving. Obviously the same holds for their skew sum,~$\C\ominus\D$.

This fact seems to have first been observed by Atkinson, Murphy, and Ru\v{s}kuc~{\cite[Lemma~2.4]{atkinson:partially-well-:}. It is possible to replace ``wqo'' in this result by ``lwqo'', although to do so we need to define sums of label functions. Note that if $(\sigma,\ell_\sigma),(\tau,\ell_\tau)\in\C\wr L$ are {$L$-labelled} permutations of lengths $m$ and $n$, respectively, then the natural way to $L$-label their sum, $(\sigma,\ell_\sigma)\oplus(\tau,\ell_\tau)$, is to attach the labels of $\ell_\sigma$ to the first $m$ entries of $\sigma\oplus\tau$ and to attach the labels of $\ell_\tau$ to the last $n$ entries of $\sigma\oplus\tau$.
With this in mind, let $(L,\le_L)$ be any quasi-order. Given two label functions ${\ell_1 : \{1,2,\dots,m\}\to L}$ and $\ell_2 : \{1,2,\dots,n\}\to L$, we define the label function ${\ell_1\oplus\ell_2 : \{1,2,\dots,m+n\}\to L}$ by
\[
	(\ell_1\oplus\ell_2)(i) =
	\left\{
	\begin{array}{ll}
	\ell_1(i)&\mbox{for $1\le i\le m$,}\\
	\ell_2(i-m)&\mbox{for $m+1\le i\le m+n$.}
	\end{array}
	\right.
\]
For $(\sigma,\ell_\sigma),(\tau,\ell_\tau)\in\C\wr L$, the definition of $(\sigma,\ell_\sigma)\oplus(\tau,\ell_\tau)$ is then
\[
	(\sigma,\ell_\sigma)\oplus(\tau,\ell_\tau)
	=
	(\sigma\oplus\tau, \ell_\sigma\oplus\ell_\tau).
\]
Our argument that the sum~$\C\oplus\D$ is wqo whenever both~$\C$ and $\D$ are wqo now extends to show that~$\C\oplus\D$ is lwqo whenever both~$\C$ and $\D$ are lwqo. We simply need to take $(L,\le_L)$ to be an arbitrary wqo set and consider the surjective order-preserving mapping $\Phi : (\C\wr L)\times(\D\wr L)\to (\C\oplus\D)\wr L$ defined by
\[
	\Phi((\sigma,\ell_\sigma),(\tau,\ell_\tau))
	=
	(\sigma,\ell_\sigma)\oplus(\tau,\ell_\tau)
	=
	(\sigma\oplus\tau, \ell_\sigma\oplus\ell_\tau).
\]
We record this fact below.

\begin{proposition}
\label{prop-sums-wqo-lwqo}
If the classes~$\C$ and $\D$ are both wqo (resp., lwqo), then~$\C\oplus\D$ and~$\C\ominus\D$ are also wqo (resp., lwqo).
\end{proposition}

The class~$\C$ is \emph{sum closed} if~$\C\oplus\C\subseteq\C$. Given any class~$\C$, its \emph{sum closure}, denoted by~$\bigoplus \C$, is defined to be the smallest (in terms of set containment) sum closed permutation class containing~$\C$. Equivalently,
\[
	\bigoplus \C
	=
	\{
	\alpha_1\oplus\alpha_2\oplus\cdots\oplus\alpha_m
	:
	\alpha_1,\dots,\alpha_m\in\C
	\}.
\]
We define the terms \emph{skew closed} and \emph{skew closure} analogously. For enumeration, it is important to observe that every permutation $\pi\in\bigoplus\C$ can be expressed uniquely as a sum of sum indecomposable permutations (resp., a skew sum of skew indecomposable permutations). However, wqo and lwqo arguments do not require such fine control over the structure of these classes. The wqo content of the following result was first observed by Atkinson, Murphy, and Ru\v{s}kuc~{\cite[Theorem 2.5]{atkinson:partially-well-:}}.

\begin{theorem}
\label{thm-lwqo-sum-closure}
If the class~$\C$ is wqo (resp., lwqo), then its sum closure $\bigoplus\C$ and skew closure~$\bigominus\C$ are also wqo (resp., lwqo).
\end{theorem}
\begin{proof}
The skew versions of the result follow by symmetry from the sum versions, so we consider only the latter. First suppose that~$\C$ is wqo. Higman's lemma shows that~$\C^\ast$ is wqo under the generalised subword order. It follows by inspection that the mapping $\Phi :  \C^\ast\to\bigoplus\C$ defined by
\[
	\Phi(\alpha_1\alpha_2\cdots\alpha_m)
	=
	\alpha_1\oplus\alpha_2\oplus\cdots\oplus\alpha_m
\]
is order-preserving. Therefore Proposition~\ref{prop-wqo-order-preserving} implies that $\bigoplus\C$ is wqo.

Now suppose that~$\C$ is lwqo and take $(L,\le_L)$ to be an arbitrary wqo set. Thus~$\C\wr L$ is wqo, and so $(\C\wr L)^\ast$ is wqo by Higman's lemma. Define the mapping $\Phi : (\C\wr L)^\ast\to (\bigoplus\C)\wr L$ by
\[
	\Phi((\alpha_1,\ell_1)(\alpha_2,\ell_2)\cdots(\alpha_m,\ell_m))
	=
	(\alpha_1\oplus\alpha_2\oplus\cdots\oplus\alpha_m, \ell_1\oplus\ell_2\oplus\cdots\oplus\ell_m).
\]
Again, $\Phi$ is order-preserving, so $(\bigoplus\C)\wr L$ is wqo by Proposition~\ref{prop-wqo-order-preserving}, proving that $\bigoplus\C$ is lwqo.
\end{proof}

We conclude with a result referenced several times already, that the downward closure of the increasing oscillations,~$\O_I$, is wqo. While~$\O_I$ is a sum closed class, every permutation from~$\O_I$ is contained in a sum indecomposable permutation from~$\O_I$, and this implies that~$\O_I$ is not contained in the sum closure of a smaller class. Thus the fact that~$\O_I$ is wqo is not a consequence of Theorem~\ref{thm-lwqo-sum-closure}.

Instead, to establish that~$\O_I$ is wqo, we note that every member of~$\O_I$ can be expressed as a sum of increasing oscillations. Recall that the set of increasing oscillations (under our conventions) is
\[
	\{1,
	21,
	231,
	312,
	2413,
	3142,
	24153,
	31524,
	241635,
	315264,
	2416375,
	3152746,
	\dots\}.
\]
It is not difficult to see that there are two increasing oscillations of each length $n\ge 3$. Viewing the increasing oscillations themselves as a poset under the permutation containment order---as in Figure~\ref{fig-inc-osc-Hasse}---we see that these two increasing oscillations of each length are partitioned into two chains, and that both increasing oscillations of length $n\ge 3$ are contained in both increasing oscillations of length $n+1$. The poset of increasing oscillations is therefore trivially wqo. It then follows from Higman's lemma and Proposition~\ref{prop-wqo-order-preserving} that the class~$\O_I$ is wqo. Note that we have already observed (for instance, with Figure~\ref{fig-three-antichains-lwqo}) that~$\O_I$ is not lwqo (it is not even $2$-wqo). This gives us the following.

\begin{figure}
	\begin{center}
	\begin{footnotesize}
	\begin{tikzpicture}[every node/.style={},xscale=2, yscale=1]
		\node (1) at (0.5,0) {$1$};
		\node (21) at (0.5,1) {$21$};
		\node (231) at (0,2) {$231$};
		\node (312) at (1,2) {$312$};
		\node (2413) at (0,3) {$2413$};
		\node (3142) at (1,3) {$3142$};
		\node (24153) at (0,4) {$24153$};
		\node (31524) at (1,4) {$31524$};
		\node (d0) at (0,4.75) {$\vdots$};
		\node (d1) at (1,4.75) {$\vdots$};
		\foreach \q/\p in {1/21, 21/231, 21/312, 231/2413, 231/3142, 312/2413, 312/3142, 2413/24153, 2413/31524, 3142/24153, 3142/31524}
			\draw [thick, line cap=round] (\q) -- (\p);
	\end{tikzpicture}
	\end{footnotesize}
	\end{center}
\caption{The Hasse diagram of  increasing oscillations under the permutation containment order.}
\label{fig-inc-osc-Hasse}
\end{figure}
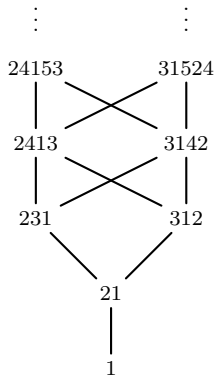

\begin{proposition}
\label{prop-OI-wqo}
The downward closure of the increasing oscillations,~$\O_I$, is wqo but not lwqo.
\end{proposition}

Another example of a class that is wqo but not lwqo (or even $2$-wqo) is given by the downward closure of the widdershins spirals; see Brignall, Engen, and Vatter~\cite[Proposition 3.3]{brignall:a-counterexampl:} for a proof.

There is a naturally defined class containing both $\bigoplus\C$ and $\bigominus\C$: the \emph{separable closure} of~$\C$ is defined to be the smallest permutation class containing~$\C$ that is both sum and skew closed. (The separable closure of~$\C$ has also been called the \emph{strong completion} of~$\C$ by some authors, including Murphy~\cite[Section~2.2.5]{murphy:restricted-perm:}.) Atkinson, Murphy, and Ru\v{s}kuc~{\cite[Theorem~2.5]{atkinson:partially-well-:}} showed that the separable closure of a wqo class is itself wqo by appealing to a more general version of Higman's lemma than we have presented. We derive and generalise this result later, in Corollary~\ref{cor-separable-closure-wqo}, as a consequence of more general results on the substitution decomposition.

\section{The Substitution Decomposition}
\label{sec-subst-decomp}

Having considered in detail the relationship between sums and skew sums and labelled well-quasi-order, we now consider the more general context provided by the substitution decomposition. This notion, introduced in the context of permutations momentarily, is common to all relational structures and has appeared in a wide variety of settings under numerous names, such as modular decomposition, X-join, and lexicographic sum.
As well as furthering our story about lwqo in permutations, the parallel notion for graphs is sufficiently strongly related to the permutation version to establish a partial answer to Question~\ref{question-prop-wqo-perms-graphs-converse}, as well as a complete answer to the lwqo analogue (at the end of this section).

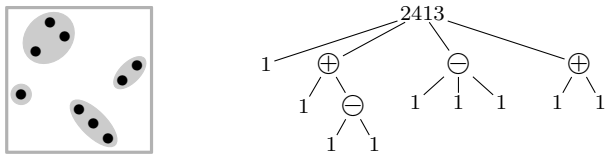
\begin{figure}
\begin{footnotesize}
\begin{center}
	\begin{tikzpicture}[scale=0.1925, baseline=(current bounding box.south)]
		\draw [lightgray, fill] (1,4) circle (20pt);
		\draw[lightgray, fill, rotate around={-45:(2.9,7.9)}] (2.9,7.9) ellipse (45pt and 55pt);
		\draw[lightgray, fill, rotate around={45:(6,2)}] (6,2) ellipse (25pt and 60pt);
		\draw[lightgray, fill, rotate around={-45:(8.5,5.5)}] (8.5,5.5) ellipse (20pt and 40pt);
		\plotpermbox{0.5}{0.5}{9.5}{9.5};
		\plotperm{4,7,9,8,3,2,1,5,6};
	\end{tikzpicture}
	\quad\quad\quad\quad
	\tikzset{inner sep=1pt, outer sep=0pt}
	\begin{forest}
		[$2413$, l sep=6pt, s sep=15pt
			[$1$, anchor=center]
			[$\bigoplus$, l sep=6pt, s sep=10pt, anchor=center
				[$1$, l=-1pt, anchor=center]
				[$\bigominus$, l=-1pt, l sep=6pt, s sep=10pt, anchor=center
					[$1$, l=-1pt, anchor=center]
					[$1$, l=-1pt, anchor=center]
				]
			]
			[$\bigominus$, l sep=6pt, s sep=10pt, anchor=center
				[$1$, l=-1pt, anchor=center]
				[$1$, l=-1pt, anchor=center]
				[$1$, l=-1pt, anchor=center]
			]
			[$\bigoplus$, l sep=6pt, s sep=10pt, anchor=center
				[$1$, l=-1pt, anchor=center]
				[$1$, l=-1pt, anchor=center]
			]
		]
	\end{forest}
\end{center}
\end{footnotesize}
\caption{The plot of the permutation $479832156$ (left) and its substitution decomposition tree (right).}
\label{fig-subst-tree}
\end{figure}

We begin with the definitions in the permutation context. An \emph{interval} in the permutation~$\pi$ is a set of contiguous indices $I=\{a,a+1,\dots,b\}$ such that the set of values ${\pi(I)=\{\pi(i) : i\in I\}}$ is also contiguous.  Given a permutation~$\sigma$ of length $m$ and nonempty permutations $\alpha_1,\dots,\alpha_m$, the \emph{inflation} of~$\sigma$ by $\alpha_1,\dots,\alpha_m$,  denoted by $\sigma[\alpha_1,\dots,\alpha_m]$, is the unique permutation of length ${|\alpha_1|+\cdots+|\alpha_m|}$ obtained by replacing each entry $\sigma(i)$ by an interval that is order isomorphic to $\alpha_i$ in such a way that the intervals are themselves order isomorphic to~$\sigma$.  For example,
\[
	2413[1,132,321,12]=4\ 798\ 321\ 56,
\]
the permutation plotted on the left of Figure~\ref{fig-subst-tree}.

Every permutation of length $n\ge 1$ has \emph{trivial} intervals of lengths $0$, $1$, and $n$; all other intervals are termed \emph{proper}. We further say that the empty permutation and the permutation $1$ are \emph{trivial}. A nontrivial permutation is \emph{simple} if it has no proper intervals. The shortest simple permutations are thus $12$ and $21$, there are no simple permutations of length three, and the simple permutations of length four are $2413$ and $3142$. We have seen many simple permutations in the preceding pages---the permutation $36285714$ plotted in Figure~\ref{fig-perm-contain}, the permutations in the centre and right of Figure~\ref{fig-three-inc-osc-antichains}, the underlying permutations of Figure~\ref{fig-three-antichains-lwqo}, and the widdershins spirals of Figure~\ref{fig-widdershins} are all simple permutations.

The following result follows immediately from the definitions.

\begin{proposition}
\label{prop-simple-decomp-basic}
Every nontrivial permutation can be expressed as an inflation of a nontrivial simple permutation.
\end{proposition}

This process of expressing a permutation as the inflation of a simple permutation is what we call the \emph{substitution decomposition}. By repeatedly applying Proposition~\ref{prop-simple-decomp-basic} to decompose a permutation, and then to decompose its nontrivial intervals, and so on, one obtains a \emph{substitution decomposition tree}. An example of a substitution decomposition tree is shown on the right of Figure~\ref{fig-subst-tree}.

We are also interested in the inflation of one class by another. Given two classes~$\C$ and $\U$, the \emph{inflation} of~$\C$ by $\U$ is defined as
\[
	\C[\U]
	=
	\{\sigma[\alpha_1,\dots,\alpha_m]
	:
	\mbox{$\sigma\in\C_m$ and $\alpha_1,\dots,\alpha_m\in\U$}
	\}.
\]

For enumeration, it is essential to associate each permutation to a \emph{unique} substitution decomposition. The standard uniqueness result here is Albert and Atkinson~{\cite[Proposition~2]{albert:simple-permutat:}}, while Brignall~{\cite[Lemma 3.1]{brignall:wreath-products:}} gives a version for classes of the form~$\C[\U]$. However, all we need to establish our lwqo results is Proposition~\ref{prop-simple-decomp-basic}. We begin by considering classes of the form~$\C[\U]$.

\begin{theorem}
\label{thm-CU-inflate-wqo}
If the permutation class~$\C$ is lwqo and the class $\U$ is wqo, then the class~$\C[\U]$ is wqo.
\end{theorem}
\begin{proof}
Suppose the class~$\C$ is lwqo and the class $\U$ is wqo. Thus the set of $\U$-labelled permutations of~$\C$,~$\C\wr\U$, is wqo. We define the mapping $\Phi : \C\wr\U\tosurj\C[\U]$ by
\[
	\Phi((\pi,\ell_\pi))
	=
	\pi[\ell_\pi(1),\dots,\ell_\pi(n)],
\]
where $n$ denotes the length of~$\pi$. Note that $\Phi$ is surjective by the definition of~$\C[\U]$.

Suppose that $(\sigma,\ell_\sigma)\le (\pi,\ell_\pi)\in\C\wr\U$, where~$\sigma$ and~$\pi$ have lengths $k$ and $n$, respectively. As witness to this containment, there must exist indices $1\le i_1<\cdots<i_k\le n$ so that $\pi(i_1)\cdots\pi(i_k)$ is order isomorphic to~$\sigma$ and $\ell_\sigma(j)\le\ell_\pi(i_j)$ for all indices $1\le j\le k$. Using this witness, we see that the permutation $\sigma[\ell_\pi(i_1),\dots,\ell_\pi(i_k)]$ contains the permutation $\sigma[\ell_\sigma(1),\dots,\ell_\sigma(k)]$, and is contained in the permutation $\pi[\ell_\pi(1),\dots,\ell_\pi(n)]$. Therefore,
\[
	\Phi((\sigma,\ell_\sigma))
	\le
	\sigma[\ell_\pi(i_1),\dots,\ell_\pi(i_k)]
	\le
	\Phi((\pi,\ell_\pi)),
\]
establishing that $\Phi$ is order-preserving. The result now follows from Proposition~\ref{prop-wqo-order-preserving}.
\end{proof}

\begin{table}
\begin{center}
\begin{tabular}{ccl}
	$\C$&$\U$&$\C[\U]$\\\hline
	wqo&wqo&not necessarily wqo\\
	wqo&lwqo&not necessarily wqo\\
	lwqo&wqo&wqo by Theorem~\ref{thm-CU-inflate-wqo} but not necessarily lwqo\\
	lwqo&lwqo&lwqo by Corollary~\ref{cor-CU-inflate-lwqo}
\end{tabular}
\end{center}
\caption{The wqo/lwqo status of various inflations of permutation classes.}
\label{table-wqo-lwqo-inflations}
\end{table}

The four possible variations on the hypotheses of Theorem~\ref{thm-CU-inflate-wqo} are considered in Table~\ref{table-wqo-lwqo-inflations}. To justify the first two rows of this table, take~$\C$ to be the downward closure of the increasing oscillations,~$\O_I$, which is wqo by Proposition~\ref{prop-OI-wqo}, and take $\U=\{1,12\}$, which is lwqo because it is finite (Proposition~\ref{prop-finite-lwqo}). Then~$\C[\U]$ contains the infinite antichain shown on the left of Figure~\ref{fig-three-antichains}, and thus is not wqo. Beyond Theorem~\ref{thm-CU-inflate-wqo}, the third row of Table~\ref{table-wqo-lwqo-inflations} says that the inflation of an lwqo class by a wqo class is not necessarily lwqo; for example, consider~$\C=\{1\}$ and~$\U=\O_I$. The fourth line of Table~\ref{table-wqo-lwqo-inflations} is settled by Corollary~\ref{cor-CU-inflate-lwqo}, which follows from our result on substitution closures below.


The class~$\C$ is said to be \emph{substitution closed} if~$\C[\C]\subseteq\C$. The \emph{substitution closure} $\langle\C\rangle$ of a class~$\C$ is defined as the smallest substitution closed class containing~$\C$.  A standard argument shows that $\langle\C\rangle$ exists, and the following result also follows readily.

\begin{proposition}
\label{simples-in-substitution-completion}
The substitution closure $\langle \C \rangle$ of the class~$\C$ is the largest class of permutations that contains precisely the same simple permutations as~$\C$.
\end{proposition}

We have observed that the inflation of a wqo class by another wqo class is not necessarily wqo, so the substitution closure of a wqo class need not be wqo; a concrete example is $\langle\O_I\rangle$. However, as we show below, the substitution closure of an \emph{lwqo} class is always \emph{lwqo}. In order to establish this result, we must extend our notion of substitution decomposition to labelled permutations. We do this exactly like we extended the notion of sums to labelled permutations in Section~\ref{sec-sums}. Thus we define the inflation of an unlabelled permutation~$\sigma$ of length $m$ by a sequence of $m$ labelled permutations $(\alpha_1,\ell_1),\dots,(\alpha_m,\ell_m)$ as the permutation $\sigma[\alpha_1,\dots,\alpha_m]$ in which the first $|\alpha_1|$ entries are labelled by $\ell_1(1)$, $\dots$, $\ell_1(|\alpha_1|)$, the next $|\alpha_2|$ entries are labelled by $\ell_2(1)$, $\dots$, $\ell_2(|\alpha_2|)$, and so on. Thus we formally have, using the definition of sums of label functions from Section~\ref{sec-sums}, that
\[
	\sigma[(\alpha_1,\ell_1),\dots,(\alpha_m,\ell_m)]
	=
	(\sigma[\alpha_1,\dots,\alpha_m], \ell_1\oplus\cdots\oplus\ell_m).
\]
It follows immediately from Proposition~\ref{prop-simple-decomp-basic} that every nontrivial \emph{labelled} permutation can be expressed as an inflation of a nontrivial (unlabelled) simple permutation by \emph{labelled} permutations.

\begin{theorem}
\label{thm-subst-closure-lwqo}
If the permutation class~$\C$ is lwqo, then its substitution closure $\langle\C\rangle$ is also lwqo.
\end{theorem}
\begin{proof}
Let~$\C$ be an lwqo class and take $(L,\mathord{\le_L})$ to be an arbitrary wqo set. Suppose to the contrary that $\langle\C\rangle\wr L$ is not wqo and take a minimal bad sequence $S=\{(\pi_1,\ell_1),(\pi_2,\ell_2),\dots\}$ from $\langle\C\rangle\wr L$. Thus $S^{\suplessthan}$ is wqo by Proposition~\ref{prop-min-bad-seq-prop-closure-wqo}. Some members of $S$ may be labelings of the trivial permutation~$1$, that is, members of $\{1\}\wr L$. However, $\{1\}$ is trivially lwqo (by Proposition~\ref{prop-finite-lwqo}, for example), and thus only finitely many members of $S$ may lie in $\{1\}\wr L$. For all of the other members of $S$, Proposition~\ref{prop-simple-decomp-basic} implies that there is a simple permutation $\sigma_i\in\C$ of length $m_i\ge 2$ and labelled permutations $(\alpha_{i,1},\ell_{i,1}),\dots,(\alpha_{i,m_i},\ell_{i,m_i})\in S^{\suplessthan}$ such that
\[
	(\pi_i,\ell_i)
	=
	\sigma_i[(\alpha_{i,1},\ell_{i,1}),\dots,(\alpha_{i,m_i},\ell_{i,m_i})].
\]
We define the mapping $\Phi : \C\wr S^{\suplessthan}\to \langle\C\rangle\wr L$ by
\[
	\Phi((\sigma,\ell_\sigma))
	=
	\sigma[\ell_\sigma(1),\dots,\ell_\sigma(m)],
\]
where $m$ denotes the length of~$\sigma$; note here that each label $\ell_\sigma(i)\in S^{\suplessthan}$ is itself a labelled permutation since $S^{\suplessthan}\subseteq \langle\C\rangle\wr L$. The mapping $\Phi$ is order-preserving, as can be seen by an argument analogous to that used in the proof of Theorem~\ref{thm-CU-inflate-wqo}. We know that $S^{\suplessthan}$ is wqo and thus our hypothesis implies that~$\C\wr\S^{\suplessthan}$ is wqo, so the range of $\Phi$ is wqo by Proposition~\ref{prop-wqo-order-preserving}. By our observations above, this range contains all but finitely many members of $S$, but this is a contradiction because a wqo set cannot contain a bad sequence.
\end{proof}

Let $S$ denote the set of simple permutations in the class~$\C$. It follows from Proposition~\ref{prop-simple-decomp-basic} that $\langle\C\rangle=\langle S^{\suplessthaneq}\rangle$. Moreover, if the set $S$ of simple permutations is lwqo then Theorem~\ref{thm-lwqo-downward-closure} shows that $S^{\suplessthaneq}$ is lwqo, so we have the following.

\begin{corollary}
\label{cor-simples-lwqo}
If the set of simple permutations contained in the permutation class~$\C$ is lwqo, then $\langle\C\rangle$ is also lwqo. In particular, if the set of simple permutations of the permutation class~$\C$ is lwqo, then~$\C$ itself is lwqo.
\end{corollary}

We observed with Proposition~\ref{prop-finite-lwqo} that finite sets of permutations are always lwqo. Thus we can further specialise Corollary~\ref{cor-simples-lwqo} to obtain an lwqo strengthening of a result of Albert and Atkinson.

\begin{corollary}
[Cf. Albert and Atkinson~{\cite[Corollary~8]{albert:simple-permutat:}}]
\label{cor-simples-lwqo-finite}
Every permutation class containing only finitely many simple permutations is lwqo.
\end{corollary}

To complete the discussion of Table~\ref{table-wqo-lwqo-inflations}, we observe that the inflation of an lwqo class by another lwqo class is also lwqo.

\begin{corollary}
\label{cor-CU-inflate-lwqo}
If the permutation classes~$\C$ and $\U$ are both lwqo, then the class~$\C[\U]$ is also lwqo.
\end{corollary}
\begin{proof}
If both~$\C$ and $\U$ are lwqo then their union is trivially lwqo, so $\langle\C\cup\U\rangle$ is lwqo by Theorem~\ref{thm-subst-closure-lwqo}, and since~$\C[\U]\subseteq \langle\C\cup\U\rangle$, we see that~$\C[\U]$ is lwqo, as desired.
\end{proof}

Another way to define the separable permutations (first defined in Section~\ref{subsec-sums}) is as the substitution closure $\langle\{1,12,21\}\rangle$. Thus Corollary~\ref{cor-simples-lwqo-finite} immediately implies that the class of separable permutations is lwqo. Moreover, the separable closure of the class~$\C$ (defined at the end of Section~\ref{sec-sums}) is equal to $\langle\{1,12,21\}\rangle[\C]$, and thus we have the following.

\begin{corollary}
[Cf. Atkinson, Murphy, and Ru\v{s}kuc~{\cite[Theorem 2.5]{atkinson:partially-well-:}}]
\label{cor-separable-closure-wqo}
If the class~$\C$ is wqo (resp., lwqo), then its separable closure is also wqo (resp., lwqo).
\end{corollary}

Corollary~\ref{cor-simples-lwqo-finite} and Proposition~\ref{prop-lwqo-fin-basis} imply that every permutation class with only finitely many simple permutations is finitely based. This fact was first proved by Albert and Atkinson via a result about substructures of simple relational structures due to Schmerl and Trotter~\cite{schmerl:critically-inde:}%
\footnote{See Brignall and Vatter~\cite{brignall:a-simple-proof-:} for a proof of Schmerl and Trotter's theorem in the special case of permutations.}.

\begin{corollary}
[Cf. Albert and Atkinson~{\cite[Theorem~9]{albert:simple-permutat:}}]
\label{cor-simples-basis-finite}
Every permutation class containing only finitely many simple permutations is finitely based.
\end{corollary}

We mention Corollary~\ref{cor-simples-basis-finite} here only for historical interest; it is a much, much weaker result than what Theorem~\ref{thm-subst-closure-lwqo} and Proposition~\ref{prop-lwqo-fin-basis} imply, which is that the substitution closure of \emph{any} lwqo class is finitely based. The bases of substitution closures in general have particularly nice forms%
\footnote{However, in practice it can be difficult to establish precisely what the members of the basis are, and there are frequently infinitely many of them---see Atkinson, Ru\v{s}kuc, and Smith~\cite{atkinson:substitution-cl:} for one such example.},
as shown below. This result is another ingredient in Albert and Atkinson's proof of Corollary~\ref{cor-simples-basis-finite} and follows readily from Proposition~\ref{simples-in-substitution-completion}.

\begin{proposition}[Albert and Atkinson~\cite{albert:simple-permutat:}]
\label{prop-substitution-completion-basis}
The basis of the substitution closure of a class~$\C$ consists of the minimal simple permutations not contained in~$\C$.
\end{proposition}

The result below follows from Schmerl and Trotter's theorem and Proposition~\ref{prop-substitution-completion-basis}. We mention this result because it relates to many of the ideas discussed here, though it is not implied by our lwqo results unless Conjecture~\ref{conj-C+1-wqo-C-lwqo} holds.

\begin{proposition}
[Albert, Ru\v{s}kuc, and Vatter~{\cite[Proposition~2.9]{albert:inflations-of-g:}}]
If the class~$\C^{+1}$ is wqo, then the class $\langle\C\rangle$ is finitely based.
\end{proposition}

Conjecture~\ref{conj-C+1-wqo-C-lwqo} has a natural analogue to substitution closures, stated below.

\begin{conjecture}
\label{conj-C-subst-wqo-C-lwqo}
If the permutation class $\langle\C\rangle$ is wqo, then~$\C$, and thus also $\langle\C\rangle$, is lwqo.
\end{conjecture}

We conclude this section by briefly discussing analogues in the graph context, and then establishing links between wqo/lwqo for permutation classes and their corresponding graph classes.

In the context of graphs, the substitution decomposition is generally called the \emph{modular decomposition}, and the analogues of simple permutations are most commonly called \emph{prime graphs}. The analogue of our Theorem~\ref{thm-subst-closure-lwqo} was established by Atminas and Lozin~\cite[Theorem~2]{atminas:labelled-induce:}. 

The notion of modular decomposition of graphs dates back to Gallai's ground-breaking paper on transitive orientations~\cite{gallai:transitiv-orien:}%
\footnote{See~\cite{gallai:a-translation-o:} for an English translation of Gallai's paper.}, %
and indeed this paper also provides the first connection between these concepts for permutations and inversion graphs. It is easy to establish that the inversion graph of a simple permutation is prime, but indeed more is true.

\begin{lemma}[Gallai~\cite{gallai:transitiv-orien:}]\label{lem-gallai-simple-prime}
If $G_\sigma$ is a prime inversion graph, then~$\sigma$ is simple, and the only permutations whose inversion graphs are isomorphic to $\G_{\sigma}$ are $\{\sigma,\sigma^{-1},\sigma^{\text{rc}},(\sigma^{\text{rc}})^{-1}\}$.
\end{lemma}

When combined with Proposition~\ref{prop-graph-to-perm-containment}, Lemma~\ref{lem-gallai-simple-prime} tells us that if~$\sigma$ is a simple permutation and~$\pi$ is any permutation, then $G_\sigma\leq G_\pi$ implies that one of~$\sigma$, $\sigma^{-1}$, $\sigma^{\text{rc}}$, or $(\sigma^{\text{rc}})^{-1}$ is contained in~$\pi$. This puts us in a position to provide the following partial answer to Question~\ref{question-prop-wqo-perms-graphs-converse} (if~$G_\C$ is wqo, is~$\C$ necessarily wqo?).

\begin{proposition}
\label{prop-GC-wqo-C-simples-wqo}
Let~$\C$ be a permutation class such that $G_\C$ is wqo in the induced subgraph order. Then the simple permutations in~$\C$ are wqo.
\end{proposition}
\begin{proof}
Suppose that $G_\C$ is wqo, and consider an arbitrary sequence $\sigma_1$, $\sigma_2$, $\dots$ of simple permutations in~$\C$. Since we are assuming that $G_\C$ is wqo, Proposition~\ref{prop-wqo-ramsey-inf-inc} shows that the corresponding sequence of inversion graphs $G_{\sigma_1}$, $G_{\sigma_2}$, $\dots$ contains an infinite increasing subsequence, although all we need is an increasing subsequence of length five, say
\[
	G_{\sigma_{i_1}}
	\le
	G_{\sigma_{i_2}}
	\le
	G_{\sigma_{i_3}}
	\le
	G_{\sigma_{i_4}}
	\le
	G_{\sigma_{i_5}}
\]
for $1\le i_1<i_2<i_3<i_4<i_5$. By Proposition~\ref{prop-graph-to-perm-containment} and Lemma~\ref{lem-gallai-simple-prime}, each of these inclusions must be witnessed by a symmetry of the shorter simple permutation embedding in the longer. Thus
\[
	\sigma_{i_1}^{s_1}
	\leq
	\sigma_{i_2}^{s_2}
	\leq
	\sigma_{i_3}^{s_3}
	\leq
	\sigma_{i_4}^{s_4}
	\leq
	\sigma_{i_5},
\]
where $s_1$, $s_2$, $s_3$, and $s_4$ are each one of the four graph-preserving symmetries. If none of these four symmetries is the identity, then two of them are the same (by the pigeonhole principle), and in either case we can find indices $j$ and $k$ with $1\leq j<k\leq 5$ such that $\sigma_{i_j}\le\sigma_{i_k}$. This implies that every infinite sequence of simple permutations of~$\C$ contains a good pair (a pair of elements in increasing order), proving the result.
\end{proof}

For the remainder of this section, we answer the lwqo analogue of Question~\ref{question-prop-wqo-perms-graphs-converse} in the affirmative, proving that a permutation class~$\C$ is lwqo if and only if the corresponding graph class $G_\C$ is lwqo. In some sense, this proof consists merely of ``adding labels'' to the proof of Proposition~\ref{prop-GC-wqo-C-simples-wqo} and then appealing to Corollary~\ref{cor-simples-lwqo}.

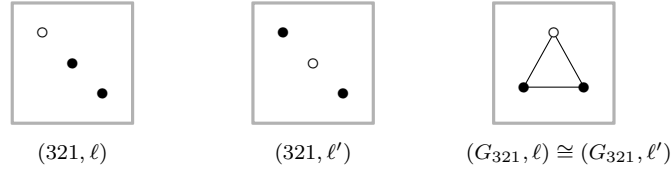
\begin{figure}
	\begin{center}
	\begin{footnotesize}
	\begin{tikzpicture}[scale=0.4, baseline=(current bounding box.center)]
		\plotpermbox{0.5}{0.5}{3.5}{3.5};
		\plotpartialperm{3/1,2/2};
		\plotpartialpermhollow{1/3};
		\node at (2,-1) {$(321,\ell)$};
	\begin{scope}[shift={(8,0)}]
		\plotpermbox{0.5}{0.5}{3.5}{3.5};
		\plotpartialperm{3/1,1/3};
		\plotpartialpermhollow{2/2};
		\node at (2,-1) {$(321,\ell')$};
	\end{scope}
	\begin{scope}[shift={(16,0)}]
		\plotpermbox{0.5}{0.5}{3.5}{3.5};
		\draw (1,1.2)--(2,3)--(3,1.2) --cycle;
		\plotpartialperm{1/1.2,3/1.2};
		\plotpartialpermhollow{2/3};
		\node at (2.5,-1) {$(G_{321},\ell)\cong (G_{321},\ell')$};
	\end{scope}
	\end{tikzpicture}
	\end{footnotesize}
	\end{center}
\caption{Two labelings of 321 (on the left) by the antichain $L=\{\circ,\bullet\}$ that are not related by any symmetry, but whose inversion graphs (on the right) are isomorphic.}
\label{fig-labeling-example}
\end{figure}%

One issue that we must handle more carefully when ``adding labels'' concerns automorphisms of inversion graphs. The four graph-preserving symmetries of  permutations necessarily induce automorphisms of their inversion graphs (indeed, this is why they are called graph-preserving), but the converse is not generally true; simply consider $G_{n\cdots 21}=K_n$, which has all $n!$ possible automorphisms. The significance of this is that automorphisms of inversion graphs can be used to rearrange the labels assigned to vertices in a way that is not represented by any symmetry of the underlying permutation---see Figure~\ref{fig-labeling-example} for an example. 

However, the situation is much nicer when the permutation is simple. 

\begin{proposition}
[{Klav\'ik and Zeman~\cite[Lemma~6.6 and its geometric interpretation]{klavik:automorphism-gr:}}]
\label{prop:klavik}
If~$G_\sigma$ is a prime inversion graph, then every automorphism of $G_\sigma$ corresponds to a symmetry of~$\sigma$, in particular, one of~$\sigma$, $\sigma^{-1}$, $\sigma^{\text{rc}}$, or $(\sigma^{\text{rc}})^{-1}$.
\end{proposition}

Equipped with the restrictions imposed by Proposition~\ref{prop:klavik}, we now show that given two labelings of the same simple permutation, the corresponding inversion graphs are isomorphic if and only if the two labelings are related by a symmetry. 
To state this result formally, we first need some notation. Given an {$L$-labelled} permutation $(\sigma,\ell_\sigma)$, we denote by $(\sigma^{-1},\ell^{-1}_\sigma)$, $(\sigma^{\text{rc}},\ell^{\text{rc}}_\sigma)$, and $((\sigma^{\text{rc}})^{-1},(\ell^{\text{rc}}_\sigma)^{-1})$ the {$L$-labelled} permutations obtained by applying each of the three graph-preserving symmetries to~$\sigma$, whilst preserving the label of each entry, as illustrated in Figure~\ref{fig-labeling-symmetry}. If~$\sigma$ has length $n$ then the resulting label functions can be described by%
\begin{align*}
	\ell_\sigma^{-1}(i) &= \ell_\sigma(\sigma^{-1}(i)),\\
	\ell_\sigma^{\text{rc}}(i) &= \ell_\sigma(n+1-i),\\
	(\ell_\sigma^{\text{rc}})^{-1}(i) &= \ell_\sigma(\sigma^{-1}(n+1-i)).
\end{align*}%
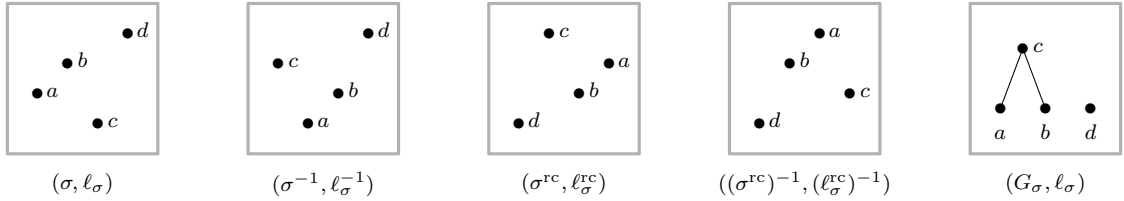
\begin{figure}
	\begin{center}
	\begin{footnotesize}
	\begin{tikzpicture}[scale=0.4, baseline=(current bounding box.center)]
		\plotpermbox{0.5}{0.5}{4.5}{4.5};
		\plotperm{2,3,1,4};
		\foreach \y/\lab [count = \n] in {2/a,3/b,1/c,4/d}
			\node[anchor=base] at (\n+0.5,\y-0.1) {$\lab$};
		\node at (2.5,-1) {$(\sigma,\ell_\sigma)$};
	\begin{scope}[shift={(8,0)}]
		\plotpermbox{0.5}{0.5}{4.5}{4.5};
		\plotperm{3,1,2,4};
		\foreach \y/\lab [count = \n] in {3/c,1/a,2/b,4/d}
			\node[anchor=base] at (\n+0.5,\y-0.1) {$\lab$};
		\node at (2.5,-1) {$(\sigma^{-1},\ell_\sigma^{-1})$};
	\end{scope}
	\begin{scope}[shift={(16,0)}]
		\plotpermbox{0.5}{0.5}{4.5}{4.5};
		\plotperm{1,4,2,3};
		\foreach \y/\lab [count = \n] in {1/d,4/c,2/b,3/a}
			\node[anchor=base] at (\n+0.5,\y-0.1) {$\lab$};
		\node at (2.5,-1) {$(\sigma^{\text{rc}},\ell_\sigma^{\text{rc}})$};
	\end{scope}
	\begin{scope}[shift={(24,0)}]
		\plotpermbox{0.5}{0.5}{4.5}{4.5};
		\plotperm{1,3,4,2};
		\foreach \y/\lab [count = \n] in {1/d,3/b,4/a,2/c}
			\node[anchor=base] at (\n+0.5,\y-0.1) {$\lab$};
		\node at (2.5,-1) {$((\sigma^{\text{rc}})^{-1},(\ell^{\text{rc}}_\sigma)^{-1})$};
	\end{scope}
	\begin{scope}[shift={(32,0)}]
		\plotpermbox{0.5}{0.5}{4.5}{4.5};
		\draw (1,1.5)--(1.75,3.5)--(2.5,1.5);
		\plotpartialperm{1/1.5,1.75/3.5,2.5/1.5,4/1.5};
		\node[anchor=base] at (1,0.5) {$a$};
		\node[anchor=base] at (2.5,0.5) {$b$};
		\node[anchor=base] at (4,0.5) {$d$};
		\node[anchor=base] at (2.25,3.4) {$c$};		
		\node at (2.5,-1) {$(G_\sigma,\ell_\sigma)$};
	\end{scope}
	\end{tikzpicture}
	\end{footnotesize}
	\end{center}
\caption{The first four diagrams show the symmetries of the permutation $\sigma=2314$ with the labeling $\ell_\sigma: 1\mapsto a,\ 2\mapsto b,\ 3\mapsto c,\ 4\mapsto d$. All of these correspond to the same labelled inversion graph, shown on the far right.}
\label{fig-labeling-symmetry}
\end{figure}%
%
%
\begin{proposition}\label{prop-labelled-gallai}
	Let $(\sigma,\ell_\sigma)$ and $(\tau,\ell_\tau)$ be two {$L$-labelled} simple permutations such that $(G_\sigma,\ell_\sigma)$ and $(G_\tau,\ell_\tau)$ are isomorphic. Then $(\tau,\ell_\tau)\in\{(\sigma,\ell_\sigma)$, $(\sigma^{-1},\ell^{-1}_\sigma)$, $(\sigma^{\text{rc}},\ell^{\text{rc}}_\sigma)$, $((\sigma^{\text{rc}})^{-1},(\ell^{\text{rc}}_\sigma)^{-1})\}$.
\end{proposition}

\begin{proof}
Since~$\sigma$ and $\tau$ are simple permutations with the property that $G_\sigma\cong G_\tau$, Lemma~\ref{lem-gallai-simple-prime} shows that $\tau$ is one of~$\sigma$, $\sigma^{-1}$,$\sigma^{\text{rc}}$, or $(\sigma^{\text{rc}})^{-1}$.

Suppose that~$\sigma$ (and hence also $\tau$) has length $n$. Let $\phi : \{1,2,\dots,n\} \to \{1,2,\dots,n\}$ induce an isomorphism from $(G_\sigma,\ell_\sigma)$ to $(G_\tau,\ell_\tau)$, so $i\sim j$ in~$G_\sigma$ if and only if $\phi(i)\sim\phi(j)$ in~$G_\tau$. Thus $\ell_\sigma(i)=\ell_\tau(\phi(i))$ for each vertex~$i$ of $G_\sigma$. By Proposition~\ref{prop:klavik},~$\phi$ must correspond to a symmetry of~$\sigma$. This correspondence shows both which symmetry relates $\tau$ and~$\sigma$ and also that~$\ell_\tau$ must equal the corresponding labeling (for example, if $\tau=\sigma^{\text{rc}}$ then we must have $\ell_\tau=\ell_\sigma^{\text{rc}}$). The result follows.
\end{proof}

We are now ready to establish the promised relationship between lwqo in permutations and lwqo in inversion graphs.

\begin{theorem}
\label{thm-C-lwqo-GC-lwqo}
The permutation class~$\C$ is lwqo if and only if the corresponding class $G_\C$ of inversion graphs is lwqo.
\end{theorem}

\begin{proof}
Part (a) of Proposition~\ref{prop-lwqo-perms-graphs} states that $G_\C$ is lwqo whenever~$\C$ is lwqo, giving us half of the result. Now suppose that~$\C$ is a permutation class for which $G_\C$ is lwqo.

By Corollary~\ref{cor-simples-lwqo}, it suffices to prove that the set $S$ of simple permutations in~$\C$ is lwqo, and we show this by adapting the proof of Proposition~\ref{prop-GC-wqo-C-simples-wqo}. Take $(L,\mathord{\le_L})$ to be an arbitrary wqo set and let
\(
	(\sigma_1,\ell_1),
	(\sigma_2,\ell_2),
	\dots
\)
be any sequence of labelled simple permutations. Consider the sequence
\(
	(G_{\sigma_1},\ell_1),
	(G_{\sigma_2},\ell_2),
	\dots
\)
of images of these labelled simple permutations under the mapping to labelled inversion graphs.

Since $G_\C\wr L$ is wqo, Proposition~\ref{prop-wqo-ramsey-inf-inc} implies that this sequence contains an infinite increasing chain. In particular, we can find a chain of length five, say
\[
	(G_{\sigma_{i_1}},\ell_{i_1})
	\le
	(G_{\sigma_{i_2}},\ell_{i_2})
	\le
	(G_{\sigma_{i_3}},\ell_{i_3})
	\le
	(G_{\sigma_{i_4}},\ell_{i_4})
	\le
	(G_{\sigma_{i_5}},\ell_{i_5})
\]
for $1\le i_1<i_2<i_3<i_4<i_5$. By Proposition~\ref{prop-labelled-gallai}, each of the inclusions in this chain must be witnessed by a symmetry of the shorter labelled simple permutation embedding in the longer. Thus
\[
	(\sigma_{i_1}^{s_1},\ell_{i_1}^{s_1})
	\leq
	(\sigma_{i_2}^{s_2},\ell_{i_2}^{s_2})
	\leq
	(\sigma_{i_3}^{s_3},\ell_{i_3}^{s_3})
	\leq
	(\sigma_{i_4}^{s_4},\ell_{i_4}^{s_4})
	\leq
	(\sigma_{i_5},\ell_{i_5}),
\]
where $s_1$, $s_2$, $s_3$, and $s_4$ are each one of the four graph-preserving symmetries. If none of these four symmetries is the identity, then two of them are the same (by the pigeonhole principle), and in either case we can find indices $j$ and $k$ with $1\leq j<k\leq 5$ such that $(\sigma_{i_j},\ell_{i_j})\leq(\sigma_{i_k},\ell_{i_k})$. This shows that every infinite sequence in $S\wr L$ contains a good pair, and thus $S\wr L$ is wqo. As~${(L,\le_L)}$ was an arbitrary wqo set, this shows that $S$ is lwqo. Corollary~\ref{cor-simples-lwqo} then implies that~$\C$ is lwqo, completing the proof.
\end{proof}

\section{Grid Classes}
\label{sec-geomgrid}

A grid class consists of those permutations whose plots can be subdivided into rectangles by a finite number of vertical and horizontal lines so that the subpermutations lying in the resulting rectangles satisfy conditions specified by a matrix. Three flavours of grid classes have been studied. Ordered by increasing specificity, these are generalised grid classes, monotone grid classes, and geometric grid classes. While the main result of this section applies to geometric grid classes, we briefly describe the other two types of grid classes to put this result in context.

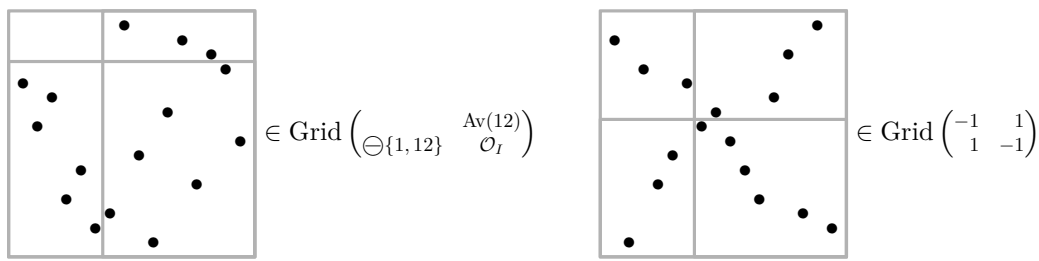
\begin{figure}
\[
	\begin{tikzpicture}[scale=0.1925, baseline=(current bounding box.center)]
		\plotpermbox{0.5}{0.5}{16.5}{16.5};
		\plotpermbox{0.5}{0.5}{16.5}{13};
		\plotpermbox{7}{0.5}{16.5}{16.5};
		\plotperm{12,9,11,4,6,2,3,16,7,1,10,15,5,14,13,8};
		\node at (17,8.5) [right] {$\in\Grid\fnmatrixc{&\Av(12)\\\bigominus\{1,12\}&\O_I}$};
	\end{tikzpicture}
	\quad\quad
	\begin{tikzpicture}[scale=0.1925, baseline=(current bounding box.center)]
		\plotpermbox{0.5}{0.5}{16.5}{16.5};
		\plotpermbox{0.5}{0.5}{16.5}{9};
		\plotpermbox{7}{0.5}{16.5}{16.5};
		\plotperm{15,1,13,5,7,12,9,10,8,6,4,11,14,3,16,2};
		\node at (17,8.5) [right] {$\in\Grid\fnmatrix{-1&1\\1&-1}$};
	\end{tikzpicture}
\]
\caption{Members of a generalised grid class (left) and of a monotone grid class (right).}
\label{fig-gen-grid}
\end{figure}

A \emph{generalised grid class} is defined by a matrix $\M$ of permutation classes. If the matrix $\M$ is of size $t\times u$, then the permutation~$\pi$ lies in the grid class of $\M$ if its plot can be divided by vertical and horizontal lines into a $t\times u$ arrangement of rectangles so that the subpermutation in each rectangle lies in the class specified by the corresponding entry of $\M$. An example is shown on the left of Figure~\ref{fig-gen-grid}. This example displays two common conventions: first, we do not write entries of $\M$ that consist of the empty class $\emptyset$, and second, we index $\M$ in Cartesian coordinates so that its indices correspond to coordinates in the plots of our permutations.

Generalised grid classes were first defined in, and have continued to be essential to, the determination of the set of all growth rates of permutation classes; see Vatter~\cite{vatter:small-permutati:,vatter:growth-rates-of:}.
The wqo (let alone lwqo) properties of generalised grid classes received little attention to-date. The most general result in this direction is due to Brignall~\cite[Theorem 3.1]{brignall:grid-classes-an:}. Note that the generalised grid class shown on the left of Figure~\ref{fig-gen-grid} is not wqo, as it contains the infinite antichain shown on the right of Figure~\ref{fig-three-inc-osc-antichains}.

In a \emph{monotone grid class}, the entries of the defining matrix $\M$ are restricted to the empty class $\emptyset$, the class of increasing permutations $\Av(21)$, and the class of decreasing permutations~${\Av(12)}$. We specify monotone grid classes with $\zpm$ matrices in which the classes $\emptyset$, $\Av(21)$, and $\Av(12)$ are denoted by the symbols $0$, $1$, and $-1$, respectively, although again we typically do not write zeros. For example, under this convention,
\[
	\Grid\fnmatrix{-1&1\\1&-1}
	=
	\Grid\fnmatrix{\Av(12)&\Av(21)\\\Av(21)&\Av(12)}.
\]
A member of this monotone grid class is shown on the right of Figure~\ref{fig-gen-grid}. The elements of this particular grid class are called \emph{skew-merged} permutations because they can be expressed as the union of an increasing subsequence and a decreasing subsequence. Their study dates to a 1994 result of Stankova~\cite[Theorem~2.9]{stankova:forbidden-subse:} that states (in our language) that
\[
	\Grid\fnmatrix{-1&1\\1&-1}
	=
	\Av(2143, 3412).
\]
Experience suggests that the skew-merged permutations exhibit similar behaviour to the $321$-avoiding permutations%
\footnote{For specific examples of these similarities, we refer to the work of Albert and Vatter~\cite{albert:generating-and-:}, who exploit them to enumerate the skew-merged permutations (reproving a result originally due to Atkinson~\cite{atkinson:permutations-wh:}), and to the work of Albert, Lackner, Lackner, and Vatter~\cite{albert:the-complexity-:}.};
in particular, the class of skew-merged permutations is not wqo, as it contains the infinite antichain shown on the right of Figure~\ref{fig-three-antichains}.

Monotone grid classes were first considered in full generality in a 2003 paper of Murphy and Vatter~\cite{murphy:profile-classes:}%
\footnote{Note that Murphy and Vatter~\cite{murphy:profile-classes:} referred to monotone grid classes as ``profile classes''. This is because monotone grid classes can be viewed as generalisations of the profile classes defined by Atkinson in his seminal 1999 paper~\cite[Section~2.2]{atkinson:restricted-perm:}. In terms of grid classes, Atkinson's profile classes are the monotone grid classes of permutation matrices (thus all of their cells are empty or increasing, and no two non-empty cells share a row or column).}.
Their main result gives a characterisation of the $\zpm$ matrices $M$ for which $\Grid(M)$ is wqo, in terms of the \emph{cell graph}%
\footnote{We state Theorem~\ref{thm-mono-grid-wqo} in terms of the cell graph of $M$ because this form has proved easier to work with, although Murphy and Vatter~\cite{murphy:profile-classes:} actually considered a different graph---the \emph{row-column graph} of the $\zpm$ matrix $M$ is the bipartite graph whose (bipartite) adjacency matrix is the absolute value of $M$. In other words, if $M$ is a $t\times u$ matrix, its row-column graph has vertices $x_1,\dots,x_t,y_1,\dots,y_u$ where there is an edge between $x_i$ and $y_j$ if and only if $M(i,j)\neq 0$. It is not difficult to show that the cell graph of a matrix is a forest if and only if its row-column graph is also a forest (a formal proof is given in Vatter and Waton~\cite[Proposition~1.2]{vatter:on-partial-well:}), and thus our formulation of Theorem~\ref{thm-mono-grid-wqo} is equivalent to what Murphy and Vatter proved.}
of $M$. This is the graph on the vertices $\{(i,j) : M(i,j)\neq 0\}$ in which $(i,j)$ and $(k,\ell)$ are adjacent if the corresponding cells of $M$ share a row or a column and there are no nonzero entries between them in this row or column.

\begin{theorem}[Murphy and Vatter~{\cite[Theorem~2.2]{murphy:profile-classes:}}]
\label{thm-mono-grid-wqo}
The monotone grid class $\Grid(M)$ is wqo if and only if the cell graph of $M$ is a forest.
\end{theorem}

The infinite antichains used to prove one direction of Theorem~\ref{thm-mono-grid-wqo} are variants of widdershins spirals, although the construction in Murphy and Vatter~\cite[Section 4]{murphy:profile-classes:} is fairly technical. A stream-lined construction follows from the work of Brignall~\cite{brignall:grid-classes-an:}, which generalises that direction of the result. The proof of the other half of Theorem~\ref{thm-mono-grid-wqo} has also been improved upon, in the work of Vatter and Waton~\cite{vatter:on-partial-well:}. That direction also follows, via Theorem~\ref{thm-forests-are-geoms}, from our upcoming Theorem~\ref{thm-ggc-wqo}, and is generalised by Theorem~\ref{thm-ggc-lwqo} after that.

The last result about monotone grid classes that we mention, below, shows how to determine whether a permutation class is contained in some monotone grid class%
\footnote{An extension of Theorem~\ref{thm-mono-griddable} to generalised grid classes is given by Vatter~\cite[Theorem~3.1]{vatter:small-permutati:}.};
classes that are contained in a monotone grid class are called \emph{monotone griddable}.

\begin{theorem}[Huczynska and Vatter~{\cite[Theorem 2.5]{huczynska:grid-classes-an:}}]
\label{thm-mono-griddable}
A permutation class is contained in some monotone grid class if and only if it does not contain $\bigoplus\{1,21\}$ or~$\bigominus\{1,12\}$.
\end{theorem}

From our lwqo-centric perspective, geometric grid classes are the most important of the three types of grid classes. To define these we must first adopt a geometric viewpoint of permutations and the containment order. In this view, the notion of relative order can be extended to point sets in the plane: two sets $S$ and $T$ of points in the plane are of the same \emph{relative order} (or, are \emph{order isomorphic}) if the $x$- and $y$-axes can be stretched and shrunk in some manner to transform one set into the other. A point set (such as the plot of a permutation) in which no two points lie on a common horizontal or vertical line is called \emph{independent}.

Every finite independent point set in the plane is in the same relative order as the plot of a unique permutation, and we call such a point set a \emph{drawing} of the permutation. If $S$ is an independent point set with $n$ points, then we can determine the permutation it is a drawing of by labeling its points $1$ to $n$ from bottom to top and then recording these labels reading left to right. It is evident that for any drawing $S$ of a permutation~$\pi$, there is some quantity $\epsilon>0$ (depending on $S$) such that by perturbing the points of $S$ each by at most $\epsilon$, the resulting point set is still a drawing of~$\pi$.

Given a $\zpm$ matrix $M$, we denote by $\Lambda_M$ the \emph{standard figure} of $M$, which we define to be the set of points in the plane consisting of
\begin{itemize}
\item the increasing line segment from $(k-1,\ell-1)$ to $(k,\ell)$ if $M(k,\ell)=1$ and
\item the decreasing line segment from $(k-1,\ell)$ to $(k,\ell-1)$ if $M(k,\ell)=-1$.
\end{itemize}
We then define the \emph{geometric grid class} of $M$, denoted by $\Geom(M)$, to be the set of all permutations that are in the same relative order as some finite independent subset of $\Lambda_M$. One can equivalently define $\Geom(M)$ to consist of all permutations that have a drawing \emph{on} $\Lambda_M$.

\begin{figure}
\begin{footnotesize}
\begin{center}
	\begin{tikzpicture}[scale=1.4, baseline=(current bounding box.center)]
		\draw[step=1, darkgray, thick, line cap=round] (0,0) grid (2,2);
		\plotpartialperm{{7/8}/{9/8}, {10/8}/{10/8}, {5/8}/{11/8}, {4/8}/{4/8}, {13/8}/{3/8}, {2/8}/{2/8}, {1/8}/{15/8}};
		\node at ({7/8},{9/8}) [above=2pt] {$4$};
		\node at ({10/8},{10/8}) [above=2pt] {$5$};
		\node at ({5/8},{11/8}) [above=2pt] {$6$};
		\node at ({4/8},{4/8}) [above=2pt] {$3$};
		\node at ({13/8},{3/8}) [above=2pt] {$2$};
		\node at ({2/8},{2/8}) [above=2pt] {$1$};
		\node at ({1/8},{15/8}) [right=2pt] {$7$};
		\draw [thick, line cap=round] (0,2)--(2,0);
		\draw [thick, line cap=round] (0,0)--(2,2);
		\node at (2,1) [right] {$\in\Geom\fnmatrix{-1&1\\1&-1}$};
	\end{tikzpicture}
	\quad\quad
	\begin{tikzpicture}[scale=1.4, baseline=(current bounding box.center)]
		\draw [lightgray, fill=lightgray] (1.6,0.8) rectangle (2,1);
		\draw [line width=0.8pt, dash pattern=on 0pt off 2.27pt, line cap=round] (1.6,0.8)--(1.6,1.5); 
		\draw [line width=0.8pt, dash pattern=on 0pt off 2.27pt, line cap=round] (1.6,0.8)--(0.9,0.8); 
		\draw [line width=0.8pt, dash pattern=on 0pt off 2.27pt, line cap=round] (1.6,0.8)--(1.6,0.05); 
		\draw [line width=0.8pt, dash pattern=on 0pt off 2.27pt, line cap=round] (1.6,0.8)--(1.95,0.8); 
		\draw[step=1, darkgray, thick, line cap=round] (0,0) grid (2,2);
		\plotpartialperm{0.6/1.4, 0.8/0.8, 1.6/1.6};
		\node at (0.6,1.4) [above=2pt] {$3$};
		\node at (0.8,0.8) [below=2pt] {$1$};
		\node at (1.6,1.6) [above=2pt] {$4$};
		\plotpartialperm{1.8/0.9};
		\node at (1.8,0.9) [below=3pt] {$2$};
		\draw [thick, line cap=round] (0,2)--(2,0);
		\draw [thick, line cap=round] (0,0)--(2,2);
		\node at (2,1) [right] {$\notin\Geom\fnmatrix{-1&1\\1&-1}$};
	\end{tikzpicture}
\end{center}
\end{footnotesize}
\caption{The permutation $7136452$ (left) can be drawn on an $\mathsf{X}$, while the permutation $3142$ (right) cannot be.}
\label{fig-X-example-non-example}
\end{figure}
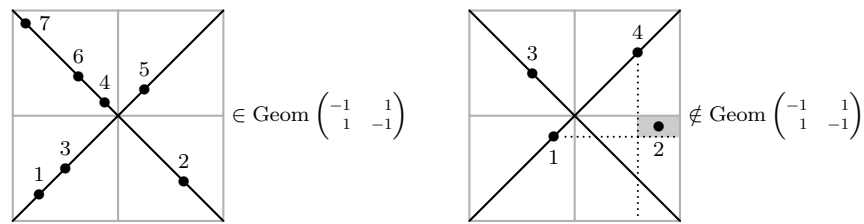

The best-studied geometric grid class is the class
\[
	\X
	=
	\Geom\fnmatrix{-1&1\\1&-1},
\]
consisting of those permutations that can be drawn on an $\Xfig$. An example of a member of this class is shown on the left of Figure~\ref{fig-X-example-non-example}. The class $\X$ has been studied by Waton~\cite[Section~5.6]{waton:on-permutation-:} and Elizalde~\cite{elizalde:the-x-class-and:}. Every permutation in $\X$ lies in the \emph{monotone} grid class of the same matrix (that is, it is skew-merged), but the converse does not hold:
\[
	\Grid\fnmatrix{-1&1\\1&-1}
	\nsubseteq
	\Geom\fnmatrix{-1&1\\1&-1}
\]
In particular, the permutation $3142$ cannot be drawn on an $\Xfig$ because, as is hinted at on the right of Figure~\ref{fig-X-example-non-example}, once the $3$, $1$, and $4$ are placed on the $\Xfig$, there is no point on the $\Xfig$ that lies simultaneously above the $1$ and to the right of the $4$.

In the the case of permutations drawn on an $\Xfig$, a simpler argument is that there must be some point that is farthest from the centre of the $\Xfig$. For this reason, every permutation in $\X$ must be of the form $1\oplus\sigma$, $\sigma\oplus 1$, $1\ominus\sigma$, or $\sigma\ominus 1$ for another permutation $\sigma\in\X$. From this observation, it follows readily that a permutation can be drawn on an $\Xfig$ if and only if it is skew-merged \emph{and separable}.

Despite this example, geometric grid classes and monotone grid classes coincide in many cases of interest. By combining Murphy and Vatter's Theorem~\ref{thm-mono-grid-wqo} with results of Albert, Atkinson, Bouvel, Ru\v{s}kuc, and Vatter~{\cite[Theorems~3.2 and 6.1]{albert:geometric-grid-:}, one can establish the following characterisation of these classes.

\begin{theorem}
\label{thm-forests-are-geoms}
For a $\zpm$ matrix $M$, we have
\[
	\Grid(M) = \Geom(M)
\]
if and only if the cell graph of $M$ is a forest.
\end{theorem}

The fact that $\Grid(M)=\Geom(M)$ whenever the cell graph of $M$ is a forest is explicitly proved in \cite[Theorem~3.2]{albert:geometric-grid-:}, and the proof given there essentially consists of ``straightening out'' the plot of an arbitrary permutation in $\Grid(M)$, although one could alternatively adapt the proof of Vatter and Waton~\cite[Proposition~3.3]{vatter:on-partial-well:} to give an order-theoretical proof. For the other direction of Theorem~\ref{thm-forests-are-geoms}, it is easiest to notice the discrepancies in their wqo properties, since Theorem~\ref{thm-mono-grid-wqo} shows that $\Grid(M)$ is not wqo when the cell graph of $M$ contains a cycle, while geometric grid classes are always wqo:

\begin{theorem}[Albert, Atkinson, Bouvel, Ru\v{s}kuc, and Vatter~{\cite[Theorem~6.1]{albert:geometric-grid-:}}]
\label{thm-ggc-wqo}
For every $\zpm$ matrix $M$, the geometric grid class $\Geom(M)$ is wqo.
\end{theorem}

The conclusion of Theorem~\ref{thm-ggc-wqo} extends to subclasses of geometric grid classes, of course. Analogous to the case with monotone grid classes, we say that a permutation class~$\C$ is \emph{geometrically griddable} if~$\C\subseteq\Geom(M)$ for some $\zpm$ matrix $M$. However, unlike the case with monotone and generalised griddability, there is no known analogue of Theorem~\ref{thm-mono-griddable} characterising the obstructions to geometric griddability.

Our aim in this section is to strengthen the conclusion of Theorem~\ref{thm-ggc-wqo} from wqo to lwqo, but to do this we need a few more definitions. Given a drawing $S$ of the permutation~$\pi$ on the standard figure~$\Lambda_M$, by our comments about perturbations above, we may assume that none of the points of~$S$ lies on the \emph{integer lattice}~$\mathbb{Z}^2$ (or equivalently, that none of the points is an endpoint of one of the line segments of~$\Lambda_M$). Thus every point of~$S$ belongs to precisely one line segment of~$\Lambda_M$, and so we can associate to it a single cell of the matrix~$M$. In this way, we obtain a \emph{gridded permutation}~$\pi^\gridded$, and we denote the set of all gridded members of~$\Geom(M)$ by~$\Geom^\gridded(M)$.

Given the viewpoint of this paper, it is natural for us to view this association of entries of a gridded permutation $\pi^\gridded\in\Geom^\gridded(M)$ to cells of $M$ as a labeling of the entries of~$\pi$. To this end, let $\Sigma_M$ denote a finite antichain consisting of the cells of $M$. (In other treatments of geometric grid classes, $\Sigma_M$ is referred to as a \emph{cell alphabet}, for reasons touched on below.) The gridded permutations of~$\Geom^\gridded(M)$ therefore correspond to certain members of the poset $\Geom(M)\wr\Sigma_M$ of $\Sigma_M$-labelled members of $\Geom(M)$.

The proof of Theorem~\ref{thm-ggc-wqo} given in \cite{albert:geometric-grid-:} relies upon a length- and order-preserving surjection
\(
	\varphi^\gridded : \Sigma_M^\ast\tosurj\Geom^\gridded(M).
\)
Having established that such a mapping exists, it follows immediately from Higman's lemma and Proposition~\ref{prop-wqo-order-preserving} that $\Geom^\gridded(M)$ is wqo. Letting~${\delta : \Geom^\gridded(M)\tosurj\Geom(M)}$ denote the order-preserving surjection that removes griddings, it follows that $\Geom(M)$ is wqo, proving Theorem~\ref{thm-ggc-wqo}.

This approach could be adapted to prove the strengthening of Theorem~\ref{thm-ggc-wqo} we desire, but we employ minimal bad sequences instead. We do this for several reasons. First, our approach does not require us to define $\varphi^\gridded$, a definition that is fairly involved%
\footnote{In particular, the definition of $\varphi^\gridded$ requires the choice of a consistent orientation of the cells of $\Lambda_M$, a process that may require a further subdivision of this figure (this is the role partial multiplication matrices play in \cite{albert:geometric-grid-:}).}.
Second, as Albert, Ru\v{s}kuc, and Vatter~\cite{albert:inflations-of-g:} put it, the mapping $\varphi^\gridded$ ``jumbles'' entries%
\footnote{It is for this reason that an index correspondence $\psi$ must be introduced in \cite[Section~3]{albert:inflations-of-g:}.},
and thus attaching labels (as we must do to establish lwqo) would be cumbersome. Third, we hope that this alternative approach may prove useful for other purposes.

\begin{theorem}
\label{thm-ggc-lwqo}
For every $\zpm$ matrix $M$, the geometric grid class $\Geom(M)$ is lwqo.
\end{theorem}

\begin{proof}
Let $M$ be a $\zpm$ matrix and take $(L,\le_L)$ to be an arbitrary wqo set. 
We begin by extending the definition of the mapping $\delta$ mentioned above to our context, defining
\[
	\delta : \Geom^\gridded(M)\wr L \tosurj \Geom(M)\wr L
\]
by
\[
	\delta((\pi^\gridded,\ell_\pi)) = (\pi,\ell_\pi).
\]
This mapping, which in effect simply ``forgets'' the gridding of a permutation, is an order-preserving surjection. Consequently, it suffices to establish that $\Geom^\gridded(M)\wr L$ is wqo.
Suppose to the contrary that this set is not wqo and consider a minimal bad sequence $S\subseteq\Geom^\gridded(M)\wr L$, which must exist by Proposition~\ref{prop-wqo-iff-min-bad-seq}.

Suppose that $(\pi^\gridded,\ell_\pi)\in S$ and that~$\pi^\gridded$ has length~$n$, so ${\ell_\pi : \{1,2,\dots,n\}\to L}$. As discussed earlier,~$\pi^\gridded$ corresponds to a particular member of $\Geom(M)\wr \Sigma_M$. Let ${c_\pi : \{1,2,\dots,n\} \to \Sigma_M}$ denote the function that records the cells of the entries of~$\pi$ in~$\pi^\gridded$, so $\pi(i)$ lies in cell $c_\pi(i)$ in~$\pi^\gridded$. Choose a drawing of~$\pi$ on the standard figure $\Lambda_M$ that witnesses this gridding, so for every index~$i$, the point corresponding to~$\pi(i)$ lies in the cell $c_\pi(i)$. In this drawing of~$\pi$, there must be one point that lies at least as close to a lattice point as any other point; in fact, by perturbing this point by a minuscule amount, we may assume that it lies closer to a lattice point than any other point of the drawing. We fix such a drawing for every $(\pi^\gridded,\ell_\pi)\in S$, and for the remainder of the proof these are the only drawings we consider.

For each $(\pi^\gridded,\ell_\pi)\in S$, let $j_\pi$ denote the index corresponding to the point in the fixed drawing of~$\pi$ that lies closest to a lattice point. We remove the entry $\pi(j_\pi)$ from~$\pi$ to obtain the permutation $\overline{\pi}=\pi-\pi(j_\pi)$. By simultaneously removing the point corresponding to this entry in the drawing of~$\pi$, we obtain a drawing of~$\overline{\pi}$. This drawing naturally induces a gridding~$\overline{\pi}^\gridded$, and we define $\overline{c}_\pi :{\{1,2,\dots,n-1\}\longrightarrow\Sigma_M}$ to be the corresponding cell labeling. Finally, we remove the label of~$\pi(j_\pi)$ from~$\ell_\pi$ to obtain the mapping 
\[
	\overline{\ell}_\pi : \{1,2,\dots,n-1\}\longrightarrow L.
\]
Thus for every index~$i$, the point corresponding to~$\overline{\pi}(i)$ lies in the cell $\overline{c}_\pi(i)$ and has label $\overline{\ell}_\pi(i)$.

With $\overline{\pi}$, $\overline{c}_\pi$, and $\overline{\ell}_\pi$ defined as above, we see that for every $L$-labelled gridded permutation ${(\pi^\gridded,\ell_\pi)\in S}$, we have an $L$-labelled gridded permutation $(\overline{\pi}^\gridded,\overline{\ell}_\pi)\in S^{\suplessthan}$. When passing from $(\pi^\gridded,\ell_\pi)$ to $(\overline{\pi}^\gridded,\overline{\ell}_\pi)$, certain information is lost---the values of $c_\pi(j_\pi)$ and $\ell_\pi(j_\pi)$, of course, but also, which lattice point in the drawing of~$\pi$ was closest to the point corresponding to~$\pi(j_\pi)$. Since we know that this point lies on the line segment of $\Sigma_M$ in the cell $c_\pi(j_\pi)$, there are precisely two possibilities: either it lies closest to the lattice point at the lefthand end of the line segment, or it lies closest to the lattice point at the righthand end. We encode these possibilities by an element $s_\pi$ of the two-element antichain~$\{\mathsf{L},\mathsf{R}\}$.

The above discussion allows us to (finally) define a mapping
\[
	\Psi : S\to S^{\suplessthan}\times \Sigma_M\times L\times \{\mathsf{L},\mathsf{R}\}
\]
by
\[
	\Psi((\pi^\gridded, \ell_\pi))
	=
	((\overline{\pi}^\gridded,\overline{\ell}_\pi), c_\pi(j_\pi), \ell_\pi(j_\pi), s_\pi).
\]
Proposition~\ref{prop-min-bad-seq-prop-closure-wqo} shows that $S^{\suplessthan}$ is wqo, and $\Sigma_M$, $L$, and $\{\mathsf{L},\mathsf{R}\}$ are wqo either by assumption or finiteness, so the image of~$S$ under~$\Psi$ is wqo by Proposition~\ref{prop-wqo-product}. To complete the proof it therefore suffices to show that~$\Psi$ is order-reflecting, as this will imply that our assumed minimal bad sequence~$S$ is wqo, a contradiction.

Consider elements $(\sigma^\gridded,\ell_\sigma)$, $(\pi^\gridded, \ell_\pi)\in S$ satisfying
\[
	\Psi((\sigma^\gridded,\ell_\sigma))\le \Psi((\pi^\gridded,\ell_\pi)).
\]
This means that $\overline{\sigma}^\gridded \le \overline{\pi}^\gridded$ as gridded permutations (or, equivalently, that ${(\overline{\sigma},\overline{c}_\sigma)\le(\overline{\pi},\overline{c}_\pi)}$ as $\Sigma_M$-labelled permutations). Thus it follows that there is a drawing $P_\pi\subseteq\Lambda_M$ of~$\pi$ satisfying the following two conditions.
\begin{itemize}
\item There is a subset ${\overline{P}_\sigma\subseteq P_\pi}$ that is in the same relative order as the plot of~$\overline{\sigma}$, and in fact, for each index~$i$, the point of~$\overline{P}_\sigma$ corresponding to~$\overline{\sigma}(i)$ lies on the line segment in the cell~$\overline{c}_\sigma(i)$.
\item The point of $P_\pi$ lying closest to a lattice point, say $p$, corresponds to $\pi(j_\pi)$ in~$\pi$, and the lattice point it lies closest to is either the left or right endpoint of the line segment in the cell ${\overline{c}_\sigma(j_\sigma)=\overline{c}_\pi(j_\pi)}$, depending on the value of~$s_\sigma=s_\pi$.
\end{itemize}
This ensures that $\sigma^\gridded \le \pi^\gridded$ as gridded permutations (or equivalently, that $(\sigma,c_\sigma)\le (\pi,c_\pi)$ as $\Sigma_M$-labelled permutations): by the above, the point set $P_\sigma=\overline{P}_\sigma\cup\{p\}$ is in the same relative order as~$\sigma$, and moreover, for every index~$i$, the point of $P_\sigma$ corresponding to~$\sigma(i)$ lies in the cell~$c_\sigma(i)$. Furthermore, because $(\overline{\sigma}^\gridded,\overline{\ell}_\sigma)\le (\overline{\pi}^\gridded,\overline{\ell}_\pi)$ and $\ell_\pi(j_\pi)\le\ell_\sigma(j_\sigma)$, we know that this embedding of~$\sigma^\gridded$ into~$\pi^\gridded$ also respects the order of the labels from $L$. Thus this embedding witnesses that ${(\sigma^\gridded,\ell_\sigma)\le (\pi^\gridded, \ell_\pi)}$, verifying that~$\Psi$ is order-reflecting, contradicting our choice of the minimal bad sequence~$S$, and completing the proof of the theorem as already described.
\end{proof}

We conclude by discussing finite bases and graphical analogues. An immediate consequence of Theorem~\ref{thm-ggc-lwqo} (via Proposition~\ref{prop-lwqo-fin-basis}) is that every geometric grid class (in fact, every geometrically griddable class) is finitely based. This is Theorem~6.2 of \cite{albert:geometric-grid-:}, and is established there by showing that for every $\zpm$ matrix $M$, $\Geom(M)^{+1}\subseteq\Geom(M')$ for some larger~$\zpm$~matrix~$M'$, and then appealing to Theorem~\ref{thm-ggc-wqo}. Neither the original proof nor our proof is constructive, and no bounds on the lengths of the basis elements of geometric grid classes have been established, with the notable exception of one special case%
\footnote{This special case is that of monotone/geometric grid classes of row vectors, which were studied in a 2002 paper of Atkinson, Murphy, and Ru\v{s}kuc~\cite{atkinson:partially-well-:} where they are called ``$W$-classes'', owing to the fact that members of one such class, $\Grid(-1$\ \ $1$ $-1$\ \ $1)$, can be drawn on the figure $\gridhoriz{-1,1,-1,1}$. Because these classes can be viewed as juxtapositions, a theorem of Atkinson~\cite[Theorem 2.2]{atkinson:restricted-perm:} implies that they have finite bases and gives a procedure to determine these bases. The enumeration of these classes is also much easier than general geometric grid classes. Albert, Atkinson, Bouvel, Ru\v{s}kuc, and Vatter~\cite[Theorem~8.1]{albert:geometric-grid-:} show that all geometrically griddable classes have rational generating functions, but the proof is nonconstructive. On the other hand, Albert, Atkinson, and Ru\v{s}kuc~\cite[Section 3]{albert:regular-closed-:} show how to compute the (rational) generating functions of arbitrary subclasses of monotone grid classes of $\zpm$ row vectors.}.

The above is all for \emph{geometric} grid classes. We do not even have a nonconstructive proof that \emph{monotone} grid classes are finitely based, although it has been conjectured that this is the case.

\begin{conjecture}[Huczynska and Vatter~{\cite[Conjecture 2.3]{huczynska:grid-classes-an:}}]
\label{conj-mono-grid-basis}
Every monotone grid class is finitely based.
\end{conjecture}

Conjecture~\ref{conj-mono-grid-basis} holds for monotone grid classes that are also geometric grid classes, namely (by Theorem~\ref{thm-forests-are-geoms}), monotone grid classes of the form $\Grid(M)$ where the cell graph of $M$ is a forest. It also holds for the skew-merged permutations by the result of Stankova~\cite[Theorem~2.9]{stankova:forbidden-subse:} mentioned earlier. Beyond this, Waton showed in his thesis~\cite[Theorem~4.7.5]{waton:on-permutation-:} that the monotone grid class of the $2\times 2$ all-one matrix is finitely based while Albert and Brignall~\cite{albert:2times-2-monoto:} have shown that every $2\times 2$ monotone grid class is finitely based (in fact their result covers certain $2\times 2$ generalised grid classes as well). 

We conclude this section by briefly discussing related graph classes. The graphical analogues of the skew-merged permutations are the \emph{split graphs}, first studied in a 1977 paper of F\"oldes and Hammer~\cite{foldes:split-graphs:}. These are defined as the graphs whose vertices can be partitioned into a clique and an independent set. F\"oldes and Hammer showed that the split graphs are characterised by the forbidden induced subgraphs $G_{2143}=2K_2$, $G_{3412}=C_4$, and $C_5$. Note that not all split graphs are inversion graphs%
\footnote{%
Foldes and Hammer~\cite[Theorem 3]{foldes:split-graphs-ha:} characterise the split comparability graphs by forbidden induced subgraphs. Because the class of inversion graphs is the intersection of the class of comparability graphs with the class of their complements (the co-comparability graphs), this result implies that the split inversion graphs are defined by the forbidden induced subgraphs $2K_2$, $C_4$, $C_5$, net, co-net, rising sun, and co-rising sun, the last four of which are shown below.
\begin{center}
	\begin{tikzpicture}[scale=0.6, yscale=-1]
		\plotpartialperm{0/0, -0.5/0.866, 0.5/0.866, 0/-0.464, -0.732/1.267824, 0.732/1.267824};
		\draw (0,0)--(-0.5,0.866)--(0.5,0.866)--cycle;
		\draw (-0.5,0.866)--(-0.732,1.267824);
		\draw (0.5,0.866)--(0.732,1.267824);
		\draw (0,-0.464)--(0,0);
	\end{tikzpicture}
	\quad\quad
	\begin{tikzpicture}[scale=0.6]
		\plotpartialperm{-1/0, 0/0, 1/0, -0.5/0.866, 0.5/0.866, 0/1.732};
		\draw (-1,0)--(0,1.732)--(1,0)--cycle;
		\draw (0,0)--(-0.5,0.866)--(0.5,0.866)--cycle;
	\end{tikzpicture}
	\quad\quad
	\begin{tikzpicture}[scale=0.6]
		\plotpartialperm{0/0, 1/0, 1/1, 0/1, 0.5/1.732, -0.732/0.5, 1.732/0.5};
		\draw (0,0)--(1,0)--(1,1)--(0,1)--(0,0)--(1,1)--(0,1)--(1,0);
		\draw (0,0)--(-0.732,0.5)--(0,1)--(0.5,1.732)--(1,1)--(1.732,0.5)--(1,0);
	\end{tikzpicture}
	\quad\quad
	\begin{tikzpicture}[scale=0.6]
		\plotpartialperm{-1/0, 0/0, 1/0, -0.5/0.866, 0.5/0.866, 1/1.732, -1/1.732};
		\draw (0,0)--(-0.5,0.866)--(0.5,0.866)--cycle;
		\draw (-1,0)--(1,0)--(0.5,0.866)--(-0.5,0.866)--(-1,0);
		\draw (0.5,0.866)--(1,1.732);
		\draw (-0.5,0.866)--(-1,1.732);
	\end{tikzpicture}
\end{center}
The class of split inversion graphs has been further studied by Korpelainen, Lozin, and Mayhill~\cite{korpelainen:split-permutati:}, who prove that it is not wqo by establishing that the set of inversion graphs of the infinite antichain of permutations shown on the right of Figure~\ref{fig-three-antichains} forms an infinite antichain of graphs.}.

Outside of the split graphs, the graphical analogues of monotone grid classes have received very little attention, although Atminas~\cite[Theorem~1.2]{atminas:classes-of-grap:} establishes a graphical analogue of (a generalisation of) Theorem~\ref{thm-mono-griddable}.

The graphical analogue of the class $\X$ of permutations that can be drawn on an $\Xfig$ is the class of \emph{threshold graphs} first defined by Golumbic in 1978~\cite{golumbic:threshold-graph:}; these are the graphs that can be built, starting from $K_1$, by repeatedly taking the disjoint union or join with $K_1$. In fact, Golumbic himself considered the permutation class $\X$ in \cite[Section 3]{golumbic:threshold-graph:} and also in his book \emph{Algorithmic Graph Theory and Perfect Graphs}~\cite[Section 10.3]{golumbic:algorithmic-gra:}. From our characterisation above it follows that if $G_\pi$ is a threshold graph, then $\pi\in\X$. However, owing to the many-to-one nature of the mapping $\pi\mapsto G_\pi$, we can obtain the class of threshold graphs by considering inversion graphs of a much smaller permutation class: every threshold graph is of the form $G_\pi$ for some permutation $\pi\in\Geom(\text{\footnotesize $-1$}\ \ \text{\footnotesize $1$})$. Put geometrically, this means that every threshold graph is the inversion graph of a permutation that can be drawn on a $\Vfig$, or, for that matter, on a~$\Vfigt$, a~$\Vfigl$, or a~$\Vfigr$.

The graph-theoretic analogues of geometrically griddable classes are the \emph{graph classes of bounded lettericity}%
\footnote{In particular, all threshold graphs have lettericity $2$, as observed in \cite[Theorem~3]{petkovsek:letter-graphs-a:}.}.
These graph classes were introduced by Petkov{\v{s}}ek~\cite{petkovsek:letter-graphs-a:} in 2002, and the connection to geometric grid classes was first noted in the literature by Alecu, Lozin, de Werra, and Zamaraev~\cite{alecu:letter-graphs-a:}%
\footnote{An extended abstract version also appears as Alecu, Lozin, Zamaraev, and de Werra~\cite{alecu:letter-graphs-a:abstract}.},
who proved that if a permutation class is geometrically griddable, then the corresponding class of inversion graphs has bounded lettericity. They also conjectured that the converse statement holds, and this has since been proved.

\begin{theorem}[Alecu, Ferguson, Kant\'e, Lozin, Vatter, and Zamaraev~\cite{alecu:letter-graphs-a:iff}]
\label{thm-ggc-lettericity}
The permutation class~$\C$ is geometrically griddable if and only if the corresponding class $G_{\C}$ of inversion graphs has bounded lettericity.
\end{theorem}

The proof that graph classes of bounded lettericity are wqo follows immediately from Higman's lemma by an argument given by Petkov{\v{s}}ek~\cite[Theorem~8]{petkovsek:letter-graphs-a:}. Atminas and Lozin~\cite[Theorem~4]{atminas:labelled-induce:} later showed that these classes are lwqo. Neither result follows from our work because graph classes of bounded lettericity may contain graphs that are not inversion graphs.

\section{Concluding Remarks}

When studying a permutation class, it is frequently critical to determine whether the class is wqo because if so, then various finiteness conditions can be brought to bear. As we have demonstrated, lwqo permutation classes have even nicer properties, such as finite bases (Proposition~\ref{prop-lwqo-fin-basis}) and the fact that their one-point extensions and substitution closures are also lwqo (Theorems~\ref{thm-lwqo-C-lwqo-C+1} and~\ref{thm-subst-closure-lwqo}, respectively). For the permutation patterns practitioner, our work provides a toolkit to establish lwqo. From this perspective, two corollaries of our work stand out for their wide applicability.

First, by combining Theorem~\ref{thm-ggc-lwqo} and Corollary~\ref{cor-simples-lwqo}, we obtain the following.

\begin{corollary}
\label{cor-conclusion-1}
If the simple permutations in a permutation class are geometrically griddable, then it is lwqo.
\end{corollary}

One might appeal to Theorem~\ref{thm-lwqo-C-lwqo-C+1} in order to strengthen Corollary~\ref{cor-conclusion-1} by saying that if the simple permutations in~$\C$ are geometrically griddable, then the class~$\C^{+t}$ is lwqo for every~${t\ge 0}$. However, this situation falls under the purview of Corollary~\ref{cor-conclusion-1} already, since~$\C^{+t}$ is geometrically griddable whenever~$\C$ is (this follows from Albert, Atkinson, Bouvel, Ru\v{s}kuc, and Vatter~\cite[Theorem~6.4]{albert:geometric-grid-:}).

As we have demonstrated, Corollary~\ref{cor-conclusion-1} generalises most of the wqo results in the permutation patterns literature, and in fact strengthens their conclusions (from wqo to lwqo). Many of the results in the literature not subsumed by Corollary~\ref{cor-conclusion-1} are subsumed by the following combination of Corollary~\ref{cor-conclusion-1} with Theorem~\ref{thm-CU-inflate-wqo}.

\begin{corollary}
\label{cor-conclusion-2}
If the simple permutations in the permutation class~$\C$ are geometrically griddable, then the class~$\C[\U]$ is wqo for every wqo permutation class $\U$.
\end{corollary}

It should be noted that there is no known systematic method for determining whether the hypotheses of these two results apply to a given permutation class. Albert, Atminas, and Brignall~\cite{albert:characterising-:} have shown how to determine whether the simple permutations of a given class are \emph{monotone} griddable, but Corollaries~\ref{cor-conclusion-1} and~\ref{cor-conclusion-2} apply only to classes whose simple permutations are \emph{geometrically} griddable.

\begin{figure}
\[
	\begin{tikzpicture}[scale=0.1925, baseline=(current bounding box.center)]
		\draw (11,2)--(1.5,2);
		\draw (12,5)--(3.5,5);
		\draw (13,8)--(5.5,8);
		\draw (14,11)--(7.5,11);
		\draw (15,14)--(9.5,14);
		\plotpermbox{0.5}{0.5}{15.5}{15.5};
		\plotpermbox{11}{0.5}{15.5}{15.5};
		\plotperm{3,1,6,4,9,7,12,10,15,13,2,5,8,11,14};
		\node at (17,8.5) [right] {$\in\Grid\fnmatrix{\bigoplus\Av(12)&\Av(21)}$};
	\end{tikzpicture}
\]
\caption{A gridded simple permutation.}
\label{fig-juxta-lwqo}
\end{figure}
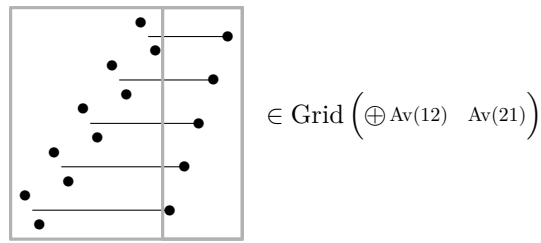

One might also wonder about the converses to these results; in particular, if~$\C$ is an lwqo permutation class, must its simple permutations be geometrically griddable? This is false, and one example is illustrated by the simple permutation shown in Figure~\ref{fig-juxta-lwqo}. It can either argued (either ad hoc, or by lifting the results of Brignall~\cite{brignall:grid-classes-an:} to the lwqo setting) that the set of all simple permutations of this form is lwqo. Letting~$\C$ denote the downward closure of these simple permutations, it then follows by Theorem~\ref{thm-lwqo-downward-closure} that~$\C$ is lwqo; indeed, it follows by Theorem~\ref{thm-subst-closure-lwqo} that~$\langle\C\rangle$ is lwqo. Moreover, Theorem~\ref{thm-CU-inflate-wqo} shows that~$\C[\U]$ is wqo for every wqo permutation class~$\U$. Nevertheless, Theorem~\ref{thm-mono-griddable} implies that these simple permutations are not monotone griddable, let alone geometrically griddable, and this shows that the converses to both Corollary~\ref{cor-conclusion-1} and Corollary~\ref{cor-conclusion-2} are false.

We close by collecting the questions and conjectures that were posed throughout. The question below remains unanswered, although we provided a partial answer with Proposition~\ref{prop-GC-wqo-C-simples-wqo} and established its lwqo analogue with Theorem~\ref{thm-C-lwqo-GC-lwqo}.

\begin{itemize}
\item
	\textbf{Question~\ref{question-prop-wqo-perms-graphs-converse}.}
	\emph{Let~$\C$ be a permutation class and $G_\C$ the corresponding graph class. If $G_\C$ is wqo in the induced subgraph order, must~$\C$ be wqo in the permutation containment order?}
\end{itemize}

We have specialised the following conjecture of Pouzet to permutation classes, but it is also open for graph classes (as well as for Pouzet's original setting of general relational structures).

\begin{itemize}
\item
	\textbf{Conjecture~\ref{conj-pouzet-2-wqo}.}
	(Cf. Pouzet~\cite{pouzet:un-bel-ordre-da:})
	\emph{A permutation class is $2$-wqo if and only if it is $n$-wqo for all $n\ge 1$.}
\end{itemize}

We have presented several questions and conjectures related to Pouzet's Conjecture~\ref{conj-pouzet-2-wqo}, beginning with the following potential strengthening of it, which was asked in the graph context by Brignall, Engen, and Vatter~\cite{brignall:a-counterexampl:}.

\begin{itemize}
\item
	\textbf{Question~\ref{ques-pouzet-2-wqo-lwqo}.}
	(Cf. Brignall, Engen, and Vatter~\cite{brignall:a-counterexampl:})
	\emph{Is every $2$-wqo permutation class also lwqo?}
\end{itemize}

Two other conjectures we presented may be viewed as variants of the above:

\begin{itemize}
\item
	\textbf{Conjecture~\ref{conj-C+1-wqo-C-lwqo}.}
	\emph{If the permutation class~$\C^{+1}$ is wqo, then~$\C$, and thus also~$\C^{+1}$, is lwqo.}
\item
	\textbf{Conjecture~\ref{conj-C-subst-wqo-C-lwqo}.}
	\emph{If the permutation class $\langle\C\rangle$ is wqo, then~$\C$, and thus also $\langle\C\rangle$, is lwqo.}
\end{itemize}

The following two conjectures would follow from Pouzet's Conjecture~\ref{conj-pouzet-2-wqo} or a positive answer to Question~\ref{ques-pouzet-2-wqo-lwqo}. 

\begin{itemize}
\item
	\textbf{Conjecture~\ref{conj-2-wqo-C-lwqo-C+1}.}
	\emph{If the permutation class~$\C$ is $2$-wqo, then the class~$\C^{+1}$ is wqo.}
\item 
	\textbf{Conjecture~\ref{conj-2-wqo-C-lwqo-C+t}.}
	\emph{If the permutation class~$\C$ is $2$-wqo, then the class~$\C^{+t}$ is $2$-wqo for every ${t\ge 0}$.}
\end{itemize}

Finally, we repeat the following conjecture concerning grid classes, although it does not directly address lwqo.

\begin{itemize}
\item
	\textbf{Conjecture~\ref{conj-mono-grid-basis}.}
	(Huczynska and Vatter~{\cite[Conjecture 2.3]{huczynska:grid-classes-an:}})
	\emph{Every monotone grid class is finitely based.}
\end{itemize}

\section*{Acknowledgements}

We thank the two anonymous reviewers for their comments and suggestions that improved this work.

\bibliographystyle{acm}
\begin{small}
\setlength{\bibsep}{2pt}


\end{small}
\end{document}